\documentclass[10pt]{article}
\usepackage{amsmath,amssymb,amsthm,amsfonts,amscd,euscript,verbatim, t1enc, newlfont, xypic}
\usepackage{hyperref}
\hypersetup{breaklinks}

 \theoremstyle{definition}
\newtheorem{theo}{Theorem}[subsection]

\newtheorem{pr}[theo]{Proposition}

 \newtheorem{lem}[theo]{Lemma}
 \newtheorem{coro}[theo]{Corollary}
  
\theoremstyle{remark}
\newtheorem{rema}[theo]{Remark}

\theoremstyle{definition}
\newtheorem{defi}[theo]{Definition}

 \newcommand\lan{\langle}
\newcommand\ra{\rangle}

\newcommand\gd{\mathfrak{D}}
\newcommand\gdp{\mathfrak{D}'}

\newcommand\opp{{^{op}}}
\newcommand\ob{^{-1}}
\newcommand\smc{{SmCor}}

\newcommand\tpi{\tilde{\pi}}
\newcommand\shi{\operatorname{SHI}}
 
\newcommand\zop{{\mathbb{Z}[1/p]}}

\newcommand\dmge{DM^{eff}_{gm}{}}
\newcommand\dmgm{DM_{gm}}
\newcommand\dme{DM_-^{eff}{}}
\newcommand\dm{\operatorname{DM}}
\newcommand\dmk{\operatorname{DM}(k)}

\newcommand\cha{\operatorname{char}}
\newcommand\proo{\operatorname{Pro}}

\newcommand\mg{M_{gm}}

\newcommand\obj{\operatorname{Obj}}
\newcommand\gal{\operatorname{Gal}}

\newcommand\id{\operatorname{id}}
\newcommand\cu{\underline{C}}
\newcommand\du{\underline{D}}
\newcommand\au{\underline{A}}
\newcommand\eu{\underline{E}}
\newcommand\bu{\underline{A}}

\newcommand\cupr{\underline{C}'}
\newcommand\cuperp{\underline{C}_{\perp}}

\newcommand\rco{R_{\mathbb{C}}}
\newcommand\reco{Re_{\mathbb{C}}}

\newcommand\re{\mathbb{R}}
\newcommand\com{\mathbb{C}}
\newcommand\n{\mathbb{N}}
\newcommand\z{{\mathbb{Z}}}
\newcommand\q{{\mathbb{Q}}}
\newcommand\af{\mathbb{A}}
 
\newcommand\p{\mathbb{P}}

\newcommand\zl{{\mathbb{Z}_l}}
\newcommand\zlz{\z/l\z}

\newcommand\pt{\operatorname{pt}}

\newcommand\lam{\Lambda}
\newcommand\al{\alpha}

\newcommand\de{\delta}

\newcommand\sss{{\mathcal{S}}}
\newcommand\gm{\mathcal{M}}
\newcommand\gmp{\gm'}

\newcommand\ns{\{0\}}

\DeclareMathOperator\prli{\varprojlim}
\DeclareMathOperator\inli{\varinjlim}

\newcommand\ihom{{\underline{Hom}}}

\DeclareMathOperator\codim{\operatorname{codim}}
\newcommand\chow{\operatorname{Chow}}

\newcommand\ab{\underline{Ab}}
\newcommand\var{Var}
\newcommand\sv{\operatorname{SmVar}}
\newcommand\spv{\operatorname{SmPrVar}}

\newcommand\spe{\operatorname{Spec}}

\newcommand\modd{\operatorname{Mod}}
\newcommand\perpp{{}^{\perp}}

\newcommand\brj{\lan j\ra}
\newcommand\brjj{\{ j\}}
\newcommand\gmm{\mathbb{G}_m}
\newcommand\gmmpl{\mathbb{G}_{m+}}

\newcommand\zoh{\z[1/2]}
\newcommand\sh{SH^{S^1}(k)}
\newcommand\shc{SH^{S^1}(k)^c}

\newcommand\psh{PSH^{S^1}(k)}
\newcommand\sht{SH(k)}
\newcommand\shtc{SH(k)^c}
\newcommand\she{SH^{eff}(k)}

\newcommand\shtop{SH}

\newcommand\psht{PSH^{T}(k)}

\newcommand\shtoh{SH(k)[2\ob]}

\newcommand\gdb{\gd^{big}}
\newcommand\cp{\mathcal{P}}

\newcommand\gdt{\gd^T}
\newcommand\gdbt{\gd_{big}^T{}}

\newcommand\gdtpl{\gd^{+}}
\newcommand\prpl{pr^+}
\newcommand\prmi{pr^-}

\newcommand\prplgd{pr^+_{\gdt}}

\newcommand\hitr{\operatorname{HI_{tr}}} 
\newcommand\tmot{t^{mot}}

\newcommand\shtpl{SH^+(k)}
\newcommand\shtmi{SH^-(k)}
\newcommand\shtcpl{SH^{+}(k)^c{}}
\newcommand\shtcmi{SH^{-}(k)^c{}}

\newcommand\sinf{\Sigma^{\infty}}
\newcommand\sinft{\Sigma_T^{\infty}}

\newcommand\omd{\proo\Sigma^{\infty}_{\gd}}
\newcommand\omdt{\Sigma^{\infty}_{\gd,T}}
\newcommand\omdp{\proo'\Sigma^{\infty}_{\gd}}

\newcommand\om{\sinf}
\newcommand\omt{\sinft}

\newcommand\omdtpl{\proo\Sigma_{+}^{\infty}}
\newcommand\pom{P\sinf}
\newcommand\ppom{\operatorname{ProP}\sinf}
\newcommand\pomt{P\sinft}

\newcommand\sm{\underline{SmVar}}
\newcommand\opa{\mathcal{OP}}
\newcommand\popa{\proo-\mathcal{OP}}
\newcommand\afo{\mathbb{A}^1}
\newcommand\dopsh{\Delta^{op}Pre_{\bullet}}
\newcommand\dosh{\Delta^{op}Shv_{{\bullet}}}
\newcommand\doshp{\Delta^{op}Pre_{\bullet}'} 

\newcommand\hk{\mathcal{H}_k}

\newcommand\pspt{Pre_{\bullet}}
\newcommand\ho{\operatorname{Ho}}
\newcommand\hogd{\ho_{\gdb}}

\newcommand\hosh{\ho_{\sh}}

\newcommand\mgl{\operatorname{MGl}}
\newcommand\mglp{\operatorname{MGl}'}
\newcommand\mglmod{\modd-\operatorname{MGl}'(k)}

\newcommand\shmgl{D^{\mgl}(k)}
\newcommand\gdmgl{\gd^{\mgl}}
\newcommand\free{\operatorname{Free}}
\newcommand\forg{\operatorname{For}}

\newcommand\tcho{t_{Chow}}

 \DeclareMathOperator\ke{\operatorname{Ker}}
 
\DeclareMathOperator\imm{\operatorname{Im}}
\DeclareMathOperator\co{\operatorname{Cone}}

\DeclareMathOperator\kar{\operatorname{Kar}}
\DeclareMathOperator\adfu{\operatorname{AddFun}}
\DeclareMathOperator\picz{\operatorname{Pic}^{0}}

\newcommand\hrt{{\underline{Ht}}}
\newcommand\hrtt{{\underline{Ht}}^T}

\newcommand\pshtp{P'SH^{T}(k)}
\newcommand\pdmk{P\dmk}
\newcommand\shgm{SH^{\gm}}
\newcommand\shcp{SH^{\cp}}
\newcommand\shccp{SH^{\cp}_c}
\newcommand\gdbgm{\gd^{\gm}_{big}}
\newcommand\gdcp{\gd^{\cp}}
\newcommand\gdm{\gd^{mot}}

\newcommand\omm{\sinf_{\gm}}
\newcommand\ommd{\sinf_{\gd,\gm}}
\newcommand\phwgm{{\tilde{\phi}}^{\gm}}
\newcommand\phgm{\phi^{\gm}}
\newcommand\phgdm{\phi^{\gm}_{\gd}}
\newcommand\phgdbm{\phi^{\gm}_{\gd_{big}}}
\newcommand\psigm{\psi^{\gm}}
\newcommand\phimotgd{\phi^{mot}_{\gd}}

\newcommand\phimot{\phi^{mot}}
\newcommand\psimot{\psi^{mot}}
\newcommand\phishmot{\phi_{\sh}^{mot}}
\newcommand\psishmot{\psi_{\sh}^{mot}}

\newcommand\phimgl{\phi^{\mgl}}
\newcommand\psimgl{\psi^{\mgl}}

\newcommand\gdbgmp{\gd^{\gm'}_{big}}
\newcommand\shgmp{SH^{\gm'}}

\newcommand\gdtb{\gd^T_{big}}
\newcommand\gdpshtp{\gd^{\pshtp}_{big}}

\newcommand\tcp{t^{\cp}}
\newcommand\wcp{w^{\cp}}
\newcommand\wmot{w^{mot}}
\newcommand\hrtmot{\hrt^{mot}}
\newcommand\phwmot{\tilde{\phi}^{mot}}
\newcommand\mgd{M_{\gd}}
\newcommand\gdo{\gd^{old}}

\newcommand\hw{{\underline{Hw}}}
\numberwithin{equation}{subsection}

\begin{document}

 \title{Gersten weight structures for motivic homotopy categories; retracts of cohomology of function fields,  motivic dimensions, and 
 coniveau spectral sequences } 
 \author{M.V. Bondarko
   \thanks{ 
 The work is supported by the 
Russian Science Foundation grant no. 16-11-10073.}}\maketitle

\begin{abstract}
In this paper for any cohomology theory $H$ that can  be
factored through (the Morel-Voevodsky's triangulated motivic homotopy category) $\sh$
we establish  the $\sh^{op}$-functoriality of  coniveau spectral sequences for $H$; similar results also hold for $\sht$, Voevodsky motives $\dmk$,  and for some more "motivic spectral" categories.
We also prove the following interesting result: for any affine essentially smooth semi-local $S/k$ (and for any its localization by  parameters) the Cousin complex for $H^*(S)$ splits if $H$ is a cohomology theory that factors through $\sh$. 
Moreover, the cohomology of such an $S$ is a direct summand of the cohomology of any its open dense subscheme; furthermore, this is true for $S$ being any {\it $\afo$-point} (see below). 

 To prove these facts we construct 
a triangulated category $\gd$ of {\it $\afo$-pro-spectra} that naturally contains the category $\shc$ of compact objects of $\sh$ as well as certain 
pro-spectra of 
 essentially smooth $k$-schemes (along with its $\sht$-version $\gdt$, an also similarly defined $\gdm$, $\gdtpl$, and $\gdmgl$). We decompose (in the sense of Postnikov towers) the  $\sh$-spectrum of a smooth variety   into twisted (pro)spectra of its
points. 
 Those belong to the heart of a {\it Gersten} weight structure $w$ on $\gd$, whereas 
  {\it weight spectral   sequences} corresponding to $w$ generalize  "classical" coniveau ones   (to the $H$-cohomology of arbitrary objects of $\shc$ and $\gd$). When 
  $H$ is represented  by  $Y\in \operatorname{Obj} \sh$,  the corresponding coniveau spectral sequences can be expressed in  terms of the (homotopy)
$t$-truncations of $Y$; this extends to motivic spectra the seminal coniveau spectral sequence computations of Bloch and Ogus.

We also study in detail the following notion: 
 an essentially smooth  scheme $S/k$ is said to be  {\it of $\afo$-cohomological dimension at most $d$}  if $H_{Nis}^n(S,N)=\ns$ for any $n>d$ and any {\it strictly homotopy invariant} Nisnevich 
sheaf $N$. Any localization  of  an affine essentially smooth semi-local scheme at parameters is of $\afo$-cohomological dimension $\le 0$ (we call schemes of these sort  $\afo$-points); an essentially smooth scheme over $k$ is an $\afo$-point if and only if  
 its pro-spectrum $\omd(S_+)$ is a retract of $\omd(S'_+)$ for any dense pro-open subscheme $S'$ of $S$.
 We also prove that the pro-spectra of function fields (in $\gd$, $\gdt$, and in other categories of "motivic" pro-spectra) contain twisted pro-spectra of their residue fields (for all geometric valuations) as retracts; hence the same is true for any cohomology of these fields. 

The analogues of some of  these results for Voevodsky motivic complexes (i.e., for $\dmk$ 
instead of $\sh$) over countable fields were proved by the author in a previous paper. In the current paper we use new methods
 to 
 study a collection of  motivic categories over an arbitrary $k$ (that includes  $\dmk$ also). 
The notion of {\it $\afo$-cohomological dimension} (and its analogues for other motivic categories) is also new. We prove that the {\it $\zoh$-linear $T$-cohomological dimension} (corresponding to $\sht$) of a scheme $S/k$ equals its {\it $\zoh$-linear motivic dimension} whenever $k$ is of finite $2$-adic cohomological dimension; we also apply this result (along with several related ones) to the study of cohomology. 


\end{abstract}

\tableofcontents

 \section*{Introduction}

 In  \cite{bger}  for a countable perfect field $k$  and any cohomology theory $H$ that factors through the Voevodsky's  $\dmge(k)$ the following results were established: 
 coniveau spectral sequences (for the cohomology of smooth varieties) can be functorially extended to $\dmge(k)$;  the cohomology of an essentially smooth affine semi-local scheme  $S/k$  
   is a retract of the cohomology of any its pro-open dense subscheme. 
  Note that both of these results are far from being obvious; the second one implies 
  that the augmented Cousin complex for the cohomology of an (affine essentially smooth)  semi-local 
	 $S$ splits. 
   Also, for  $H$ represented by an object $Y$ of $\dme(k)$ (a {\it motivic complex}; recall that $\dme\supset \dmge$) it was proved that  coniveau spectral sequences for $H$  can be "motivically functorially" expressed in terms of the $t$-truncations of $Y$ (with respect to the homotopy $t$-structure on $\dme(k)$); this is an extension of 
\cite[Proposition 6.4]{blog} (see \S5.3 of \cite{ndegl}).

The goal of the current paper is to extend these results to arbitrary perfect base fields and to a much wider class of cohomology theories. 
This class includes $K$-theory of various sorts, algebraic, semi-topological, and complex cobordism, and Balmer's Witt cohomology.
We succeed in proving these results via considering  certain triangulated categories $\gd\supset \shc$ and $\gdt\supset \sht^c$ (as well as $\gdm\supset \dmgm$, 
$\gdtpl\supset \shtpl^c$, and cobordism-module categories $\gd^{\mgl}\supset \shmgl{}^c =\ho(\mglmod)^c$), and introducing  {\it Gersten weight structures} on them. These weight structures are {\it orthogonal} to the homotopy $t$-structures for the corresponding motivic categories (via certain {\it nice dualities} that we construct). 
The formalism of weight structures easily yields 
 several functoriality  
  and direct summand results (see Theorems \ref{tshtt}(II--V), \ref{tfunct}, \ref{tmotdimt}, and \ref{tprimacycl}   below).
 Our direct summand results are much stronger than the {\it universal exactness} statement
 of \cite{suger} (see Remark \ref{runiv} for more detail). 

Moreover, in this paper we introduce the notion of $\afo$-cohomological dimension (as well as its analogues for $\sht$ and other motivic categories) for 
essentially smooth $k$-schemes (i.e., for localizations of smooth $k$-varieties; actually, we also consider inverse limits of localizations that do not correspond to any schemes). 
An essentially smooth $S/k$ is of $\afo$-cohomological dimension at most $d$ whenever $H_{Nis}^n(S,N)=\ns$ for any $n>d$ and any {\it strictly homotopy invariant} sheaf $N$ (on $\sm$); we prove several other characterizations of this condition in Theorem \ref{tds} (cf. also  Remark \ref{rds}(2)).
In particular, a connected $S$ is of dimension at most $0$ (we will say that $S$ is an $\afo$-point in this case) if and only if its the corresponding 
 $\afo$-pro-spectrum $\omd(S_+)\in \obj \gd$ is a retract of  $\omd(S_{0+})$, where $S_0$ is the generic point of $S$. 
 Moreover, if $S$ is an $\afo$-point then  $\omd(S_+)$ is a retract of  $\omd(S'_{+})$  for any dense pro-open subscheme $S'$ of $S$.
 More generally, $S$ is of $\afo$-dimension at most $d$ if and only if $\omd(S_+)$ is a retract of  $\omd(S'_{+})$ whenever $S'$ is an open subscheme of $S$ with complement of codimension at least $d+1$. 
 The notions of $\afo$-points and of $\afo$-cohomological dimension appear to be interesting and non-trivial; in particular, we prove that any  complement to an (essentially smooth) semi-local scheme of a divisor with smooth crossings is an  $\afo$-point. 

We also study {\it $T$-cohomological dimension} (corresponding to $\sht$), {\it motivic dimension} (corresponding to $\dmk$) and other versions of this notion. 
Somewhat surprisingly, the motivic dimension of $S$ is not smaller than its {\it $+$-dimension} (corresponding to $\shtpl$); it equals its $2$-linear $T$-cohomological dimension if $k$ is of finite $2$-adic cohomological dimension. These results have several consequences for cohomology of (essentially smooth) {\it primitive schemes}. 

We also note that our consideration of (triangulated)  motivic pro-spectral categories (along with the purity distinguished triangle for the pro-spectra of {\it pro-schemes}; see Proposition \ref{pinfgy}) can be interesting for itself independently from the rest of the paper. It appears that the only triangulated motivic categories of pro-objects considered in the existing literature are the {\it comotivic} one 
of \cite{bger} and the motivic pro-spectral  categories studied  in the current paper; note in contrast that the category of {\it pro-motives} introduced in \S3.1 of \cite{deggenmot} is not triangulated. Our methods of constructing motivic pro-spectra are not  original (we mostly use the results of \cite{tmodel} and related papers) and not specific to motivic settings; yet this has 
the following advantage: they 
  allow the construction of motivic pro-spectra over an arbitrary base scheme $S$ (see Remark \ref{relgws}(1)). 

Let us  describe the contents  of the paper. Some more information of this sort can be found at the beginnings of sections; moreover, we  list some of the main definitions and notation in the end of \S\ref{snot}.

In \S\ref{sra1} we recall some definitions (related to triangulated categories), introduce {\it pro-schemes} (certain projective limits of pointed smooth varieties),  and recall some properties of the $\af^1$-stable homotopy category, 
 of the Morel's homotopy $t$-structure $t$ on it. 

In \S\ref{sws} we recall the general formalism of weight structures (along with  the properties of {\it cocompactly cogenerated} ones)  as well as  their relation to ({\it orthogonal}) $t$-structures, weight filtrations, and spectral sequences. Note that the more complicated matters considered in the end of the section are not necessary for the proof of  the aforementioned "splitting" results. 

In \S\ref{scomot} we embed $\sh$ into a certain  triangulated  category $\gdb$ of  {\it $\afo$-pro-spectra}, and also construct a closely related embedding $\shc\to \gd$ ($\gd$ is the subcategory of $\gdb$ cogenerated by $\shc$). We use the  main properties of $\gdb$ for the construction of a certain {\it Gersten weight structure} $w$ on  $\gd\subset \gdb$. 
$w$ possesses several nice properties; in particular, the $\afo$-pro-spectra of function fields and of (affine essentially smooth) semi-local schemes over $k$
  belong to  $\gd_{w= 0}$ (and "cogenerate" it), whereas the $\afo$-pro-spectra of arbitrary pro-schemes belong to $\gd_{w\ge 0}$. We also introduce the notion of $\afo$-cohomological dimension and $\afo$-points: a pro-scheme is an $\afo$-point if its pro-spectrum belongs to the heart of $w$; this is equivalent to the vanishing of all its higher cohomology with coefficients in strictly homotopy invariant sheaves. We prove that localizations of (essentially smooth affine) semi-local schemes at parameters are $\afo$-points if $k$ is infinite, and that  $\omd(S_+)$ is a retract of $\omd(S_{0+})$ for any dense open sub-pro-scheme $S_0$ of $S$ if and only if $S$ is an $\afo$-point.
It follows that $\omd(\spe (K)_+)$ (for a function field $K/k$)  contains (as retracts)   the pro-spectra of semi-local schemes whose generic point is $\spe (K)$,     as well as the twisted pro-spectra of residue  fields of $K$ (for all geometric valuations). Note that all our arguments easily carry over to the $T$-stable setting (i.e., to $\sht$ instead of $\sh$) 
and to other "motivic" categories; we discuss this matter in \S\ref{ssupl}. 

In \S\ref{sapcoh} 
 we apply the results  of the previous section to cohomology. So we prove that the  augmented Cousin complex for the cohomology of 
 an $\afo$-point splits. We also prove that weight spectral sequences and filtrations corresponding to $w$ are canonically isomorphic to the "usual" coniveau ones (for any {\it extended} cohomology of smooth varieties; note that all "reasonable" cohomology theories are extended). On the other hand, the general theory of weight spectral sequences yields that the corresponding spectral sequence $T(H,M)$ converging to $H^*(M)$ (for $M$ being an object of $\shc$ or $\gd$) is $\shc{}^{,op}$-functorial in $M$ starting from $E_2$; this is far from being trivial from the "classical" definition of coniveau spectral sequences. Next we construct a {\it nice duality} $\Phi:\gd^{op}\times \sh\to \ab$ (one may say that this is a "regularization" for the 
 bifunctor $\gdb(-,-)$). Since  $w$ is {\it orthogonal} to $t$ with respect to $\Phi$,  for a cohomology theory "$\Phi$-represented" by $Y\in \obj \sh$ our (generalized) coniveau spectral sequences can be expressed in terms of  the $t$-truncations of $Y$ (starting from $E_2$); this vastly generalizes 
 Proposition 6.4 of \cite {blog}. Moreover, if one defines generalized coniveau spectral sequences for inverse limits of smooth varieties as the corresponding direct limits (starting from $E_2$) then this fact remains to be valid.  

In \S\ref{ssupl} we describe 
possible variations of our methods and results. 
 First we prove the natural analogues  of our main results for the triangulated categories of motivic $T$-spectra (and the corresponding pro-spectra). Next we pass to the category $\gdm$ of {\it comotives} that contains the Voevodsky category $\dmgm$,\footnote{As a consequence, we prove that the results of \cite{bger} are 
 valid for  any perfect 
 $k$.}  
 $\tau$-positive $T$-spectra $\shtpl$ and the corresponding $\gdtpl$, and also triangulated categories of $\mgl$-modules. We use Quillen's model category formalism to construct exact functors $\gdt\to \gdm$ and $\gdt\to \gdmgl$, and certain "simpler triangulated"  methods for studying the natural functor $\gdt\to \gdtpl$.
We study in detail the relations between the "cohomological dimensions" corresponding to these categories and their relation to cohomology theories (including oriented 
 cohomology  and  cohomology coming from the "complex" topological realization). 
So we prove that motivic dimensions of pro-schemes "control splittings" for their $+$-motivic pro-spectra and the vanishing of the corresponding terms of generalized coniveau spectral sequences for a wide range of cohomology theories. 
We also consider the "Artin-Tate substructures" for our weight structures; we obtain that weight complexes (and so, also generalized coniveau spectral sequences) for the corresponding motivic spectra are "quite economical". Lastly we  discuss "relative motivic" analogues of our (Gersten) weight structures.

Our notation below partially follows the one of  \cite{morintao}. 
 All "concrete" $t$-structures considered in this paper are certain versions of (Morel's) homotopy ones (yet our general notation for $t$-structures is quite distinct from the one of ibid., and we introduce it in \S\ref{dtst}).

The author is deeply grateful to  prof. M. Levine for his very useful remarks and  for his hospitality during the author's staying in the Essen University, and also to the officers the Max Planck Institut f\"ur Mathematik for the wonderful working conditions. 
He would also like to express his gratitude to prof. A. Ananyevskiy, prof. T. Bachmann, prof. B. Chorny, prof. F. D\'eglise, prof. G. Garkusha, prof. A. Gorinov,  prof. F. Muro,  prof. D. Nardin, prof. I. Panin, prof. D. Pavlov, prof. P. Pelaez, prof. T. Porter, and prof. D. White for their extremely helpful comments.

\section{Preliminaries: $t$-structures, Postnikov towers, and $\af^1$-connectivity of spectra}\label{sra1}


In this section we recall some notation, introduce some new one, and state some facts on motivic homotopy categories.


In \S\ref{snot} we introduce some (mostly, categorical) notation.

In \S\ref{dtst} we recall the notion of  $t$-structure (and introduce some notation for it) and of a Postnikov tower for an object of a triangulated category. 

In \S\ref{sprs} we define and study the categories $\opa$ (essentially of open embeddings in $\sv$) and $  \popa$; certain objects of the latter category are called pro-schemes. 
 
In \S\ref{ssh}
 we recall some 
properties of $\sh$ and the Morel's homotopy $t$-structure $t$ on it.

\subsection{Notation and conventions}\label{snot} 

For categories $C,D$ we write $D\subset C$ if $D$ is a full 
subcategory of $C$.

 Given a category $C$ and  $X,Y\in\obj C$ we we will write
$C(X,Y)$ for  the set of morphisms from $X$ to $Y$ in $C$.
We will say that $X$ is  a {\it
retract} of $Y$ if $\id_X$ can be factored through $Y$.\footnote{Certainly,  if $C$ is triangulated or abelian then $X$ is a retract of $Y$ if and only if $X$ is its direct summand.}\

For any $D\subset C$ the subcategory $D$ is called {\it Karoubi-closed} in $C$ if it
contains all retracts of its objects in $C$. We will call the
smallest Karoubi-closed subcategory of $C$ containing $D$  the {\it
Karoubi-closure} of $D$ in $C$; sometimes we will use the same term
for the class of objects of the Karoubi-closure of a full subcategory
of $C$ (corresponding to some subclass of $\obj C$).

The {\it Karoubi envelope} $\kar(B)$ (no lower index) of an additive
category $B$ is the category of ``formal images'' of idempotents in $B$
(so $B$ is embedded into an idempotent complete category; it is
triangulated if $B$ is).
We will say that $B$ is Karoubian if the canonical embedding $B \to \kar
(B)$ is an equivalence, i.e.,
if any idempotent morphism yields a direct sum decomposition in $B$.

For a category $\cu$ the symbol $\cu^{op}$ will denote its opposite category.

$\ab$ is the category of abelian groups.

 For an additive category $\cu$ an  $X\in \obj\cu$ is called {\it cocompact} if
$\cu(\prod_{i\in I} Y_i,X)=\bigoplus_{i\in I} \cu(Y_i,X)$ for any
set $I$ and any $Y_i\in\obj \cu$ (below we will only consider cocompact objects in categories closed with respect to arbitrary small products). 
Dually, a {\it compact} object of $\cu$ is a cocompact object of $\cu^{op}$, i.e., $M$ is compact (in a $\cu$  closed with respect to arbitrary small coproducts) if  the functor $\cu(M,-):\cu\to \ab$ respects coproducts.

For $X,Y\in \obj \cu$ we will write $X\perp Y$ if $\cu(X,Y)=\ns$.
For $D,E\subset \obj \cu$ we will write $D\perp E$ if $X\perp Y$
 for all $X\in D,\ Y\in E$.
Given $D\subset\obj \cu$ we  will write $D^\perp$ for the class
$$\{Y\in \obj \cu:\ X\perp Y\ \forall X\in D\}.$$
Sometimes the symbol $D^\perp$  will be used to denote the corresponding
 full subcategory of $\cu$. Dually, ${}^\perp{}D$ is the class
$\{Y\in \obj \cu:\ Y\perp X\ \forall X\in D\}$. 

In this paper all complexes will be cohomological, i.e., the degree of
all differentials is $+1$; respectively, we will use cohomological
notation for their terms. We will need the following easy observation on the homotopy category of complexes for an arbitrary 
additive category $B$.

\begin{lem}\label{lrcomp}
Suppose that  a complex $M=(M^i)\in \obj K(B)$ is a $K(B)$-retract of an object of $B$ (i.e., of a complex of the form $\dots 0\to 0\to M'\to 0\to 0\to \dots$ for some  $M'\in \obj B$; we will use the notation $M'[0]$ for this complex). Then $M$ is also a retract of $M^0[0]$.  

\end{lem}
\begin{proof}
Suppose that the factorization $M\stackrel{f}{\to} M'[0] \stackrel{g}{\to}M$ yields the $K(B)$-retraction mentioned. Then $h=g\circ f$ cannot have non-zero components in non-zero degrees, and we can consider it as a morphism $M\to M^0[0]$ and vice versa. Hence $h$  yields the retraction in question.
\end{proof}

$\cu$ and $\du$ will usually denote some triangulated categories.
 We will use the
term {\it exact functor} for a functor of triangulated categories (i.e.,
for a  functor that preserves the structures of triangulated
categories).

A class $D\subset \obj \cu$ will be called {\it extension-closed} if $0\in D$ and for any
distinguished triangle $A\to B\to C$  in $\cu$ we have the implication $A,C\in
D\implies B\in D$. In particular, an extension-closed $D$ is strict (i.e., contains all
objects of $\cu$ isomorphic to its elements).

The smallest extension-closed $D$ containing a given $D'\subset \obj \cu$ will be called the {\it extension-closure} of $D'$.

We will call the smallest Karoubi-closed triangulated subcategory 
of $\cu$ containing  $D'$ the {\it triangulated subcategory  densely generated by} $D'$; we will 
write $\lan D'\ra$ for this category.

We will say that $D'$ {\it  cogenerates} $\cu$ if $\cu$ is closed with respect to small products and  coincides with its smallest strict triangulated subcategory that fulfils this property and contains $D'$.\footnote{In  \cite{bger} $D$ was called a class of {\it weak cogenerators} of $\cu$ in this case.}  

We will say that a partially ordered index set $I$ is {\it filtered} if  if for any $i,j\in I$ there exists some maximum, i.e., an $l\in I$ such that $l\ge i$ and $l\ge j$.
For the category $\underline{I}$ corresponding to $I$ (with $\underline{I}(i,j)$ consisting of a single element if $j\le i$ and empty otherwise) and any category $C$ we will call a functor $\underline{I}\to C:i\mapsto X_i$ a
 {\it projective system} (so, all projective systems of this paper are filtered; cf. also  the discussion 
 in the beginning of \S\ref{sprs} below). 
 For such a system we have the natural notion of  inverse limit. 
 Dually, we will call the inverse limit of a system of $X_i\in \obj C^{op}$ the direct limit of $X_i$ in $C$.
 

$\au$ will usually denote some abelian category.
We will 
usually assume that $\au$ satisfies the AB5 axiom, i.e., 
 is closed with respect to all small coproducts, and  filtered direct limits of exact sequences in $\au$ are exact.

We will call a covariant additive functor $\cu\to \au$
for an abelian $\au$ {\it homological} if it converts distinguished
triangles into long exact sequences; homological functors
$\cu^{op}\to \au$ will be called {\it cohomological} when considered
as contravariant functors from $\cu$ into $ \au$.

We will often specify a distinguished triangle by two of its
arrows.
For $f\in\cu (X,Y)$, $X,Y\in\obj\cu$, we will call the third vertex
of (any) distinguished triangle $X\stackrel{f}{\to}Y\to Z$ a cone of
$f$.

 For additive categories $C,D$ we will write $\adfu(C,D)$ for the
category of additive functors from $C$ to $D$
(we will ignore set-theoretic difficulties here since
we will mostly need the categories of functors from those $C$ that are skeletally small).

$k$ will be our perfect 
 base field of characteristic $p$; $p$ may be $0$. 

We also list some more definitions and the main notation of this paper.

$t$-structures (along with $\hrt$) and Postnikov towers are considered in \S\ref{dtst}; the categories $\opa\subset \popa$,  the twists $-\brj$ for them,  pro-schemes and some notions related to them are 
 recalled in \S\ref{sprs} (whereas 
 normal bundles for closed embeddings of pro-schemes  are 
defined in \S\ref{scgersten});  $\sh$, the homotopy $t$-structure $t$ on it, strictly 
 $\afo$-invariant sheaves, the functor 
$\sinf$,  "shifted twists"  $-\brjj$, 
semi-local (schemes and) pro-schemes, and normal bundles $N_{X,Z}$  are mentioned in \S\ref{ssh}; 
weight structures (along with $\cu_{w\le i}$, $\cu_{w= i}$, $\cu_{w\ge i}$, $\hw$, and $\cu_{[i,j]}$), 
 weight Postnikov towers, and cocompactly cogenerated weight structures are defined in \S\ref{swr}; 
 weight complexes, weight filtrations $W^mH$, weight spectral sequences $T(H,M)$ and $T^{\ge 2}(H,M)$, and  pure cohomology theories 
are defined in  \S\ref{swfs}; (nice) dualities $\Phi$ of triangulated categories and orthogonal weight and $t$-structures are defined in \S\ref{sort}; the categories 
 $\gdp$,  $\gdb=\ho(\gdp)$, and $\omd$ 
 are introduced in \S\ref{comot}; our main (Gersten) weight structure $w$ (for the category $\gd$) along with the notions of $\afo$-points and $\afo$-cohomological dimension are defined in \S\ref{scwger}; we introduce extended cohomology theories in \S\ref{sextkrau}; Cousin complexes $T_H(-)$ are studied in \S\ref{sext};    generalized coniveau spectral sequences are introduced in \S\ref{sdconi}; 
 $\sht$, $\gdt$, $\omt$, and $T$-points are   are considered in \S\ref{sht}; the category $\gdm$ of comotives along with the general $\phgm:\sht\to \shgm$, $\cp$, and $\gdcp$ are 
defined in \S\ref{sdm};  
Morel's $\eta$,  various notions of orientability, and modules over the Voevodsky spectra $\mgl$ are treated in \S\ref{swo}; the category $\shtpl$ of {\it $\tau$-positive motivic spectra}, the corresponding $\gdtpl$ and $w^+$ are considered in \ref{sshinvp}.


\subsection{On $t$-structures and Postnikov towers in triangulated categories}
\label{dtst}

To fix  notation, let us recall the definition of a $t$-structure.

\begin{defi}\label{dtstr}

A couple of subclasses  $\cu^{t\ge 0},\cu^{t\le 0}\subset\obj \cu$
for a triangulated category $\cu$ will be said to define a
$t$-structure $t$ if $(\cu^{t\ge 0},\cu^{t\le 0})$  satisfy the
following conditions:

(i) $\cu^{t\ge 0},\cu^{t\le 0}$ are strict (i.e., contain all
objects of $\cu$ isomorphic to their elements).

(ii) $\cu^{t\ge 0}\subset \cu^{t\ge 0}[1]$, $\cu^{t\le
0}[1]\subset \cu^{t\le 0}$.

(iii) {\bf Orthogonality}. $\cu^{t\le 0}[1]\perp
\cu^{t\ge 0}$.

(iv) {\bf $t$-decompositions}. For any $X\in\obj \cu$ there exists a distinguished triangle
\begin{equation}\label{tdec}
A\to X\to B[-1]{\to} A[1]
\end{equation} such that $A\in \cu^{t\le 0}, B\in \cu^{t\ge 0}$.

\end{defi}

We will need some more notation. 

\begin{defi} \label{dt2}

1. The subcategory  $\hrt\subset \cu$ whose objects are $\cu^{t=0}=\cu^{t\ge 0}\cap
\cu^{t\le 0}$ 
 will be called the {\it heart} of
$t$. Recall (cf. Theorem 1.3.6 of \cite{bbd}) that $\hrt$ is abelian
(and short exact sequences in $\hrt$ come from distinguished triangles in $\cu$).

2. $\cu^{t\ge l}$ (resp. $\cu^{t\le l}$) will denote $\cu^{t\ge
0}[-l]$ (resp. $\cu^{t\le 0}[-l]$).


\end{defi}

\begin{rema}\label{rts}

1. Recall (cf. Lemma IV.4.5 in \cite{gelman}) that (\ref{tdec})
defines additive functors $\cu\to \cu^{t\le 0}:X\to A$ and $C\to
\cu^{t\ge 0}:X\to B$. The objects $A$ and $B$ will be denoted by $X^{t\le 0}$ and
$X^{t\ge 1}$, respectively.

The triangle (\ref{tdec}) will be called the {\it t-decomposition} of $X$. If
$X=Y[i]$ for some $Y\in \obj\cu$, $i\in \z$, then  $A$ will be denoted 
by $Y^{t\le i}$ (note that it belongs to $\cu^{t\le 0}$) and $B$ by $Y^{t\ge
i+1}$ (it belongs to  $\cu^{t\ge 0}$), respectively. 
Objects of the type $Y^{t^\le i}[j]$ and
$Y^{t^\ge i}[j]$ (for $i,j\in \z$) will be called {\it
$t$-truncations of $Y$}.

2. We will write $X^{t=i}$ for the $i$-th (co)homology of $X$ with respect
to $t$, that is $(X^{t\le i})^{t\ge 0}$ (cf. part 10 of \S IV.4 of
\cite{gelman}). 

3.  The following statements are obvious (and well-known): $\cu^{t\le 0}={}^\perp
(\cu^{t\ge 1})$; $\cu^{t\ge 0}= (\cu^{t\le -1})^\perp$.

4. We also recall that Theorem A.1 of \cite{talosa} says the following: if $\cu$ is a triangulated category closed with respect to (small) coproducts and $\cp$ is a set of compact objects in it then there exists a unique $t$-structure $t_{\cp}$ on $\cu$ that is {\it generated} by $\cp$, i.e., $\cu^{t_{\cp}\ge 1}=
\cp\perpp$. Moreover,   $\cu^{t_{\cp}\le 0}$ equals the smallest extension-closed subclass of $\obj \cu$ that contains $\cup_{i\ge 0}\cp[-i]$ and is closed with respect to coproducts. 

5. Our conventions for $t$-structures and "weights" (see Remark \ref{rstws}(2) below) follow the ones of \cite{bbd}. So, for any $n\in \z$ "our" $(\cu^{t\le n},\cu^{t\ge n})$ corresponds to $(\cu_{t\ge -n},\cu_{t\le -n})$ in the notation of the papers \cite{morintao}, \cite{tmodel}, and \cite{degorient}.
\end{rema}

Below we will need the notion of  Postnikov tower in
a triangulated category several times (cf. \S IV.2 of \cite{gelman}; Postnikov towers are closely related (in an obvious way) to {\it triangulated exact couples} of \S2.4.1 of \cite{degorient}).

\begin{defi}\label{dpoto} 
Let $\cu$ be a triangulated category.

1. Let $l\le m\in \z$.

We will call a {\it bounded Postnikov tower} for $M\in\obj\cu$ the
following data: a sequence of $\cu$-morphisms $(0=)Y_l\to
Y_{l+1}\to\dots \to Y_{m}=M$ along with distinguished triangles
\begin{equation}\label{wdeck3}
 Y_{ i-1} \to Y_{i}\to M_{i}
\end{equation}
for some $M_i\in \obj \cu$;
here $l<i\le m$.

2. An unbounded Postnikov tower for $M$ is a collection of $Y_i$ for
$i\in\z$ that is equipped (for all $i\in\z$) with the following: connecting arrows
$Y_{i-1}\to Y_{i}$ (for $i\in\z$), morphisms $Y_i\to M$ such that all
the corresponding triangles commute, and distinguished triangles
(\ref{wdeck3}).

In both cases the object  $M_{-p}[p]$  will be denoted by $M^p$; we will call $M^p$ the {\it factors} of our Postnikov tower.

\end{defi}

\begin{rema}\label{rwcomp}
1. Composing (and shifting) arrows from   triangles (\ref{wdeck3}) for two subsequent $i$ one can construct a complex whose terms are
$M^p$ (it is easily seen that this is a complex indeed; see Proposition 2.2.2 of \cite{bws}). 

2. Certainly, a bounded Postnikov tower can be easily completed to an unbounded one. For example, one can take $Y_i=0$ for $i<l$, $Y_i=M$ for $i>m$; then $M^i=0$ if $i<l$ or $i\ge m$.
\end{rema}

\subsection{On open pairs and pro-schemes}\label{sprs}
We introduce a collection of definitions and observations that are essentially related to relative cohomology. One may note that we do not really need the full generality of our constructions; hence they possibly can be simplified or modified.

\begin{itemize}

\item
$ \var$ (resp. $\sv$) 
will denote the class of all (smooth) varieties over $k$ 
 (we call a $k$-scheme a variety if it is reduced, separated, and of finite type over $k$, i.e., we do not assume it to be connected; we also allow empty varieties). 
$\sm$ is the category of smooth $k$-varieties.

We will write $\pt$ for the point $\spe k$, and $\gmm$ for $\afo \setminus \ns$. 

\item
We consider the category $\opa'$ 
of  open 
 embeddings of smooth varieties over $k$. The object corresponding to an embedding $U\to X$ will be denoted by $X/U$, and $X/\emptyset$ will be denoted by just by $X$.

We define $\opa$ as the localization of $\opa'$ by the natural morphisms of the sort $X/U\to (X/U) \sqcup (Y/Y)$ (for $Y\in \sv$; so, objects of $\opa$ are open embeddings also). In particular, any $X/U\in \obj \opa$ is isomorphic to $(X/U)_+=
(X/U)\sqcup (\pt/\pt)$; 
 hence $X\cong X_+=(X\sqcup \pt)/\pt$. 

We have the obvious (componentwise) disjoint union bi-functor operation on $\opa'$ that obviously induces and operation on $\opa$;  we will use the notation  $\sqcup$ for the latter. 

\item Now we define certain "twists" on $\opa$ (one may call this twists shifted Tate ones).

We start from a more general definition.  We define the bi-functor $\wedge$ as follows: $X/U\wedge X'/U'=X\times X'/(X\times U'\cap X'\times U)$. Obviously, $\wedge$ gives an 
 associative functorial operation both on $\opa'$ and on $\opa$, and  
 $\wedge$ is distributive with respect to $\sqcup$. 

Now we define $T=\afo/\gmm$ and for any $j\ge 0$ we set $-\brj=-\wedge (T^{\wedge j})$; thus for $Z=X\setminus U$ we have
$X/U\lan 1\ra=X\times \afo/(X\times \afo\setminus Z\times \ns)$.  Moreover, for any $X\in \sv$ we have $X\brj\cong X_+\brj$ in $\opa$.

Note also that the functor $-\wedge \afo$ sends $X/U$ into $X\times \afo/U\times \afo$.

\item
Moreover, we will need the category $\popa$ of (filtered formal) 
 pro-objects of $\opa$. 

In  the current paper we only consider pro-objects "indexed"  by filtered sets. Note however that in 
 lots of  papers on pro-objects (including \cite{tmodel}
that provided results important for our treatment of categories of pro-objects)  
  the "categorical" version of pro-objects was considered instead, i.e., instead of limits over (filtered) projective systems limits over filtered small categories are considered. This certainly yields "more pro-objects"; however, 
this (formal) distinction is well-known not to be important (in particular, for  
 the purposes of the the current paper as well as for the ones of \cite{bpure}; see Remark 5.5.1 of ibid. or the discussion in \S2 of \cite{isalim} for more detail).
	Our reason to prefer "set-indexed pro-objects" (and the corresponding inverse limits) is that the main pro-objects of this paper (i.e., pro-schemes that we will now define) "naturally" correspond to indexed sets; 	this choice also gives an excuse  not to specify the indexing category (usually). 
	
	\item Now let us introduce certain operations on $\popa$.
		
Since the endo-functor $-\brj$ is well-defined on $\opa$, it also extends to $\popa$. Similarly we obtain the existence of finite disjoint unions ($\sqcup$) on $\popa$.

Now we define the   disjoint union of  a  set   $S_i\in \obj\popa$ (for $i\in I$, $I$ is a possibly infinite set). We can certainly assume that $S_i=S'_{i+}$ for some $S_i'\in \obj \opa$; hence one can define $\sqcup S_i$ as the projective limit $(\sqcup_{i\in J} S'_{i})_+$; here $J$ runs through all finite subsets of $I$ and the transition morphisms map $S'_i$ either onto themselves or into the "extra point".\footnote{Recall (from Theorem 4.1 of \cite{isalim}) that $\popa$ is closed with respect to arbitrary (filtered) inverse limits. Note also that  the transition maps between  $(\sqcup_{i\in J} S'_{i})_+$ also induce canonical $\opa$-morphisms between $\sqcup_{i\in J} S_{i}$; hence we can write $\sqcup S_i = \prli_J (\sqcup_{i\in J} S_{i})$.}

It is easily seen that this operation "commutes with $-\brj$" (up to canonical isomorphisms) since this is true for finite disjoint unions.

\item Now we  introduce a special sort of objects of $\popa$. We will call objects of $\popa$ of the form $\prli X_i$, where $X_i$ are (non-empty) connected smooth varieties, and the transition morphisms are open embeddings, {\it connected pro-schemes}. 
The disjoint union $S$ of connected pro-schemes $S_i$ will be called (just) pro-schemes; we will say that $S_i$ are {\it connected components} of $S$. 

 Note that connected pro-schemes (and their finite disjoint unions) often yield actual $k$-schemes.

\item Let us introduce some conventions for pro-schemes.

For pro-schemes $U=\prli U_i$ and $V=\prli V_j$ (for $U_i,V_j\in \sv$) we will call 
a $\popa$-morphism $i$ from $U$ into $V$ an embedding (resp. an open embedding, resp. a closed embedding)
if  it
 can be obtained (as the limit) from (the corresponding disjoint unions of)  
  embeddings (resp. open embeddings, resp. closed embeddings) $U_i\to V_j$; if this is the case we will say that $U$ is a sub-pro-scheme (resp.  an open sub-pro-scheme, resp. a closed sub-pro-scheme) of $V$. Note that one can easily define what does it mean for subscheme to be of codimension at least $d$ (resp. $d$) in a similar way. Moreover, for an open sub-pro-scheme we will speak about the codimension of the complement (even though it does not satisfy any smoothness assumptions in general). 

\item 
Any collection of spectra of function fields over $k$ certainly yields a pro-scheme. Thus one can also speak of  Zariski points of  pro-schemes $S$ (defined as sub-pro-schemes of the form $\spe (K)_+$ for $K$ being a $k$-function field); moreover, points have well-defined codimensions. 

Furthermore,  we will say that $S$ is of dimension $\le d$ if $S_+\cong \prli S_{i+}$ for some $S_i$ of dimension at most $d$. Alternatively, one can check whether the transcendence degrees of all residue fields of $S$ 
  over $k$ are at most $d$.\footnote{So, the dimension of a pro-scheme  coming from an essentially smooth $k$-scheme $S$ can be bigger than the Krull dimension of $S$.}
\end{itemize}

\begin{rema}\label{resmooth}
1. We do not have to claim that connected components, Zariski points, and  dimensions of pro-schemes do not depend on the choice of "presentations" of the corresponding objects and morphisms in $\popa$ (as infinite disjoint unions and so on). The reason for this is that in the situations we are interested in one can easily make a "standard" choice of these presentations that "comes from geometry". However, we describe a method that can probably yield the independence from choices in question.

One should start from defining certain functors on $\opa$. In particular, to $X/U$ one can associate the set of connected components of $Z=X\setminus U$ and the set of Zariski points of $Z$. Next one should pass to the (inverse) limit to obtain certain "generalized" connected components and Zariski points for all objects of $\popa$. It is easily seen that the restriction of this definition to pro-schemes does yield connected components and Zariski points in the sense that we have described above.  One can also consider the limit topology on these points, generic points, their closures, etc.

2. Certainly, the most important pro-schemes are the ones coming from essentially smooth $k$-schemes. 
However, we will sometimes need (connected) pro-schemes that are not (necessarily) of this sort; cf. Remark \ref{lger}(1) below.

3. Until \S\ref{ssupl} we will only consider various twists ($\brj$ and $\brjj$) in the case $j>0$. Thus in most of the statements concerning them we could have set $j=1$.
\end{rema}

\subsection{$\sh$ and the homotopy $t$-structure on it: reminder}\label{ssh}

Let us recall some properties of  $\sh$ and the injective model structure for it. 
In this paper all the model categories will have functorial factorizations of morphisms (though this assumption is probably not really important).

Denote by $\pspt$ 
 the category of  presheaves 
 of pointed sets on $\sm$.  

\begin{pr}\label{ppsh}
There exist 
closed 
model structures for the categories $\dopsh$ 
for $\doshp$ that is equal to $\dopsh$ as an "abstract" category,
 and for a certain category $\psh$ (see below)
satisfying the following properties.

1.  The homotopy categories of $\doshp$  and $\psh$ are  
naturally isomorphic to the Morel-Voevodsky categories $\hk$ and $\sh$, respectively. Moreover, $\psh$ is a proper simplicial model category. 

2. The cofibrations 
in $\dopsh$ 
 are exactly the (levelwise) injections; 
hence all objects of this category are cofibrant. 

3. The natural comparison functors $\dopsh 
 \to \doshp \to  \psh$ 
are left 
Quillen functors. 

4. 
$\doshp$ is a monoidal model category in the sense of 
\cite[Definition 4.2.6]{hovey}. Moreover, there exists a natural bifunctor 
  $\otimes:\doshp\times \psh\to \psh$ that turns $\psh$ 
into a	$\doshp$-module category in the sense of \cite[Definition 4.2.18]{hovey} such that the connecting functor $\doshp\to \psh$ is a $\doshp$-module one.

\end{pr}
\begin{proof}
1,2. 
We take the Nisnevich-local injective model structures for  $\dopsh$ and $\doshp$ (see 
\S7.2 and Example 7.20 of \cite{jardbook}) 
and the model structure for $\psh$ mentioned in Example  10.38 of ibid.

 Then 
all the statements 
in question are contained in ibid.; see Theorem 7.18  and Theorem 10.36 of ibid. 

3. This statement is well-known to experts. 
The first of these functors is just a left Bousfiled localization one; so the result is straightforward. 

Now we treat the second functor. Recall that $\psh$ is the category of {\it $S^1$-spectra} in $\doshp$, whereas the latter is (left) proper cellular according to Theorem 7.18 of \cite{jardbook} (see also the text preceding Example 7.20 of ibid).
Thus combining  Proposition 1.15 of \cite{hoveysp} (applied to 
 the functor $F_0$ in the notation of loc. cit.) with the results of \S3 of ibid. (note that the stable model structure that  is considered in loc. cit. 
  is obtained via a left  Bousfield localization from the so-called projective model structure considered in \S1 of ibid.) we obtain the statement in question. 

4. These statements appear to be well-known also. $\doshp$ is a monoidal model category according to Theorem 1.9 of \cite{hornloc}. Thus the second part of the assertion is an easy consequence of Theorem 6.3 of \cite{hoveysp} (whose assumptions are fulfilled since all objects of $\doshp$ are cofibrant). 
\end{proof}

\begin{rema}\label{rppsh}
1. Moreover,  $\sh$ is triangulated monoidal with respect to the operation $\wedge$ compatible with the obvious $\wedge$ for pointed presheaves. 
This is due to the fact that $\psh$ is Quillen equivalent to the corresponding model category of {\it symmetric} $\afo$-spectra that is symmetric monoidal itself. 
Thus one may say that our distinction between the operations $\wedge$ and $\otimes$ is "not essential".
However, the pairing introduced in part 4 of our proposition is sufficient for our purposes, and we prefer to avoid symmetric spectra until \S\ref{sdm}.

2. For $X/U\in\obj \opa$ we define $\om(X/U)$ (resp. $\pom(X/U)$) as the image in $\sh$ (resp. in $\psh$) of the  
discrete pointed simplicial presheaf $\sm(-,X \sqcup \pt)/\sm(-,U\sqcup \pt)\in \obj \dopsh$,  i.e., of the presheaf  sending $Y\in \sv$ into $\sm(Y,X)/\sm(Y,U)$ pointed by the image of $\sm(Y,U)$ if the latter is non-empty, and into $\sm(Y,X)\sqcup\pt$ pointed by this point in the opposite case (cf. \S2.3.2 
 of \cite{degdoc}). We note that this correspondence yields  well-defined functors indeed since $\pom$ certainly sends all the morphisms of the form $X/U\to (X\sqcup Y)/(U\sqcup Y)$ into $\psh$-isomorphisms. 

3. 
Below we will only need the following property of the bi-functor $\otimes$ for any cofibrant $U\in \obj \doshp$\footnote{Recall that actually all objects of $\doshp$ are cofibrant.} the endo-functor $U\otimes -:\psh\to \psh$ is a left Quillen one. 
Moreover, the only $U$ that we will take here is $U=T$ and its wedge powers, where 
	$T$ is the Thom spectrum $T=\om(\afo/\gmm)$ (corresponding to the line bundle $\afo\to \pt$).
The corresponding "twist" operation $ T^{\wedge j}\otimes -$ 
 on $\sh$  will be denoted by $\lan j\ra$ (for any $j\ge 0$).

4. We will also use the notation $\brj$ for the derived version of this functor (as well as for its "pro-spectral" versions that we will introduce below), whereas the composition $\brj\circ [-j];\sh\to \sh$ will be denoted by $\brjj$.
Certainly, the latter functor possesses a right adjoint that will be denoted by $-_{-j}$ (essentially following 
Definition 4.3.10 of \cite{morintao}). Note here that the endo-functor $-_{-1}$ is easily seen to be isomorphic to the (corresponding restriction of the) functor $\ihom(\mathbb{G}_m,-)$  considered in Lemma 4.3.11 of ibid.

5. Our arguments related to twists are rather easy; yet they would certainly be even more simple if (the derived version) of the twist functors were invertible, We will treat "motivic" categories satisfying this condition only in section \ref{ssupl}; still a reader only interested in "twist-stable" categories (such as $\sht$) may probably ignore   all the difficulties caused by the non-invertibility of twists on $\sh$.
\end{rema}

We will need  the following properties of  the functor $\om$. 

\begin{pr}\label{psh}
\begin{enumerate}
\item\label{ish1} 
The functor $\sm\to \sh:X\mapsto \om(X)\cong \om(X_+)$ is essentially the "usual" one considered in \S4.2 of \cite{morintao}.

\item\label{isht} For any $P\in \obj \opa$ and $j\ge 0$ we have $T^{\wedge j}\otimes \pom(X/U)\cong \pom (P\brj)$ (see the beginning of \S\ref{sprs} and Remark \ref{rppsh}(3) 
 for the notation).

\item\label{ish3} 
For any 
open embeddings $Z\to Y\to X$ of smooth varieties the 
natural morphisms $\om(Y/Z) \to \om(X/Z) \to \om(X/Y)$  can be completed to a distinguished triangle, and the corresponding morphisms in $\dopsh$ yield a  cofibre sequence.

\item\label{ish4}
For any finite set of  $P_i\in \obj \opa$ 
 the projection morphisms $(\sqcup P_i)_+\to P_{i+}$ yield an isomorphism $\om(\sqcup P_i)\cong \bigoplus \om(P_i) $. 

\item\label{ish7} $\om$ is {\it homotopy invariant}, i.e., 
for any $P\in \obj \opa$ we have $\om(P)\cong \om (P\wedge \afo)$.

\item\label{ish8} $\om(-)$ converts Nisnevich distinguished squares of smooth varieties into distinguished triangles, i.e., 
for a cartesian square 
\begin{equation}\label{enisq}
\begin{CD}
 W@>{j}>>Y\\
@VV{g}V@VV{f}V \\
V@>{i}>>X
\end{CD}
\end{equation}
of smooth $k$-varieties, where  $i$ and $j$  are open 
 embeddings and $f$ is an \'etale morphism 
  whose base change  to  $X\setminus j(V)$  is an isomorphism, there is a distinguished triangle
\begin{equation}\label{enistr}
\om(W_+)\to\om(Y_+)\bigoplus \om(V_+)\stackrel{h} \to \om(X_+)\to \om(W_+)[1]\end{equation}
with $h=\omt(-_+)(i)-\omt(-_+)(f)$.

\item\label{ish5} 
Let 
$i:Z\to X$ be a closed embedding of smooth varieties,  
and denote by  $B(X,Z)$  the corresponding deformation to the normal cone variety (see the proof of Theorem 3.2.23 of \cite{movo} or \S4.1 of \cite{degdoc}). Then  the natural 
$\opa$-morphisms $X/X\setminus Z \to B(X,Z)/B(X,Z)\setminus Z\times \afo$ and $N_{X,Z}/N_{X,Z}\setminus Z\to B(X,Z)\setminus Z\times \afo$ (where $N_{X,Z}$ is the normal bundle for $i$) become isomorphisms after we apply $\om$.
 \end{enumerate}
\end{pr}
\begin{proof}
\ref{ish1}. The difference between the model category $\psh$ and the model for $\sh$ considered in \cite{morintao} is that in ibid. simplicial sheaves are considered instead of sheaves. However, the obvious functor $\dopsh\to \dosh$ is well-known to induce the corresponding Quillen equivalence (see Theorem 5.9(2) of \cite{jardbook} and Theorem 5.7 of \cite{hoveysp}). 

 Thus assertion 
 \ref{ish8} now follow from the corresponding well-known properties of $\sh$ (see \cite{morintao} and Proposition \ref{ppsh}).

\ref{isht}. These objects are naturally isomorphic already in $\dopsh$. Indeed, it suffices to note that our definition of $\wedge$ in $\opa$ is compatible with the corresponding operation on presheaves (cf.  \S2.2.2 of \cite{movo}). 

\ref{ish3}. Obvious  (note that all objects of $\dopsh$ are cofibrant).

\ref{ish4}. Obviously, it suffices to prove the statement for the union of two open embeddings (i.e., of objects of $\opa$). In this case it suffices to note that the projection morphism $(P_1\sqcup P_2)_+\to P_{1_+}$ gives a splitting of the distinguished triangle $\om (P_1)\to \om (P_1\sqcup P_2)\to \om (P_2)$ provided by assertion \ref{ish3}. 

Assertion \ref{ish7} is an easy consequence of the well-knon homotopy invariance property of the functor $X\mapsto \om(X_+)$ (see \cite{morintao}) 
  along with assertion \ref{ish3}.

\ref{ish5}. 
The corresponding $\opa$-morphisms  map actually into isomorphisms already in  $\hk$ (this is precisely Proposition 3.2.24 
of \cite{movo}). All the more  we obtain isomorphisms in $\sh$.

\end{proof}

The following observation relates twists to shifts (somehow).

\begin{rema}\label{rpsh}
1. In particular, we obtain a distinguished triangle $\om (\gmmpl)\to \om(\afo_+) \to \om (\afo/\gmm)\to \om (\gmmpl)[1]$. 
Similarly, for any $X\in \sv,\ j\ge 0$, we have 
$\om (X\lan j+1\ra )\cong \co (\om (X\brj)\to \om(X\times \gmm\lan j\ra ))[1] $.  Moreover, $\om (X\lan j+1 \ra)$ is 
obviously a retract of $\om(X\times \gmm\lan j\ra ))[1]$. 

2. Alternatively, for the proof of assertion \ref{ish4} one can note that the disjoint union 
 in $\popa$ gives the bouquet operation in $\dopsh$. 

3. Since  $P\cong P_+$ for any $P\in \opa$, there is actually no difference between $\om(X)$ and $\om(X_+)$ for $X\in \sv$. We will mostly use the latter notation  due to the reason that it was also used in most of other papers on the subject (however,  these papers did not treat "general" objects of $\opa$ and values of $\om$ on them).    
\end{rema}

Let us now recall the basic properties of Morel's homotopy $t$-structure on $\sh$ (see Theorem 4.3.4
of \cite{morintao}). We will write just $t$ for it; this is the ``main'' $t$-structure of this paper. Recall that in ibid. for any $n\in \z$ the class $\sh^{t\le n}$ was called the one of $1-n$-connected ($\afo$-local) spectra 
(note  that our $\sh^{t\le n}$ is $\sh{}_{t\ge -n}$ in the notation of ibid.; see Remark \ref{rts}(5)). 

\begin{defi}\label{dsh}

1. For any $E\in \obj \sh$ we will write   $E^n$ (resp. by $E^n_j$) for the 
 functor $\popa\opp\to \ab$ that sends 
 $\prli_i P_i$ (for $P_i\in \obj \opa$) into $\inli_i \sh(\om(P_i),E[n])$ (resp. into $\inli \sh(\om(P_i\brj),E[n+j])\cong \inli_i \sh(\om(P_i), E_{-j})=(E_{-j})^n$).  

Moreover, we will write $\tpi^n(E)$ for the presheaf of abelian groups on $\sm$ (i.e., for a functor $\sm\opp\to \ab$) that sends $X\in \sv$ into $E^n(X)\cong E^n(X_+)$.
We will use the notation $\pi^n(E)$ for the Nisnevich sheafification of this presheaf.

2. Following Definition 4.3.5 of \cite{morintao}, we will say that a Nisnevich sheaf $N$ of abelian groups on $\sm$ is {\it strictly $\afo$-invariant} if $H^n_{Nis}(-,N)$ are homotopy invariant functors $\sm^{op}\to \ab$ for all $n\in\z$, i.e., for any $X\in \sv$ we have $H^n_{Nis} (X,N)\cong  H^n_{Nis}(X\times \afo,N)$. 

We will 
 write $\shi$ for the category of strictly $\afo$-invariant (Nisnevich) sheaves. 
\end{defi}


\begin{pr}\label{psht} 

Let $E\in \obj \sh$. Then the following statements are valid.

1. For any $X\in \sv$ we have $\om(X_+)\in \sh^{t\le 0}$.

2. $E\in  \sh^{t\ge 0}$ if and only if $E^n(X)=\ns $ for any $X\in \sv$, $n<0$.

3.  $E\in  \sh^{t\le 0}$ if and only if $E^n(\spe (K)_{+})=\ns$ for all $n>0$ and for all function fields $K/k$. 

4. The endo-functor $-_{-j}$ (see Remark \ref{rppsh}(4)) is $t$-exact with respect to $t$.

5. The functor $\pi^0$ gives an equivalence of $\hrt$ with $\shi$. 
Moreover, for any $E$ and any $n\in \z$ we have $\tpi^0(E^{t=n})\cong \pi^n(E)$.

6. For any $E\in \sh^{t=0}$, $X\in \sv$, and $n\in \z$ 
there is a natural isomorphism $E^n(X)\cong H^n_{Nis} (X,\pi^0(E))$.
\end{pr}
\begin{proof}

1. Immediate from 
Lemma 4.3.3 of  \cite{morintao}. 

2. This is the just the definition of  $\sh^{t\ge 0}$ (see Definition 4.3.1(1) of ibid.).

3. Immediate from Lemmas 4.2.7  and 4.3.11 of ibid.; note here that our notation $E^*_{j}$ corresponds to $E^*_{-j}$ in the notation of ibid.

4. Immediate from Remark 4.3.12 of loc. cit. (see also Remark \ref{rppsh}(4)).

5. The first part of the assertion is precisely Lemma 4.3.7(2) of ibid.; the "moreover" part follows immediately from Theorem 4.3.4 of ibid. (cf. also the proof of Proposition \ref{pshtt}(\ref{ihrtnis} below). 

6. This fact appears to be well-known; cf. Remark 4.3.9 of ibid (whereas a related result is given by Theorem 3.7 of \cite{anancurve}). 
 Note also that the existence of an embedding $I$ of $\sh$ into the stable homotopy category of simplicial  Nisnevich sheaves $SH^{S^1}_s(k)$ along with the (localization) functor $L^{\infty}$ left adjoint to $I$  (see \S4.2 of ibid.) reduces our assertion to its $SH^{S^1}_s(k)$-analogue, whereas the latter statement easily follows from Remark 3.2.8 of ibid. (and is even more well-known).

\end{proof}

Below $\shc$ will denote the triangulated subcategory densely generated (in the sense described in \S\ref{snot}) by 
 $\{\om (X_+):\ X\in \sv\}$. 
 Note that 
  the objects of this  subcategory are  exactly the compact objects of $\sh$.

Now we recall some 
properties of semi-local schemes.

\begin{rema}\label{rsemiloc}
In this paper all semi-local schemes that we consider will be affine essentially smooth ones. 
 Such an $S$ is the semi-localization of a smooth affine variety $V/k$ at a finite collection of Zariski points. 
Obviously, if $f:V'\to V$ is a finite morphism of smooth varieties (in particular, a closed embedding) then $S\times_V V'$ is (affine essentially smooth) semi-local also. 
 It is well-known that all vector bundles over connected semi-local schemes (in our sense) are trivial.

Certainly, any semi-local scheme yields a pro-scheme (with a finite number of connected components). We will call a $\popa$-disjoint union of an arbitrary set of (affine smooth) semi-local schemes   a {\it semi-local pro-scheme}.
 \end{rema}


We will need the following important 
result (that is essentially well-known also).

\begin{pr}\label{pshinv}
Assume that $k$ is infinite; let $S$ be a 
semi-local pro-scheme, $E\in \sh^{t\le 0}$. 
Then for any $j\ge 0$ and $i>0$ we have $E^{i}_j(S)=\ns$ (see Definition \ref{dsh}(1)). 
\end{pr}
\begin{proof}
Obviously, we can assume that $S$ is connected; let $S_0$ denote its generic point.

By Proposition \ref{psh}(\ref{ish3},\ref{isht}), 
 $X\mapsto E_j^*(X)$ 
 along  with $(X,U)\mapsto E_j^*(X/U)$  yields a {\it cohomology theory with supports} in the sense of Definition 5.1.1(a) of \cite{suger}. The Nisnevich excision for 
 $\om(-_+)$ provided by part \ref{ish8} of the proposition
  gives axiom COH1 of loc. cit., whereas Proposition \ref{psh}(\ref{ish7}) 
is precisely 
 axiom COH3 of (\S5.3 of) ibid. Hence we can apply Theorem 6.2.1 of ibid. to obtain an injection $E^{i}_j(S)\to E^{i}_j(S_{0})$. To conclude the proof it suffices to apply Proposition \ref{psht}(3,4). 

\end{proof}

\section{Weight structures: reminder}
\label{sws}

In this section we recall the  general formalism of weight structures.

In \S\ref{swr} we recall some basics of the theory along with the properties of  {\it cocompactly cogenerated} weight structures (here we follow \cite{paucomp} and \cite{bpure}). 

 In \S\ref{swfs} we recall rather more complicated parts of the general theory --- weight complexes and weight spectral sequences. We relate the latter to orthogonal $t$-structures in \S\ref{sort}. These matters can be ignored by the readers not interested in the study of (generalized) coniveau spectral sequences.

\subsection{
Basic  definitions and properties; cocompactly cogenerated weight structures}\label{swr}

\begin{defi}\label{dwstr}

 A couple of subclasses $\cu_{w\le 0}$ and $ \cu_{w\ge 0}\subset\obj \cu$ 
will be said to define a {\it weight structure} $w$ for a triangulated category  $\cu$ if 
they  satisfy the following conditions.

(i) $\cu_{w\le 0}$ and $\cu_{w\ge 0}$ are 
retraction-closed in $\cu$
(i.e., contain all $\cu$-retracts of their elements).

(ii) {\bf Semi-invariance with respect to translations.}

$\cu_{w\le 0}\subset \cu_{w\le 0}[1]$ and $\cu_{w\ge 0}[1]\subset
\cu_{w\ge 0}$.

(iii) {\bf Orthogonality.}

$\cu_{w\le 0}\perp \cu_{w\ge 0}[1]$.

(iv) {\bf Weight decompositions}.

 For any $M\in\obj \cu$ there
exists a distinguished triangle
\begin{equation}\label{wd}
X\to M\to Y
{\to} X[1] \end{equation} 
such that $X\in \cu_{w\le 0} $ and $ Y\in \cu_{w\ge 0}[1]$.
\end{defi}

We will also need the following definitions.

\begin{defi}\label{dwso}

Let $i,j\in \z$; assume that $\cu$ is endowed with a weight structure $w$.

\begin{enumerate}
\item\label{id1} The full subcategory  $\hw\subset \cu$ whose object class is $\cu_{w=0}=\cu_{w\ge 0}\cap \cu_{w\le 0}$ 
 is called the {\it heart} of  $w$.

\item\label{id2} $\cu_{w\ge i}$ (resp. $\cu_{w\le i}$,  $\cu_{w= i}$) will denote $\cu_{w\ge 0}[i]$ (resp. $\cu_{w\le 0}[i]$,  $\cu_{w= 0}[i]$).

\item\label{id3} $\cu_{[i,j]}$  denotes $\cu_{w\ge i}\cap \cu_{w\le j}$; so, this class  equals $\ns$ if $i>j$.

$\cu^b\subset \cu$ will be the category whose object class is $\cup_{i,j\in \z}\cu_{[i,j]}$.

\item\label{idndeg} $w$ will be called {\it right (resp. left) non-degenerate} if $\cap_{l\in z} \cu_{w\le l}=\ns$ (resp. $\cap_{l\in \z} \cu_{w\ge l}=\ns$).

\item\label{id6} Let $\bu$ be a 
full additive subcategory of a triangulated category $\cu$.

We will say that $\bu$ is {\it negative} (in $\cu$) if
 $\obj \bu\perp (\cup_{i>0}\obj (\bu[i]))$.

\item\label{idwpt}
We will call a Postnikov tower for $M$ (see Definition \ref{dpoto})  a {\it weight Postnikov tower} if
all $Y_j$ are some choices for $w_{\le j}M$. 
In this case we will call the corresponding complex whose terms are $M^p$
(see Remark \ref{rwcomp}(1))
a {\it weight complex} for $M$.


\item\label{idwe}
Let  $\cu'$ be a triangulated category endowed with a weight structure $w'$; 
let $F:\cu\to \cu'$ be an exact functor.

We will say that $F$ is {\it left weight-exact} (with respect to $w,w'$) if it maps $\cu_{w\le 0}$ to $\cu'_{w'\le 0}$; it will be called {\it right weight-exact} if it maps $\cu_{w\ge 0}$ to $\cu'_{w'\ge 0}$. $F$ is called {\it weight-exact} if it is both left and right weight-exact.
\end{enumerate}
\end{defi}

\begin{rema}\label{rstws}

1. A weight decomposition of  $M\in \obj\cu$ is (almost) never canonical; still we will sometimes write $(w_{\le 0}M,w_{\ge 1}M)$ for    the couple $(X,Y)$ coming from (any choice of)  (\ref{wd}). 
For an $l\in \z$ we will write $w_{\le l}M$ (resp. $w_{\ge l}M$) for a choice of  $w_{\le 0}(M[-l])[l]$ (resp. of $w_{\ge 1}(M[1-l])[l-1]$).

Moreover, when we will write arrows of the type $w_{\le l}M\to M$ or $M\to w_{\ge l+1}M$ we will always assume that they come from a weight decomposition 
 of $M[-l]$.

2. In the current paper we use the "homological convention" for weight structures; 
it was previously used in 
\cite{wildic},    
\cite{bpure}, \cite{bsnew}, and \cite{bokum}, whereas in \cite{bws} and in \cite{bger} the "cohomological convention" was used.\footnote{Recall also that 
D. Pauksztello has
introduced weight structures independently (see \cite{paucomp}); he called them
co-t-structures. }\ In the latter convention 
the roles of $\cu_{w\le 0}$ and $\cu_{w\ge 0}$ are interchanged, i.e., one considers   $\cu^{w\le 0}=\cu_{w\ge 0}$ and $\cu^{w\ge 0}=\cu_{w\le 0}$. So,  a complex $M\in \obj K(B)$ whose only non-zero term is the fifth one 
 has weight $-5$ in the homological convention, and has weight $5$ in the cohomological convention. Thus the conventions differ by "signs of weights"; 
 $K(B)_{[i,j]}$ is the class of retracts of complexes concentrated in degrees $[-j,-i]$. 

3. In \cite{bws} the axioms of a weight structure also 
required $\cu_{w\le 0}$ and $\cu_{w\ge 0}$ to be additive. 
 Yet this additional restriction is easily seen to follow from the remaining axioms; see Remark 1.2.3(4) of \cite{bonspkar} (or Proposition \ref{pbw}(\ref{iextw}) below).

  
4. Moreover, in the current paper we shift the numeration for $Y_i$ (in the definition of a weight Postnikov tower) by $[1]$ if compared with \cite{bger}.  
 \end{rema}

Now  we recall some basic 
properties of weight structures. 

\begin{pr} \label{pbw}
Let $\cu$ be a triangulated category endowed with a weight structure $w$, $M,M'\in \obj \cu$, $i\in \z$. Then the following statements are valid.

\begin{enumerate}
\item \label{idual}
The axiomatics of weight structures is self-dual, i.e., for $\du=\cu^{op}$ (so $\obj\cu=\obj\du$) there exists the (opposite)  weight
structure $w'$ for which $\du_{w'\le 0}=\cu_{w\ge 0}$ and $\du_{w'\ge 0}=\cu_{w\le 0}$.

\item\label{iextw}  $\cu_{w\le i}$, $\cu_{w\ge i}$, and $\cu_{w=i}$ are Karoubi-closed and extension-closed in $\cu$ (and so, additive). 

Besides, if $M\in \cu_{w\le 0}$, then $w_{\ge 0}M\in \cu_{w=0}$ (for any choice of $w_{\ge 0}M$).

\item \label{isump}
If  we have a $\cu$-distinguished triangle $A\to B\to C$ with $B\in \cu_{w=0}$, $C\in \cu_{w\ge 1}$, then $A\cong B\bigoplus C[-1]$.

    \item\label{iwpost} For any choice of $w_{\le j}M$ for $j\in \z$ there exists  a weight Postnikov tower for $M$.
Moreover, for any weight Postnikov tower we have $\co(Y_i\to M)\in \cu_{w\ge i+1}$ and    $M^i\in \cu_{w=0}$.

\item\label{iwpostc}
    Conversely, any bounded Postnikov tower (for $M$) with $M^j\in \cu_{w=0}$ for all $j\in \z$ is a weight Postnikov tower for it.

 \item\label{iort} $\cu_{w\ge i}=(\cu_{w\le i-1})^{\perp}$ and  $\cu_{w\le i}={}^{\perp}\cu_{w\ge i+1}$.
 
\item\label{igenlm} For any $l\le m \in \z$ the class $\cu_{[l,m]}$ is the smallest Karoubi-closed extension-stable subclass of $\obj\cu$
containing $\cup_{l\le j\le m}\cu_{w=j}$.

\item\label{ilrwe} Let $\du$ be a  triangulated category endowed with a weight structure $v$; let $F\colon\cu \leftrightarrows \du:\!G$ be exact 
adjoint functors. 
 Then $F$ is left weight-exact if and only if $G$ is right weight-exact.

\item\label{ipostn} 
 Suppose that a morphism $g\in \cu(M,M')$ is compatible with an isomorphism $w_{\le i}M\to
w_{\le i}M'$. 
Then $\co(g)\in \cu_{w\ge i+1}$. 

Moreover, if $i=0$ and $M'\in \cu_{w\le 0}$ then $g$ yields a projection of $M$ onto $M'$ (i.e., $M'$ is a retract of $M$ via $g$).

\end{enumerate}
\end{pr}

\begin{proof} 

Assertions \ref{idual}--\ref{igenlm} 
 are essentially contained in Theorem 2.2.1 of \cite{bger} (whereas their proofs relied on  \cite{bws}; pay attention to Remark \ref{rstws}(2)!). 
Assertion \ref{ilrwe} is given by Remark 2.1.5(3) of \cite{bpure}.

\ref{ipostn}. Note first that the octahedral axiom of triangulated categories yields that $\co(g)\cong \co(w_{\ge i+1}M\to
w_{\ge i+1}M')$. Hence the first part of the assertion follows from the extension-closedness of $\cu_{w\ge i+1}$ (see assertion \ref{iextw}).

Next, if $i=0$ and $M'\in \cu_{w\le 0}$ then   $\co(g)\in \cu_{w\ge 1}$ (as we have just proved). Hence $M'\perp \co(g)$ by the orthogonality axiom of weight structures, and the splitting of the corresponding distinguished triangle yields the "moreover" part of the assertion.
\end{proof}

\begin{rema}\label{rwgen} 

  A very useful statement (that was 
 applied in several papers) is that any negative 
densely generating (see \S\ref{snot} for the definition) 
subcategory  $N\subset \cu$ 
  yields a {\it bounded} weight structure $w$ on $\cu$ (i.e., $\cu^b=\cu$)  such that $ \hw\cong \kar_{\cu}(N)$ (see Proposition 5.2.2 of \cite{bws} and 
	 Corollary 2.1.2 of \cite{bonspkar}). Yet this is rather a tool for constructing (bounded) weight structures on "small" triangulated categories; so we will now recall an alternative existence statement.
\end{rema}

Recalling some results of  \cite{bpure} (along with \cite{paucomp}) we obtain 
a collection of properties of {\it cocompactly cogenerated} weight structures. We start from introducing the following simple definition.

\begin{defi}\label{dcosm}
 Let $\cu$ be a triangulated category closed with respect to arbitrary small products. Then we will say that a weight structure $w$ is {\it cosmashing} whenever $\cu_{w\ge 0}$ is closed with respect to (small) products.
\end{defi}

\begin{rema}\label{rcosm}
1. 
Proposition \ref{pbw}(\ref{iort}) implies immediately that $\cu_{w\ge 0}$ is closed with respect to all $\cu$-products. Hence for any cosmashing $w$ the classes
 $\cu_{w\le j}$, $\cu_{w\ge i}$, and $\cu_{[i,j]}$, as well as the category $\hw$ are closed with respect to $\cu$-products  for arbitrary $i,j\in \z$;  
 see Proposition 2.5.1(1,2) of \cite{bpure} for the dual to this statement.

Moreover,  (small) products of weight decompositions in $\cupr$ are weight decompositions;  and products of weight Postnikov towers 
are weight Postnikov towers 
according to (the categorical dual to) parts 3 and 4 of loc. cit.

2. The definition of cosmashing weight structures above is  (not new and) coincides with the one given in Remark 2.1.5(1) of \cite{bpure}. Note that in ibid. also the dual notion of {\it smashing} weight structures (and more generally, smashing {\it torsion pairs}) was considered (following preceding papers of other authors). 
However, will not meet smashing weight structures in the current paper.
\end{rema}

\begin{theo}\label{tnews}

Let $\cu$ be triangulated category that is closed with respect to  small products;
let $C\subset\obj \cu$ be a 
 set of cocompact objects such that $C[1]\subset C$. 
Then for the classes $C_1={}^{\perp}(C[1])$ and 
$C_2=C_1^\perp[1]$ 
are valid.

I. $w=(C_1,C_2)$ is a  cosmashing weight structure on $\cu$, and $C\subset C_2$.

II. Denote by $\cu'$ the smallest full triangulated subcategory of $\cu$ 
containing $C$ and closed with respect to products;  $\cuperp$ is the subcategory whose objects are ${}^\perp \lan C\ra$. 

1. 
$\cuperp$ is triangulated and closed with respect to all (small) products.

2. $C_1'=C_1\cap \obj \cupr$ and $C_2\subset \obj \cupr$ yield a cosmashing weight structure $w_{\cupr}$ on $\cupr$.

3. 
 $w_{\cupr}$ is right non-degenerate. 

 4. $C_1$ is the extension-closure of $C_1'\cup \obj \cuperp$ in $\cu$. 

5. The heart of 
$w_{\cu'}$ equals $\hw$. 
 
III. For each $c\in C$ 
choose a weight complex $(c^i)$ (with respect to $w_{\cupr}$). Then $\hw_{\cu'}$ is equivalent to the Karoubi envelope of the category of all (small) products of $c^i$ for $c$ running through  all objects $C$, $i\in \z$. 

IV. There exists a left adjoint $L$ to the embedding $\cu'\to \cu$. It is identical on $\cu'$, respects products, and weight-exact (with respect to $w$ and $w_{\cupr}$, respectively). 

V. Let $F:\cu'\to \du$ be an exact functor respecting products; assume that $w_{\du}$ is a weight structure for $\du$ such that $\du_{w_{\du}\le 0}=\perpp F(C)$.
Then $F$ is right weight-exact (with respect  to  $w_{\cupr}$ and $w_{\du}$).

\end{theo}
\begin{proof}
I. This statement is exactly the dual to  the main result of \cite{paucomp}; cf. also Corollary 5.4.1(7) of \cite{bpure}.

II.1. Obvious.

2,3,4. Immediate from   Corollary 5.4.1(7,8) of loc. cit. 

5. 
Obvious from construction; see 
 part 8 of loc. cit.  

III. Immediate from part 10 of loc. cit. 

IV. See parts 3 and 8 of loc. cit.

V. Immediate from Remark 2.1.5(4) of ibid.

\end{proof}

\begin{rema}\label{rnew}
1. Obviously, our description of $w$ implies that $C_2=\cu_{w\ge 0}$ is the smallest subclass of $\obj\cu$ that contains $C$  and can be completed to a weight structure. So, we say that $w$ is {\it cogenerated} by $C$ if $w$ equals $(C_1,C_2)$ (for some $(\cu,C)$); cf. Proposition 3.3.6 of \cite{bpure}.

Moreover, in the setting of our theorem there exist some more descriptions of $C_2=\cu_{w\ge 0}=\cu'_{w_{\cupr}\ge 0}$; see Theorem 4.2.1(1) of ibid. Instead of applying them, below we will just prove that the objects that we are interested in belong to $C_2$ "directly".

2. In this paper it will be somewhat more convenient for us to "stay inside" the corresponding category $\cu'$ (note here that $\cu'$ is "more reasonable" since it is {\it cogenerated} by a set $C$ of cocompact objects in contrast to $\cu$), and essentially the only place where we will 
 possibly need elements  of $\obj\cu\setminus\obj\cu'$ is \S\ref{sprovar}  below (and we will "put" those into $\cu'$ by means of $L$).
 However, one may say that the "difference" between $\cu$ and $\cu'$ is the category $\cuperp$, and the latter is easily seen to be "invisible" by means of  matters studied in this paper. Hence it is 
 not that necessary to introduce $\cu'$. 
\end{rema}

\subsection{On weight complexes,  filtrations, and spectral sequences} 
\label{swfs}

Let us recall some more complicated notions and constructions related to weight structures.

Till the end of this section $\cu$ will be endowed with a weight structure $w$. 

We start from the notion of a {\it weight  complex}. The term originates from \cite{gs}; still weight complexes for  objects of triangulated categories (endowed with weight structures)  were only introduced in \cite{bws}. Moreover, the functoriality of weight complexes was treated carefully in (\S2.2 of) \cite{bpure} only. However, for the purposes of the current paper we prefer to describe certain consequences of this part of the theory instead. 

\begin{defi}\label{dwc}
Let $M,M'\in \obj \cu$, $g\in \cu(M,M')$; choose some weight Postnikov towers  $Po_M=(M,Y_j=w_{\le j}M,M^j=M_{-j}[j]:\ j\in \z)$  and  $Po_{M'}=(M', Y'_j=w_{\le j}M',M'^j=M'_{-j}[j])$ (along with morphisms connecting these objects according to  Definition \ref{dwso}(\ref{idwpt}) and the related  Definition \ref{dpoto}) for $M$ and $M'$, respectively.

1.  We will call the corresponding complex whose terms are $M^p$ (see Remark \ref{rwcomp}) a (choice of a) {\it weight complex} for $M$ and denote it by $t(M)$. 

2. We will call a collection of morphisms $w_{\le j}M\to w_{\le j}M'$ and $g^i:M^j\to M'^j$ a morphism 
 $Po_M\to Po_{M'}$ compatible with $g$ if all the obvious squares commute; see \S2.2 of \cite{bpure} for more detail. 
\end{defi}

In the notation of the previous definition  the following statements are valid.

Now we recall some properties of weight complexes; some of them follow immediately from definitions.

\begin{pr}\label{pwc}
\begin{enumerate}
\item\label{iwc1}
The arrows $g^i$ give a morphism of complexes $t(M)\to t(M')$; we will say that this morphism is {\it compatible with $g$} 
 and denote it by $t(g)$.

\item\label{iwc2} Any choice of $\{w_{\le j}M:\ j\in \z\}$ can be completed to $Po_M$, and  for any (fixed) $Po_M$ and $Po_{M'}$ there exist a morphism $Po_M\to Po_{M'}$ compatible with $g$.\footnote{Moreover,  $Po_M$ is unique up to a non-canonical isomorphism.}

\item\label{iwc3} If $g$ is an isomorphism then (any choice of) $t(g)$ is a homotopy equivalence (i.e., it is a $K(\hw)$-isomorphism). 

\item\label{iwcoh} Let $A:\hw^{op}\to\au$ ($\au$ is an arbitrary abelian category) be an additive functor. Then the correspondence that maps $M$ into the zeroth homology of the complex $A(M^{-*})$ that will be denoted by $H_A(M)$,  and that sends $g$ into the  morphism $H_A(M)\to H_A(M')$ induced by $t(g)$, gives a cohomological functor from $\cu$ into $ \au$ that is well-defined up to a canonical isomorphism.

We will call cohomological functors obtained this way {\it pure} ones (following Remark 2.4.5(1) of \cite{bpure}; see also part 5 of that remark for a description of the relation between pure functors and Deligne's purity).

\item\label{iwcohcoprod} 
Moreover, for $A$ as above assume that it  converts $\hw$-products into $\au$-coproducts, $\au$ is an AB4 category, and $w$ is a cosmashing weight structure (see Definition \ref{dcosm}). Then $H_A$ converts products into coproducts also.

 Moreover, the correspondence $A\mapsto H_A$ gives an equivalence between the category of  those additive functors $\hw^{op}\to\au$ that respect coproducts and those cohomological functors  from $\cu$ into $ \au$ that respect $\cu\opp$-coproducts and annihilate both $\cu_{w\le -1}$ and $\cu_{w\ge 1}$.

\item\label{iwc4}  Assume that $w_{\le i}M=0$ for $i<0$. Then the (corresponding) complex $t(M)$ is concentrated in non-positive degrees.

\item\label{isplwc} 
Let $M\in \cu_{w=0}$ and assume that $w_{\le i}M=0$ for $i<0$ (for our choice of $Po_M$).
Then there exists some  $N^j\in \cu_{w=0},\ j\le 0$, such that $t(M)$  is $C(\hw)$-isomorphic to $M\bigoplus (\bigoplus_{j\le 0}(N^j\stackrel{\id_{N^j}}{\to}N^j)[-j])$ (i.e., $N^j$ is put in degrees $j-1$ and $j$ if $j<0$, and we consider $M$ in the right hand side as an object of $\hw$). 

\end{enumerate}

\end{pr}
\begin{proof}

\ref{iwc1}--\ref{iwc3}. Immediate from Proposition  2.2.4(1,2,4,6) of \cite{bpure} (that relies on \cite[\S3]{bws}; cf. also Proposition \ref{pbw}(\ref{iwpost})).

\ref{iwcoh}. 
This is the dual to Proposition 2.3.1 of ibid. 

\ref{iwcohcoprod}. Dualize Proposition 2.5.1(6) of ibid. 

\ref{iwc4}. Obvious.

\ref{isplwc}. Take $M'=M$, $g=\id$, and consider the weight Postnikov tower for $M'$ with  $w_{\le j}M'= M$ for $j\ge 0$ and $w_{\le j}M'=0$ otherwise. Then we obviously have $t(M')=\dots 0\to M\to 0\to\dots $ ($M$ is in degree $0$), whereas $t(g)$ is a $K(\hw)$-isomorphism according to assertion \ref{iwc3}. Hence the cone complex  $C=\dots M^{-2}\to M^{-1}\to M^{0}\to M\to 0\to\dots$ (here we apply the previous assertion) is zero in $K(\hw)$. Thus  
it suffices to verify by induction the following fact: if a bounded above complex $P=\dots\to P^{j-2}\to P^{j-1} \stackrel{d^{i-1}}\to P^{j}\to 0\to\dots$ is zero in $K(\hw)$, then $d^{j-1}$ is a projection of  $P^{j-1}$ onto a direct summand in $\hw$. Now, 
  $d^{j-1}$ is split by the corresponding component of (any) contracting homotopy for $P$. Hence $\co d^{j-1}[-1]$ is the complement 
	of $P^j$ to $P^{j-1}$ in $\cu$. Since $\hw$ is Karoubi-closed in $\cu$, we obtain the result.
\end{proof}

Now we pass to weight filtrations and weight spectral sequences. Till the end of the section $H$ will be a cohomological functor from $\cu$ into $\au$. For any $r\in \z$ denote $H\circ [-r]$ by $H^r$.  

\begin{pr}\label{pwss}

Let $M$ be an object of $\cu$.

1. For any $m\in \z$ the object $(W^{m}H)(M)=\imm (H(w_{\ge m}M)\to H(M))= \ke (H(M)\to H(w_{\le m-1}M))$ does not depend on the choice of $w_{\ge m}M$; it is $\cu^{op}$-functorial in $M$.

We call the filtration of $H(M)$ by $(W^{m}H)(M)$ its {\it weight} filtration.

2. 
 There is a spectral sequence $T(H,M)$ with $E_1^{pq}=H^{q}(M^{-p})$. It comes from (any possible) weight Postnikov tower of $M$; so the boundary of $E_1$ is obtained by applying $H^*$ to the corresponding choice of a weight complex for it. 

3. $T$ is (canonical and) naturally functorial in $H$ (with respect to exact functors between the target abelian categories) and in $M$ starting from $E_2$ (and we will use the notation $T^{\ge 2}(H,M)$ for  this "part" of $T(H,M)$). Moreover, $E_2^{pq}(H,M)$ is $\cu^{op}$-functorially isomorphic to $H_{H^q}^p(M)$ (see the notation above) for all $p,q\in \z$. 

4. $T(H,M)$ converges to 
$H^*(M)$ either if (i) $M$ is $w$-bounded 

or if (ii) $H$ kills both $\cu_{w\le -i}$ and 
$\cu_{w\ge i}$ for some (large enough) $i\in \z$. Furthermore, the  filtration step given by ($E_{\infty}^{l,m-l}:$ $l\ge k$)
 on $H^{m}(X)$ equals $(W^k H^{m})(M)$ (for any $k,m\in \z$). 
\end{pr}

\begin{proof}
1. This is (a particular case of) Proposition 2.1.2(2) of ibid.

2,3.4. This is (most of) Theorem 2.4.2 of ibid.
\end{proof}

\begin{rema}\label{rintel}

1. Actually, the object $(W^{m}H)(M)=\imm (H(w_{\ge m}M)\to H(M))$ does not depend on the choice of $w_{\ge m}M$ and is functorial in $M$ for any contravariant $H$ from $\cu$ into $\au$; see Proposition 2.1.2(2) of \cite{bws}.

2. Certainly, weight spectral sequences can also be constructed for a homological $H$; see Theorem 2.3.2 of ibid. 

3. Interesting properties of weight spectral sequences include its description in terms of {\it virtual $t$-truncations} (see \S\S2.3--2.4 of \cite{bger}) and the "continuity" Remark 5.1.4(II.3--4) of \cite{bpure}.

\end{rema}

Let us also consider two particular cases of our proposition.

\begin{coro}\label{cwss}
For $H$ and $M$ as above the following statements are valid.

1. Assume that $M\in \cu_{w=0}$. Then $T(H,M)$ converges to $H^*(M)$  and degenerates at $E_2$. Moreover,  $E_2^{pq}$  vanishes for $p\neq 0$ and
 equals $H^{q}(M)$ otherwise. 

2. More generally, if $M\in \cu_{[0,d]}$ for some $d\ge 0$ then  $T(H,M)$ converges to $H^*(M)$ and  $E_2^{pq}$  vanishes for $p< 0$ and $p>d$.

3. Assume that $H$ annihilates both $\cu_{w\le -1}$ and  $\cu_{w\ge 1}$.  Then $T(H,M)$ converges to $H^*(M)$ and degenerates at $E_2$ also; for all $p\in \z$ we have  
$E_n^{pq}=0$ for all $q\neq 0$ and $n\ge 1$,  whereas  $E_2^{p0}\cong H^p(M)$. 

\end{coro}
\begin{proof} All the assertions follow from  
  Proposition \ref{pwss} easily. We note that  it is convenient to take the weight complex  $t(M)$  consisting of $M$ put in degree $0$ for the proof of assertion 1 and take $t(M)$ concentrated in degrees $[-d,0]$ for assertion 2, and  recall that a spectral sequence whose $E_2$-terms are concentrated in a single row or a single column degenerates at $E_2$ and the non-zero $E_2$-terms are isomorphic to the corresponding limit ones. 
\end{proof}

\subsection{The relation to orthogonal $t$-structures}\label{sort}

Let $\cu,\du$  be triangulated categories.
We consider certain pairings 
$\cu^{op}\times \du\to \ab$ (though in most of the statements $\ab$ can be replaced with a more general abelian category $\au$; see \cite[\S2.5--6]{bger} and \cite[\S5.2]{bpure}).
We note that it is not really important to grasp the technical definition of niceness (given below) since we will have to mention this notion only once.  

\begin{defi}\label{ddual}

1. We will call a (covariant) bi-functor
  $\Phi:\cu^{op}\times\du\to \ab$ a {\it duality} if  it is bi-additive, homological with respect
to both arguments, and is equipped with a (bi)natural bi-additive transformation
$\Phi(A,Y)\cong \Phi (A[1],Y[1])$.


2. We will say that $\Phi$ is {\it nice} if for any distinguished
triangles $A\stackrel{l}{\to} B \stackrel{m}{\to} C\stackrel{n}{\to} A[1]$ in $\cu$ and $X\stackrel{f}{\to} Y\stackrel{g}{\to} Z\stackrel{h}{\to}X[1]$ in $\du$ 
we have the following:  the natural homomorphism $p$:
$$ \begin{gathered} 
 \ke (\Phi(A,X)\bigoplus \Phi(B,Y) \bigoplus \Phi(C,Z))\\
\xrightarrow{\begin{
pmatrix}
\Phi(A,-)(f) & -\Phi(-,Y)(l) &0  \\
0& g(B) &-\Phi(-,Z)(m)  \\
- \Phi(-,X)([-1](n)) & 0 &\Phi(C,-)(h)
\end{
pmatrix}}
\\ (\Phi(A,Y) \bigoplus \Phi(B,Z) \bigoplus \Phi(C[-1],X))
 \stackrel{p}{\to} \ke ((\Phi(A,X)\bigoplus \Phi(B,Y))\\ \xrightarrow{\Phi(A,-)(f)\oplus - \Phi(-,Y)(l)}
 \Phi(A,Y)) 
 \end{gathered}$$
is epimorphic.

 3. Suppose that  $\cu$ is endowed with a weight structure $w$,  $\du$ is endowed with a $t$-structure $t$. Then we will say that $w$
 is (left) {\it orthogonal} to $t$ with respect to $\Phi$  if the following  {\it orthogonality condition} is fulfilled:
 $\Phi (X,Y)=\ns$ if $X\in \cu_{w\ge 0}$ and $Y\in \du^{t \ge 1}$ or if $X\in \cu_{w\le 0}$ and $Y\in \du^{t \le -1}$.

\end{defi}

All "natural" dualities are nice; see  Proposition 5.2.5 of \cite{bpure} or Proposition 2.5.6(3) of  \cite{bger}.

Now let us  describe the relation of weight spectral sequences to orthogonal structures.

\begin{pr}\label{pdual}
Assume that $w$ on $\cu$ and $t$ on $\du$ are orthogonal with respect
 to a nice duality $\Phi$;  $M\in \obj \cu$, $Y\in \obj \du$. 

 Consider  the spectral sequence $S$  coming from
    the following exact couple: $D_2^{pq}(S)=\Phi(M,Y^{t\ge q}[p-1])$, 
$E_2^{pq}(S)=\Phi(M,Y^{t=q}[p])$ (we start $S$ from $E_2$). 

This spectral sequence is naturally isomorphic to $T^{\ge 2}(H,M)$ for $H:\ N\mapsto \Phi(N,Y)$. 

\end{pr}
\begin{proof}
This is (a part of) Theorem 2.6.1 of  \cite{bger}. 
\end{proof}

\section{$\afo$-pro-spectra and the Gersten weight structure} 
\label{scomot}

In this section we construct the Gersten weight structure on  triangulated categories $\gd\subset \gdb$ of $\afo$-prospectra; this weight structure along with its variations (corresponding to  motivic categories distinct from $\sh$) is the main subject of this paper.

In \S\ref{comot} we embed $\sh$ into a certain 
 triangulated 
 category $\gdb$; we will call its objects {\it $\afo$-pro-spectra}.

In \S\ref{scgersten} we construct certain Gysin distinguished triangles and Postnikov towers for the pro-spectra of pro-schemes.


 Next we construct (in \S\ref{scwger}) 
  certain {\it Gersten weight structures}  on  $\gdb\supset \gd$. $\gd$ is the subcategory of $\gdb$ cogenerated by $\shc$; one may say that it is the largest subcategory of $\gdb$ "detected by the Gersten weight structure".
The 
(common) heart of these Gersten weight structures is "cogenerated" (via products and direct summands) by certain twists of pro-spectra of functions fields, whereas the pro-spectra of arbitrary pro-schemes belong to $\gd_{w\ge 0}$. It follows immediately that the Postnikov tower $Po_X$
  provided by 
  Proposition \ref{post} is a   weight Postnikov tower with respect to $w$. 
  
  Moreover, if $k$ is infinite then $\hw$ also contains the  pro-spectra of  semi-local pro-schemes over $k$; we call schemes satisfying the latter property {\it $\afo$-points}.
  In \S\ref{stds}   we describe several equivalent definition of this notion. In particular,  $S$ is an $\afo$-point if and only if its higher cohomology with coefficients in all strictly homotopy invariant sheaves vanish; this is also equivalent to  $\omd(S_+)$ being a retract of $\omd(S_{0+})$ for $S_0$ being any open dense pro-scheme of $S$.
Moreover, localizations of (essentially smooth affine) semi-local schemes at parameters are $\afo$-points also if $k$ is infinite. 

\subsection {On  $\afo$-pro-spectra} 
\label{comot}
We recall the construction of a certain (stable) model category $\gdp$ (that relies on the results of \cite{tmodel}); this construction was described in detail in \S5.5 of \cite{bpure}.

As a category $\gdp$ is just the category of 
pro-objects of $\psh$ (see \S5 of ibid. for the general construction, whereas $\psh$ is the category mentioned in Proposition \ref{ppsh}). We endow it with the {\it strict} model structure; see \S5.1 of \cite{tmodel} (so, weak equivalences and cofibrations are {\it essential levelwise} weak equivalence and cofibrations of pro-objects, respectively).\footnote{Note that in \cite{tmodel}  actually "categorical" versions of pro-objects were treated, i.e., instead of limits over (filtered) projective systems limits over filtered small categories were considered. However, it is easily seen that this (formal) distinction makes no difference for the results of the current paper; cf. \S\ref{sprs}.} 

Now we list some basic properties of $\gdp$ and its homotopy category $\gdb$ (that are particular cases of the corresponding general statements). The pro-object corresponding to a projective system $X_i$ will be denoted by $(X_i)$. Note that $(X_i)$ is exactly the (inverse) limit of the system $X_i$ in $\gdp$ (by the definition of morphisms in this category).

\begin{pr}\label{pgdb}

Let  $X_i,Y_i,Z_i\ i\in I$, be projective systems in $\psh$. Then the following statements are valid.

1. $\gdp$ is a proper stable simplicial model category. 

2.  
If  some morphisms $X_i\to Y_i$ for all $i\in I$ 
 yield a compatible system of cofibrations (resp. of weak equivalences; resp. some couples of morphisms 
$X_i\to Y_i\to Z_i$ yield a compatible system of 
 homotopy cofibre sequences) then the corresponding  morphism $(X_i)\to (Y_i)$ is a cofibration also (resp. a weak equivalence; resp. the couple of morphisms $(X_i)\to (Y_i)\to (Z_i)$ is a 
 homotopy cofibre sequence). 
 
3. The natural embedding $c:\psh\to \gdp$ is a 
left  Quillen functor; it also respects weak equivalences and fibrations.

4. For any $M\in \obj \psh$ 
we have $\gdb((X_i),c(M))\cong \inli \sh(X_i,M)$.
In particular, the 
 functor $\ho(c):\sh\to \gdb$ is a full embedding.

We will often 
write $\sh\subset \gdb$ 
(without mentioning the embedding functor). 

5. More generally, for any projective system 
$\{N_i\}$ in $ \gdp$ and any $M\in \obj \psh$ the inverse limit of $N_i$ exists in $\gdp$ and
we have $\gdb(\prli M_i,c(N))\cong \inli_{i\in I} \cu (M_i,c(N))$.

6. All objects of $\ho(c)(\sh)$  
are cocompact in $\gdb$. 

7. The class $\ho(c)(\obj \sh)$  cogenerates $\gdb$.

8. For any $j\ge 0$ the correspondence $(X_i)\to (X_i\brj)$ gives a left Quillen endofunctor $\brj_{\gdp}$ on $\gdp$,
and its homotopy functor $-\brj_{\gdb}$ respects products.



\end{pr}
\begin{proof}
Since $\psh$ is a proper stable simplicial 
 model category,  
 we can apply Proposition 5.5.2 of \cite{bpure} to obtain assertions 1--7. 

Assertion 8 is a particular case 
  of Proposition \ref{proobj}(\ref{ifunprotr}) below (since $-\brj$ is a left Quillen endofunctor; see Remark \ref{rppsh}(3)).
\end{proof}

\begin{pr}\label{pprop}

\begin{enumerate}

\item\label{ifun}
There exists a commutative diagram of  functors 
\begin{equation}\label{efun}
\begin{CD}
\opa@>{\pom}>>\psh  @>{\hosh}>> \sh \\
@VV{}V@VV{c}V@VV{\ho(c)}V \\
\popa@>{\ppom}>>\gdp  @>{\hogd}>> \gdb 
\end{CD}\end{equation}
with $\pom$ being the functor defined after Proposition \ref{ppsh} and $\ppom((X_i))=\prli_i (c\circ \pom(X_i))$ for $(X_i)\in \obj \popa$ being a projective system in 
$\opa$. Moreover, for any $j\ge 0$ the functor $c$ "commutes with $-\brj$" (in the obvious sense).

We will 
use the symbol $\omd$ 
 for the corresponding composite functors $\opa\to \gdb$ and $\popa\to \gdb$.

\item\label{ilim} $\gdp$ is closed with respect to all small filtered limits; $\gdb$ is closed with respect to all small products.

\item\label{icoclim} For any projective system 
$M_i\in \obj \gdp$ and $N\in \obj \sh$ we have $\gdb(\hogd(\prli 
 M_i),c(N))\cong \inli_{i\in I} \gdb(\hogd(M_i),c(N))$. 

In particular, for any pro-scheme $U=\prli U_i$ (for $U_i\in \sv$) 
 we have $\gdb(\omd(U_+),c(N))\cong \inli_{i\in I} \gdb(\omd(U_{i+}),c(N))$. 

\item\label{itriang} 
For 
compatible system of open embeddings  $Z_i\to Y_i\to X_i$ the natural morphisms 
$$\ppom (\prli (Y_i/ Z_i))\to \ppom (\prli (X_i/ Z_i))\to \ppom((\prli (X_i/ Y_i ))$$
yield a cofibre sequence; thus their images in $\gdb$ extend to a distinguished triangle.

\end{enumerate}

\end{pr}
\begin{proof}
Assertions \ref{ifun} and \ref{ilim}  are given by our constructions. 
Assertion \ref{icoclim}  follows from assertion \ref{ifun} 
 combined with Proposition \ref{pgdb}(4).
 Lastly, assertion \ref{itriang} is immediate from the combination of its part 2 with Proposition \ref{psh}(\ref{ish3}).  
\end{proof}

We will usually call the objects of $\gdb$ ($\afo$-) {\it pro-spectra}.

Now we deduce a few simple consequences from our proposition. 

\begin{coro}\label{css}

\begin{enumerate}
 
 \item \label{iisom}
 A $\gdb$-morphism $f:X\to Y$  is an isomorphism if and only if $\gdb(-,C)(f)$ is for any $C\in \obj \sh$.

\item \label{iisomil} 
 In particular, for any projective system $I$ and any compatible system of morphisms $f_i\in \gdp(X_i,Y_i)$ we have the following fact: $\hogd(\prli f_i)$ is an isomorphism if and only if $\inli(\gdb(-,C)(\hogd(f_i))$ is an isomorphism for any $C\in \obj \sh$.  
 
 \item \label{icopr} For any $P_i\in \obj \popa$ we have $\omd(\sqcup P_i)\cong \prod \omd(P_i)$. In particular, if  $S=\sqcup S^\al$ is a decomposition of a pro-scheme and $j\ge 0$ then the object $\omd(S_+)$ is naturally isomorphic to $\prod \omd(S^\al_+)$.
\end{enumerate}
\end{coro}
\begin{proof}
1. For $Z=\co f$ it suffices to note that 
  we have $Z=0$ if and only if $Z\perp \sh$. The latter statement is given by 
	 Lemma \ref{locat}(2) below (
	that is precisely the dual to Proposition 1.1.4(I.1) of \cite{bpure}).  

2. It suffices to combine the previous assertion with 
Proposition \ref{pprop}(\ref{icoclim}).

3. 
The remarks made in \S\ref{sprs} say that the second part of the assertion is a particular case of the first one indeed. So we verify the latter.

Denote the index set by $I$; then the obvious $\popa$-projections  $(\sqcup_{i\in I} P_i)_+\to P_{i+}$ give a canonical morphism $\omd(\sqcup P_i)\to\prod \omd(P_i)$. We have to check that this is an isomorphism. For this purpose (according to assertion \ref{iisom}) we can fix  $N\in \obj \sh$ and verify for $H=\gdb(-,N)$ that $H(\omd(\sqcup_{i\in I} P_i))\cong H(\prod_{i\in I} \omd(P_i))$. By definition, $(\sqcup_{i\in I} P_i)_+$ is the inverse limit of $(\sqcup_{i\in J} P_i)_+$ for $J$ running through finite subsets of $I$. Applying Proposition \ref{pgdb}(4) we obtain that $H(\omd(\sqcup_{i\in I} P_i))\cong\inli_J\sh(\om(\sqcup_{i\in J} P_i), N)$. Recalling that $\om(\sqcup_{i\in J} P_i)\cong \bigoplus \om(P_i)$ (see Proposition \ref{psh}(\ref{ish4})), we obtain the isomorphism in question.

\end{proof}

\subsection{%
The Gysin distinguished triangle
and "Gersten" Postnikov towers} 
\label{scgersten}

Let us define normal bundles for closed embeddings of pro-schemes.

\begin{defi}\label{dnorm}
 For 
 a pro-scheme $S=\prli S_i$ (for $S_i\in \sv$) 
 we define a vector bundle on it as an element of the direct limit of the sets of isomorphism classes of vector bundles over $S_i$ (cf. Remark \ref{resmooth}(1)). 
\end{defi}

\begin{rema}\label{rpgysin}
1. This enables us to define normal bundles to closed embeddings of pro-schemes  via "continuity". 

2. If a pro-scheme $S$ is actually a  scheme, then any closed embedding of $S$ into $Y$ does yield a normal 
 vector bundle over it; this bundle is easily seen to be isomorphic to the one defined  via the method that we have just described.
Moreover, if $S$ is affine then this normal bundle  is a projective module over the coordinate ring of $S$. 
 Furthermore, if  we assume in addition that $S$ is connected then the rank of this module is the codimension of $S$ in the corresponding component of $Y$ (cf. \S\ref{sprs}). 

3. Furthermore, these observations are compatible with the isomorphisms given by Proposition \ref{psh}(\ref{ish5}). Moreover, we can also pass to inverse limits for distinguished triangles given by part \ref{ish3} of the proposition if the connecting morphisms come from $\opa$; see Proposition \ref{pprop}(\ref{itriang}).
 
\end{rema} 
 
\begin{pr}\label{pinfgy}
Let $X$ be a pro-scheme and assume that $Z=\sqcup Z^\al$ ($Z^\al$ are connected) 
 is its closed 
 sub-pro-scheme. Then the following statements are fulfilled. 
 
 1. 
The natural morphism $\omd(X\setminus Z_+)\to \omd(X_+)$
  extends to the following  distinguished triangle (in $\gdb$):
  $$\omd(X\setminus Z_+)\to \omd(X_+)\to \prod \omd(N_{X,Z^\al}/N_{X,Z^\al}\setminus Z^\al )$$
	(see Proposition \ref{psh}(\ref{ish5} for the notation);
	 here $\al$ runs through the set of connected components of $Z$.

2. Assume that the normal bundles  for all $Z^\al$ are  trivial and that $Z$ (and so, all $Z_\al$) are of codimension $c\ge 0$ in $X$.  Then the latter product (in the formula) converts into $\prod \omd(Z^\al \lan c\ra )\cong \prod \omd(Z^\al_+)\lan c\ra$ (see Proposition \ref{pprop}(\ref{ifun})) for the definition of the pro-spectral version of $-\lan c \ra$). 

In particular, this is the case if all $Z^\al$ come from semi-local $k$-schemes and $Z$ is of codimension $c$ in $X$. 
\end{pr}
\begin{proof}

1. As we have just noted, one can pass to the inverse limits of distinguished triangles given by Proposition \ref{psh}(\ref{ish3}). It remains to note that Corollary \ref{css} enables us to 
rewrite the third vertex of this triangle in the  desired form.

2. If the normal bundles over all $Z^\al=\prli Z^\al_i$ are trivial, then they are also trivial for certain $Z^\al_{i_{\al}}$. Hence we can pass to the limit to obtain $N_{X,Z^\al}/N_{X,Z^\al}\setminus Z^\al \cong Z^\al \lan c\ra $ (in $\popa$). 
Next, $\prod\omd(Z^\al \lan c\ra )\cong (\prod \omd (Z^\al_+))\lan c\ra$ according to Proposition \ref{pgdb}(8). 
To obtain the "in particular" part of the assertion also it remains to apply  Remark \ref{rsemiloc}.

\end{proof}

\begin{rema}\label{r552}
 The isomorphism of the second assertion is (usually) not canonical.
\end{rema}

Now let us construct a certain Postnikov tower (see Definition \ref{dpoto}) $Po_M$  for $M$ being the (twisted) $\afo$-pro-spectrum of a pro-scheme $Z$ that will be related
to the coniveau spectral sequences for (the cohomology of) $Z$. 
 Note that we consider the general case of an arbitrary pro-scheme $Z$ (since  
pro-schemes are important for this paper);  yet this case is not much distinct from the (particular) case of $Z\in\sv$.

\begin{pr}\label{post}

Let $S$
 be a pro-scheme; 
 for all $i\ge 0$ we will write $S^i$ for the set of points of $S$ of codimension $i$.

Then 
there exists a Postnikov tower for $M=\omd(S_+)$ such 
 $M_{i}\cong \prod _{s\in S^i}\omd(s_+) \lan i\ra $, the objects $Y_{i}$ are described by the formula (\ref{ept}) below (and the morphisms $Y_i\to M$ are obtained from obvious embeddings of pro-schemes). Respectively, this tower is bounded if $S$ is of dimension $\le d$ for some $d\ge 0$.
\end{pr}
\begin{proof}

Since any product of distinguished triangles is distinguished (see Remark 1.2.2 of \cite{neebook}), we  can assume $S$ to be connected (and its dimension will be denoted  by $d$).


We consider a projective system $L$ whose elements are sequences of closed subschemes
$\varnothing=Z_{d+1}\subset Z_d\subset Z_{d-1}\subset \dots \subset Z_0$. Here $Z_0\in\sv$, $Z_i\in \var$ for all $i>0$, $S$ is (pro)-open in $Z_0$,   
$Z_0$ is connected, and 
for all $i>0$ we have the following: 
(all irreducible components of) all $Z_i$ are everywhere of codimension $\ge i$ in $Z_0$; $Z_{i+1}$ contains the singular locus of $Z_i$ (for $i\le d$).
The  ordering in $L$ is given by open  embeddings of varieties $U_i=Z_0\setminus Z_i$ for $i>0$.
For $\lambda\in L$  the corresponding sequence will be denoted by $\varnothing=Z^\lambda_{d+1}\subset Z^\lambda_d\subset Z^l_{d-1}\subset
\dots \subset Z^\lambda_0$. 

Now, for any $i\ge 0$ the limit $\prli_{\lambda\in L} (Z^\lambda_i\setminus Z^\lambda_{i+1})_+$ equals $(\sqcup_{s\in S^i}s)_+$. 
Hence the previous proposition yields a distinguished
triangle $$\omd (\prli(Z^\lambda_0\setminus Z^\lambda_i)_+)\to
\omd(\prli(Z^\lambda_0\setminus Z^\lambda_{i+1})_+)\to \prod
_{s\in S^i}\omd(s_+)\lan i\ra .$$ Thus setting 
\begin{equation}\label{ept} 
Y_{i}=\omd (\prli(Z^\lambda_0\setminus Z^\lambda_{i+1})_+)\end{equation}
 for $-1\le i\le d$ 
one obtains 
a tower as desired.

\end{proof}

\begin{rema}\label{lger} 
1. 
Below for any $i\ge 0$ we will write  $S^{\le i}$ for  the pro-scheme $\prli Z^\lambda_0\setminus Z^\lambda_{i+1}$.  Note that the pro-schemes  $S^{\le i}$ can be easily defined "componentwisely" 
 for $S$ being not necessarily connected, and Corollary \ref{css}(\ref{icopr}) implies that we can take $Y_{i}=\omd(S^{\le i}_+)$ in the non-connected case also.

2. Certainly, if we shift by $[c]$ (for some $c\in \z$) the Postnikov tower  for  $\omd(S_+)$ that was constructed above, we obtain a Postnikov tower for $\omd(S_+)[c]$. Moreover, since (for any $j\ge 0$) the endo-functor $\lan j\ra$ respects products in $\gdb$, we can "twist" our towers by $\brj$. However, we are interested in those towers that are  weight Postnikov tower with respect to the Gersten weight structure $w$ (that we will construct in \S\ref{scwger}).
So we apply the functor $-\brjj=\brj\circ [-j]$ to the tower constructed above and obtain a tower for $\omd (S\brjj)$ whose "factors" $M^{i}=M^{-i}[i]$ are of the form $\prod _{s\in S^i}\omd(s_+)\{i+j\}$ for all $i\ge 0$ (and $M^i=0$  for $i>0$; recall that $S^i$ denotes the set of point of $S$ of codimension $i$) and $Y_i=\omd(S^{\le i})\{j\}$. 

3. One certainly may tensor Postnikov towers for  $\afo$-pro-spectra by (shifts of) arbitrary objects of $\hk$ (and even of $\sh$). 
 The advantage of the twists $-\brjj$ is that they are weight-exact (with respect to the Gersten weight structure; see Theorem \ref{tgw}({iwgtwn})).

We also recall that in \cite{bger} (cf. Remark \ref{roldger}(3) below) and in the previous versions of this text a more 
 ad hoc treatment of twists was used. Instead of defining twists on the whole category $\gdb$ (resp. on comotives treated in \cite{bger}) only objects similar to $\omd(S\lan j \ra)$ (for $S$ being a pro-scheme) were defined. Respectively, the corresponding triangles and Postnikov towers have to be constructed for all $j\ge 0$ "independently". However, this 
 does not cause much difficulties in our setting, and this approach allows to avoid Proposition \ref{pgdb}(8).
\end{rema}

\subsection{The 
$\afo$-Gersten weight structure:  
 construction and basic properties}\label{scwger}

Now we describe the main weight structure of this paper.

We apply Theorem \ref{tnews} in the case $\cu=\gdb$, $C=\{\omd(X_+)[i]\} $ for $X\in \sv,\ i\ge 0$. 
Denote the corresponding category $\cu'$ by $\gd$; note that the full embedding $\ho(c):\sh\to \gdb$ restricts to an embedding $\shc\to \gd$ (we will just write $\shc\subset \gd$).
We obtain a weight structure $w$ on $\gd$. We will call it the {\it Gersten} weight structure, since it is closely related to Gersten resolutions of cohomology (cf.
Remark \ref{rconiv}(2) below). 
By default, 
below $w$ will denote
this Gersten weight structure.

It will be convenient for us to use the following definition; the terminology will be justified in Theorem \ref{tds}(III) below.

\begin{defi}\label{dmotloc}
1. For $d\ge 0$ we will say that a pro-scheme $S$ is {\it  of $\afo$-cohomological dimension at most $d$} if $\omd(S_+)\in \gd_{w\le d}$.\footnote{
Note that in this case we actually have $\omd(S_+)\in \gd_{[0,d]}$; see Theorem \ref{tgw}(\ref{iwg1}).} 
 In the case $d=0$ will also say that $S$ is an {\it $\afo$-point}.

2. We will write $S^{00}$ for the object  $\om(\pt_+)$. 
\end{defi}

\begin{theo}\label{tgw}

Let $S$ be a pro-scheme, $j\ge 0$.
Then the 
following statements are valid.
\begin{enumerate}
	\item\label{icharw} 
	We have  $\gd_{w\le 0}=({}^{\perp_{\gdb}}\sh^{t\le -1})\cap \obj \gd $. In particular, $S$ is of $\afo$-cohomological dimension at most $d$ if and only if for any $E\in \sh^{t\le 0}$ the group $E^{d+1}(S)$ vanishes.

\item\label{iwgtwn}
The functor $\brjj:\gdb\to \gdb$ sends $\gd$ into itself, and its restriction to $\gd$ is weight-exact. 
	
\item\label{iwg1} $\omd(S_+)\in \gd_{w\ge 0}$. 

\item\label{iwg2} If $S$ is 
the spectrum of a function field over $k$, or if $S$ is semi-local and $k$ is infinite then  $S$ is an $\afo$-point.

\item\label{iwg5}  
 The Postnikov tower for $\omd(S_+)\brjj$ given by 
 Remark \ref{lger}(2) is a weight Postnikov tower for it. In particular, if  all points of $S$ are of codimension at most $d$ in it (this is certainly the case whenever $S$ is of dimension $\le d$) then $S$ is  of $\afo$-cohomological dimension at most $d$.  

\item\label{iwgh} $\hw$ is 
equivalent to the Karoubi envelope of the subcategory of  $\gd$ whose objects are all $\prod \omd(\spe (K_{i})_{+})\{ j_i \}$ for $K_i$ 
being function fields over $k$, $j_i\ge 0$.

\item\label{iwshcb} All objects of $\shc$ are $w$-bounded in $\gd$ (recall that we assume $\shc$ to be a subcategory of $\gd$).
 
\item\label{iwshc} We have $\obj\shc\cap \sh^{t\le 0}=\obj\shc\cap \gd_{w\ge 0}$. Moreover, these classes equal the smallest extension-closed Karoubi-closed subclass of $\obj\shc$ that contains $\{\sinf X_{+}:\ X\in \sv\}$. 

\item\label{iwg6} 
 $w$ is right non-degenerate.

\item\label{iwg3} If $f:U\to S$ is an open embedding of pro-schemes such that the complement is of codimension at least $ i$ in $S$ 
 then $\omd(S/U)\in \gd_{w\ge i}$.
\end{enumerate}

\end{theo}
\begin{proof}
\ref{icharw}. The first part of the assertion is given by 
\cite[Corollary 5.5.3(6)]{bpure}. To obtain the "in particular" part one should combine the general statement with Proposition \ref{pgdb}(4) (and recall the definition of $E^{*}(S)$ in Definition \ref{dsh}(1)). 

\ref{iwgtwn}. Since $\brjj$ respects $\gdb$-products, to prove that it sends $\gd$ and $\gd_{w\ge 0}$ into themselves it suffices to note that $\omd(X_+)\brjj$ is a retract of $\omd ((X\times \gmm^{j}))_+$ (see Remark \ref{rpsh}(1)).

It remains to verify that $\brjj$ maps   $\gd_{w\le 0}$ into itself.  Applying the previous assertion, we reduce this statement to the following one: for any $M\in \gd_{w\le 0}$ and $E\in \sh^{t\le -1}$ we have $M\brjj\perp c(E)$. Now, if $M$  can be presented as $\prli M_i$ for some $M_i\in\obj \psh$ then  Proposition \ref{pgdb}(4,8) implies that $\gdb(M,c(E))\cong \inli \sh(M_i\brjj, E)\cong \inli \sh(M_i\brjj, E_{-j})\cong \gdb(M,c(E_{-j}))$\footnote{This clumsy argument can probably be replaced by a more elegant "model-theoretic" one.} (see Remark \ref{rppsh}(4)). It remains to recall that $E_{-j} \in \sh^{t\le -1}$ (see Proposition \ref{psht}(4))   and apply the previous assertion once again.

\ref{iwg1}. 
For any $U\in \sv$ we have $\omd(U_+)\in \gd_{w\ge 0}$ by 
Remark \ref{rnew}. 
So we reduce the general case of the  assertion to this one.

By Corollary \ref{css}(\ref{icopr}) it suffices to verify the statement for the connected components of $S$ (recall that $w$ is cosmashing). Hence it is sufficient to prove 
that the result is valid for all pro-schemes of dimension $\le d$ by induction on $d$.

So, for some $d\ge 0$ we can assume  that $\omd(V_+)$ belongs to $\gd_{w\ge 0}$  if $V$ is any pro-scheme of dimension $\le d-1$. It suffices to verify that the same is true for a fixed connected $S$ of dimension $d$.

Next, using Proposition \ref{post} 
one can easily see that $S$ can be replaced by its generic point, i.e., it suffices to verify the statement for $S$ being the generic point of some smooth (connected) $U\in \sv$ of dimension $d$. 
Moreover, applying this proposition to $U$ we obtain the corresponding 
distinguished triangles (see the notation of Definition \ref{dpoto}) $$ \prod
_{u\in U^i}\omd(u_+)\lan i\ra [-1] \to Y_{i-1}\to Y_{i}$$ for $0< i\le d$. Now, 
$Y_{d}=\omd (U_+)$ belongs to 
$\gd_{w\ge 0}$  and the same is true for $\prod _{u\in U^i}\omd(u_+)\lan i\ra[-1]$ for all $0<i\le d$ according to the inductive assumption combined with the weight-exactness of $\{i\}$. Hence $Y_0=\omd(S_+)$  belongs to $\gd_{w\ge 0}$ also (since this class is extension-closed; here we use obvious downward induction).


\ref{iwg2}. 
According to 
assertion \ref{icharw} (combined with Proposition \ref{pgdb}(4)), for any $E\in \sh^{t\le 0}$ we should verify that the group $E^1(S)$ vanishes. The latter statement for  pro-schemes in question is given by 
Proposition \ref{psht}(4) and Proposition \ref{pshinv}, respectively.

\ref{iwg5}. By Corollary \ref{css}(\ref{icopr}), it suffices to verify the statement for the connected components of $S$ (recall that any product of distinguished triangles is distinguished according to Remark 1.2.2 of \cite{neebook}, and $w$ is cosmashing). Thus we can assume that the Postnikov tower provided by 
Remark \ref{lger}(2) is bounded. %
 Next, by 
assertions \ref{iwg2} and \ref{iwgtwn} we have $M^i\in \gd_{w=0}$ (see Definition \ref{dpoto}(2)) for the tower in question. Hence Proposition \ref{pbw}(\ref{iwpostc}) yields the result. 

\ref{iwgh}. Immediate from the previous assertion by Theorem \ref{tnews}(III). 

\ref{iwshcb}. Immediate from the  previous two assertions.

\ref{iwshc}. This is just a particular case of \cite[Corollary 5.4.1(11)]{bpure}.

\ref{iwg6}. See Theorem \ref{tnews}(II.3).

\ref{iwg3}. Combining  Proposition \ref{post}  with assertion \ref{iwg5} we obtain the existence of a commutative diagram
$$\begin{CD}
 w_{\le i-1}\omd (U_+)@>{}>>\omd (U_+)\\
@VV{=}V@VV{g}V \\
w_{\le i-1}\omd (S_+) @>{j}>>\omd (S_+)
\end{CD}$$
(cf. Remark \ref{rstws}(1)).
 Hence it remains to  apply Proposition \ref{pbw}(\ref{ipostn}).
\end{proof}

\begin{rema}\label{rgws}

1. Describing weight decompositions for arbitrary objects of $\shc\subset \gd$ explicitly appears to be rather difficult. Still, one can say something about these weight
decompositions and weight complexes using their functoriality properties. In particular, knowing weight complexes for $X,Y\in \obj \shc$
(or just $\in \obj \gd$) and $f\in \gd(X,Y)$ one can describe the weight complex of $\co(f)$ up to a homotopy
equivalence as the corresponding cone. 
 Moreover, let $X\to Y\to Z$ be a distinguished triangle (in $\gd$). Then for any choice of  $(w_{\le 0}X, w_{\ge 1}X)$ and $(w_{\le 0}Z, w_{\ge 1}Z)$ there exists a choice of $(w_{\le 0}Y, w_{\ge 1}Y)$ such that there exist distinguished triangles  $w_{\le 0}X \to w_{\le 0}Y\to w_{\le 0}Z$ and  $w_{\ge 1}X \to w_{\ge `}Y\to w_{\ge 1}Z$;  see Lemma 1.5.4 of \cite{bws}. 
 
 
2. The author does not know whether $w$ is also left non-degenerate; cf. Remark 2.3.5(3)  of \cite{bkw}. In any case, we will mostly be interested in bounded objects.

3. Certainly, we could also have considered the Gersten weight structure on the whole $\gdb$. 
 Yet this does not seem to make much sense, since this weight structure 
 just "ignores" the objects of $\cuperp$; cf.  Theorem \ref{tnews}(II.4). 

4. Below we will consider the analogues of the notion of $\afo$-cohomological dimension corresponding to other "motivic" categories; so that the more clumsy term "$\sh$-cohomological dimension" would have been more coherent.

\end{rema}

\subsection{$\afo$-cohomological dimension and direct summand results} 
\label{stds}

Theorem \ref{tgw} easily implies the following interesting results.

\begin{theo}\label{tds}

 Assume that  $S_0$ is a dense open sub-pro-scheme of a pro-scheme $S$; denote $\omd(S)$ by $M$.

I. Assume that $S$ is an $\afo$-point. 

1.  Then 
 $M$ is a direct summand of $\omd(S_{0+})$.

2. Consider the Postnikov tower of $M=\omd (S_+)$ given by Proposition \ref{post}; denote the corresponding complex by $t(M)=(M^i)$. Then  there exist some  
 $N^i\in \gd_{w=0},\ i\le 0$, such that $t(M)$  is $C(\hw)$-isomorphic to $M\bigoplus (\bigoplus_{i\le 0}(N^i\stackrel{\id_{N^i}}{\to}N^i)[-i])$ (here we consider $M$ as an object of $\hw$).

3. Suppose 
 that $S_0=S\setminus U$, where  $U$ is a closed
sub-pro-scheme of $S$.\footnote{So, $U$ is pro-smooth (by our convention).} 
Then we have $\omd(S_{0+})\cong 
 M \bigoplus \omd(N_{S, U}/N_{S, U}\setminus U)[-1]$.

4. Assume in addition that all the components of $U$ are of codimension $j$ in $S$ and the corresponding normal bundles (see Remark \ref{rpgysin}(1)) are trivial. Then 
we have $\omd(S_{0+})\cong  \omd(S_{+}) \bigoplus \omd(T \lan j\ra)[-1]$.

II.  Assume that $k$ is infinite. 

1.  Assume that   $S$ is semi-local and $S_0=S\setminus \cup_{1\le i\le n} D_i$, where $D_i$ are close connected 
sub-pro-schemes of $S$ of codimension $1$ with intersections of all subsets of $\{D_i\}$ being pro-schemes also (so, one may say that $\cup D_i$ is a divisor with smooth crossings).\footnote{Recall also that all divisors of $S$ are principal (if $S$ is connected).} Then $S_0$ is an $\afo$-point. 

2. In particular, the scheme $S_0=\spe (R_f)$, where $R$ is the local ring of a point $x$ of 
  some $X\in \sv$ and $f$ is a local parameter at $x$, is an $\afo$-point.
	
III. The following conditions are equivalent.
	
	a). For any $X\in \sv$ and $n>d$ we have 
	$M\perp \omd(X)[n]$.
	
	b). $S$ is of $\afo$-cohomological dimension at most $d$. 
	
	c). For any cohomological functor $H$ from $\gd $ into $  \au$ 
	 and the corresponding weight spectral sequence $T(H,M)$ (see Proposition \ref{pwss}) we have $E_2^{pq}=0$  for any $q\in \z$ and $p>d$.
	
	d). We have $E_2^{pq}T(H,M)=\ns$ for any $q\in \z$, $p>d$, and all $H$ that are $\gd$-represented by $\omd(X)$ for $X\in \sv$.

e). For any  $N\in \obj \shi$ and  $n> d$ we have $H^n_{Nis} (S,N)=\ns$, where we set 
\begin{equation}\label{enisc}
H^n_{Nis} (S,N)=\inli_i H^n_{Nis} (S_i,N)\end{equation} 
whenever $S=\prli S_i$  for some $S_i\in \sv$.

IV. Assume that $S$ is the $\popa$-inverse limit of certain pro-schemes $S_i$ of $\afo$-cohomological dimension at most $d$.\footnote{The author does not know whether all of the examples of this situation come from the case where $S_i=\prli S_i^j$ and $S=\prli S_i^i$ are certain "presentations" of pro-schemes (of the form described in the definition of this notion).} Then $S$ is  of $\afo$-cohomological dimension at most $d$ also.

 \end{theo}
\begin{proof}

Once again (cf. the proof of 
 Proposition \ref{post}), we can assume that $S$ and $S_0$ are connected.

I.1. 
 Theorem \ref{tgw}(\ref{iwg3}) says that a cone of the obvious morphism $\omd(S_{0+})\to M$ belongs to $\gd_{w\ge 1}$.
 Hence  Proposition \ref{pbw}(\ref{ipostn}) yields the result. 

2.  Theorem \ref{tgw}(\ref{iwg5}) says that $t(M)$ is a weight complex for $M$. Hence the result is given by Proposition \ref{pwc}(\ref{isplwc}).

3.  Proposition \ref{pinfgy}(1) gives a distinguished triangle $$\omd(S_{0+})\to M\to \omd(N_{S, T}/N_{S, T}\setminus T)[-1].$$ By Theorem \ref{tgw}(\ref{iwg1},\ref{iwg3})  we have 
$M\in \gd_{w= 0}$, whereas the cone object 
 $\omd(N_{S, T}/N_{S, T}\setminus T)$ belongs to $ \gd_{w\ge 1}$. Hence (the "moreover" part of) Proposition \ref{pbw}(\ref{isump}) gives the result.

4. We argue similarly to the previous proof. By Proposition \ref{pinfgy}(2) we have a distinguished triangle 
\begin{equation}\label{eagys} \omd(S_{0+})\to \omd(S_+)\to \omd(T_+\lan j\ra).
\end{equation}
 According to Theorem \ref{tgw}(\ref{iwg1},\ref{iwgtwn},\ref{iwg3}), 
we have  
 $\omd(S_+)\in \gd_{w= 0}$
 and $\omd(T_+\lan j\ra)\in \gd_{w\ge j}$. Thus applying  Proposition \ref{pbw}(\ref{isump}) to 
the triangle (\ref{eagys}) 
 we obtain the result. 

II.1.  We prove the assertion for all possible $(S,D_i)$ 
 by induction on $n$. 

The case $n=0$ is given by Theorem \ref{tgw}(\ref{iwg2}).

Assume that our assertion is valid for $n=j-1$ (for some $j\ge 1$). Hence $S'=\setminus\cup_{1\le i\le n-1} D_i$ is an $\afo$-point. Moreover, $D_n$ is semi-local
(by Remark \ref{rsemiloc}); hence its normal bundle in $S$ is trivial and the inductive assumption implies that $T'=D_n \setminus\cup_{1\le i\le n-1} D_i$ is an $\afo$-point also. Thus assertion I.4 gives an isomorphism $
\omd(S_{0+}) \cong \omd(S'_{+}) \bigoplus \omd(T'_+) \{1\}$. It  remains to apply Theorem \ref{tgw}(\ref{iwgtwn}).

2.  We apply the previous assertion for $S=\spe(R)$ and  $n=1$. It suffices to note that the corresponding $D_1$  is a closed sub-pro-scheme  of codimension $1$ of $S=\spe(R)$  indeed since $f$  is a local parameter. 

III. The implication a)$\implies$b) is given by the definition of $\gd_{w\le d}=\gd_{w\le 0}[d]$.
 The implication b)$\implies$c) is given by Corollary \ref{cwss}(2). The implication c)$\implies$d) is obvious.

Now we verify that condition d) yields condition a). 
For any $X\in \sv$ and $n>d$ we should check that $\gd(M,\omd(X)[n])= \ns$. So we fix $X$ and take $H=\gd(-,\omd(X)$;   we are to verify that $H^n(\omd(S_+))=\ns$ for all  $n>d$. 

Next, we can certainly assume that $S$ is connected (see Corollary \ref{css}(\ref{icopr})). 
Then $\omd(M)$ is $w$-bounded according to Theorem \ref{tgw}(\ref{iwg1},\ref{iwg5}); it also belongs to $\cu_{w\ge 0}$. Hence it suffices to verify for the spectral sequence $T(H,M)$ that $E_2^{pq}T(H,M)=\ns$ if $p+q>n$ (since Corollary \ref{cwss}(2) gives the convergence of $T(H,M)$).

Now, 
since $H^i$ annihilates  $\gd_{w=0}$ for all $i>0$, 
we have $E_1^{pq}T(H,M)=\ns$ for $q>0$; hence $E_2^{pq}=\ns$ for $q>0$ also. Combining this with the vanishing of  $E_2^{pq}$ for $p>d$ (that is provided by condition c)) we obtain that $E_2^{pq}=\ns$ whenever $p+q>n$ indeed.

It remains to verify that conditions a)--d) are equivalent to e). First we note that Proposition \ref{psht}(6) gives  
an isomorphism $E^n(S)\cong H^n_{Nis} (S,\pi^0(E))$ for  any $E\in \sh^{t=0}$ and any pro-scheme $S$. Hence condition b) implies condition e) according to 
Theorem \ref{tgw}(\ref{icharw}). Lastly, it is easily seen that condition e) implies condition d) if we apply Corollary \ref{ccompss}(I.2) below. 

IV. It suffices to combine condition e) of the previous assertion with Proposition \ref{pgdb}(5).  

\end{proof}

\begin{rema}\label{rds}
1. Take $S^{\le 0}$ as in Remark \ref{lger}(1); so, it  the union of (all) generic points of $S$. Then $S^{\le 0}$ is certainly an $\afo$-point. Combining this fact with part I.1 of our theorem we obtain that $S$ is an $\afo$-point if and only if $\omd(S_+)$ is a retract of $\omd(S^{\le 0})$. Recall also that these two conditions are equivalent to $\omd(S_+)$ being a retract of $\omd(S_{0+})$ for any $S_0$ as in the theorem.


2. 
Now we describe a generalization of the previous remark. 
According to Theorem \ref{tgw}(\ref{iwg3}), 
 for any $d\ge 0$ and   $S$  of $\afo$-cohomological dimension at most $d$  a cone of $\omd(S_{d+}) \to M$ belongs to $\gd_{w\ge d}$ whenever $S_{d}$ is an   open sub-pro-scheme of $S$ with complement of codimension at least $d+1$.  Hence  
 in the triangle $\omd(S_{d+}) \to M\stackrel{f}{\to} C\to \omd(S_{d+}) [1]$ we have $f=0$. Thus $M$ is a retract of $\omd(S_{d+})$ for any pro-scheme $S_d$ of this sort. 

These two conditions on $S$ are also equivalent to  $\omd(S_+)$ being  a retract of $\omd(S^{\le d}_+)$, where $S^{\le d}$ is as in Remark \ref{lger}(1). Indeed, 
one can take $S_d=S^{\le d}$ in the 
 second condition; thus to establish the equivalence it suffices to note that $\omd(S^{\le d})\in \gd_{w\le d}$ according to   Theorem \ref{tgw}(\ref{iwg5}), and  the class $\gd_{w\le d}$ is Karoubi-closed in $\obj \gd$.

So, one may say that $S$ is of $\afo$-cohomological dimension at most $d$ if and only if  $\omd(S_+)$ is  a retract of the pro-spectrum of a "standard" pro-scheme $S^{\le d}$ of $\afo$-cohomological dimension at most $d$.

Thus we have a pretty fine understanding of $\afo$-cohomological dimension of pro-schemes from the "$\sh$-point of view". However, it is probably difficult to characterize all $\afo$-points "geometrically" (cf. the succeeding parts of this remark). Note in particular that the product of any finite set of open subvarieties of $\afo$ is an $\afo$-point.  

3. The most important case of the notion of $\afo$-cohomological dimension at most $d$ in this paper is the one where $d=0$ (i.e., we mostly describe 
 $\afo$-points). However, one can easily construct various "non-standard" (cf. part 2 of this remark)  
 examples of pro-schemes of $\afo$-cohomological dimension at most $d$ that are not $\afo$-points.

Firstly, let $T$ and $S$ be as in part I.3 of the theorem. Assume that $T$ is of codimension $c$ (everywhere) in $S$; 
 suppose also  that the normal bundle to $T$ in $S$ is trivial.\footnote{The latter condition is probably not necessary for our argument. Moreover, it is automatic if $T$ is an $\afo$-point; 
see Remark \ref{rpicz} below.}
 Then combining our theorem with Theorem \ref{tgw}(\ref{iwgtwn}) one obtains that $S_0$ is of $\afo$-cohomological dimension at most $d+c-1$
if  $T$ is of $\afo$-cohomological dimension at most $d$. Moreover, the converse implication is true (at least) in the case $d=0$.\footnote{The author does not know whether the converse implication is valid in general. Note however that  the obvious $T$-cohomological version of this statement (see Theorem \ref{tshtt}(III) below) follows from Theorem \ref{tshtt}(II.\ref{wtex}) easily}.
 In particular, if $S$ is the localization of a connected smooth variety of dimension $c$ in a closed point $T$ then  $S_0$ is of  $\afo$-cohomological dimension "precisely $c-1$".

Another interesting family of examples can be obtained using products of schemes. Note here that $\afo$-cohomological dimensions "do not add when we multiply schemes"; in particular (and this appears to be equivalent to the general case of our observation) 
 the (scheme-theoretic) product of spectra of function fields over $k$ is "usually" not an $\afo$-point. 
 However, this problem disappears under certain restrictions on one of the schemes. So, assume that $X\in \sv$, $\om(X_+)$ lies in the subcategory of $\shtc$ densely generated by $\{S^{00}\brjj:\ j\ge 0\}$ (one may say that $\om(X_+)$  is a mixed Tate spectrum), and $X$ is of $\afo$-cohomological dimension at most $d_X$. One can easily check that these conditions imply that $M_X$ belongs to the smallest Karoubi-closed extension-closed subclass of $\obj\shtc$ containing $\{S^{00}\brjj[i]:\ j\ge 0,\ 0\le i\le d\}$. Hence for any pro-scheme $S$ of dimension at most $d_S$ one can easily verify that Theorem \ref{tgw}(\ref{iwgtwn}) implies the following:  the naturally defined pro-scheme $X\times S$  is of $\afo$-cohomological dimension at most $d_X+d_S$. More generally, generalizing the adjunction argument used in the proof of the aforementioned theorem, one can easily generalize this statement to the case where $\om(X_+)$ is an {\it Artin-Tate spectrum} (cf. \S\ref{sat} below).

4. Let us demonstrate the utility of  part IV of our theorem. 

An easy inductive argument shows that throwing away any finite collection of hyperplanes from $\af^n$ (for any $n\ge 0$) yields a smooth variety that is an $\afo$-point (cf. the proof of part II.1 of the theorem). Hence throwing away an arbitrary  collection of hyperplanes gives an $\afo$-point also.

One can also "pass to the limit" in part II.1 of our theorem. The essential smoothness of the corresponding intersection is somewhat difficult to control; yet one can obtain quite non-trivial examples (at least) 
 if $k$ is "large enough" (i.e., of infinite transcendence degree over its prime subfield). 

5. Note that the spectral sequence considered in part III of the theorem is the coniveau spectral sequence for $(H,M)$; see Proposition \ref{rwss}(I.3) below.

6. One may also study the weights of "general" pro-open embeddings similarly to the arguments above; this would correspond to the study of relative cohomology. Moreover, one can probably consider $\afo$-pro-spectra corresponding to certain multi-relative cohomology in a similar way. However, the author does not know how to obtain any "unexpected" results in this setting (i.e., how to prove some results that do not follow from the properties of $\afo$-cohomological dimension).  
\end{rema}

\begin{coro}\label{cds} 

1. Assume that $k$ is infinite; let $K$ be a function field  over $k$ and let $K'$ be the residue field for a geometric valuation $v$ of $K$ of rank $r$.
Then $\omd(\spe (K')_+)\{r\}$ is a retract of $\omd(\spe (K)_+)$.


2. $\hw$ is equivalent to the Karoubi envelope of the category of all  $\prod \omd(\spe (K_{i})_{+})\in \obj \gd$ for $K_i$ being 
function fields over $k$. 
\end{coro}
\begin{proof}

1. Obviously, it suffices to prove 
the statement in the case $r=1$  (see Theorem \ref{tgw}(\ref{iwgtwn})).

Next,  $K$ is the function field  of some normal projective variety over $k$. Hence there exists  $U\in \sv$ such that $k(U)=K$ and $v$ yields
a non-empty closed subscheme of $U$ of codimension $1$. 
It easily follows that there exists a pro-scheme $S$ 
 whose only points are the spectra of $K$ and $K'$ (it is  the spectrum of a discrete valuation ring). Hence $S$ is an essentially smooth local affine scheme; 
thus it is an $\afo$-point according to Theorem \ref{tgw}(\ref{iwg2}). Therefore it remains to apply 
part I.4 of the previous theorem. 

2. If $k$ is infinite, 
 assertion 1 yields that we can get rid of the twists mentioned in the (very similar) Theorem \ref{tgw}(\ref{iwgh}).

In $k$ is finite, one should apply the fact that $\omd(\spe (K')_+)\{r\}$ is a retract of $\omd(\spe (k(G_m^r(K)))$ instead (whereas the statement mentioned can be easily established using the method of the proof of Theorem \ref{tds}(I.1)). 

\end{proof}

\begin{rema}\label{rstds} 
1. Note that we do not construct any explicit splitting morphisms in any of the decompositions above. Probably, one cannot choose any canonical
splittings here (in the general case); so there is no (automatic)
compatibility for any pair of related decompositions. Respectively,
though the pro-spectra 
coming from  function fields contain tons of
direct summands, there seems to be no general way to decompose them
into indecomposable summands.

2. Still Proposition \ref{pinfgy} easily yields that
$$\omd(\spe (k(t))_+)\cong 
S^{00} \bigoplus \prod\omd(z_+)\{ 1\};$$ here $z$
runs through all closed points of $\af^1$ (considered as a scheme over $k$; 
 recall that $\omd(\af^1_+)\cong \omd (\pt_+)$). 

\end{rema}

\section{On cohomology and coniveau spectral sequences}  \label{sapcoh}

Now we relate the properties of pro-spectra to cohomology.

In \S\ref{sextkrau} we describe (following H. Krause,  as we previously did in \cite{bger} and \cite{bpure}) a natural method for extending cohomological functors from 
$\shc$ to $\gd$. This method is compatible with the usual (colimit) definition of cohomology for essentially smooth $k$-schemes. 

In \S\ref{sext} we  (easily) apply the results of the previous section to cohomology (mostly, of $\afo$-points).

In \S\ref{sdconi} we consider  weight spectral sequences corresponding to (the Gersten weight structure) $w$. We prove that these
spectral sequences naturally generalize    classical coniveau spectral sequences. Besides, for a fixed $H$ from $\gd$ into $\au$ our (generalized) coniveau
spectral sequence converging to $H^*(M)$ (where $M$ is  an arbitrary object of $\shc$  or a $w$-bounded object of $\gd$) is $\gd^{op}$-functorial in $M$ (in particular,  it is $\shc{}^{,op}$-functorial if restricted to $\shc$); this  fact is quite non-trivial (even 
 for the spectra of smooth varieties).

In \S\ref{sconi} we construct a nice duality $\Phi:\gd^{op}\times \sh\to \ab$; we prove that $w$ is orthogonal to $t$ with respect to $\Phi$. It follows that  our {\it generalized coniveau} spectral sequences can be expressed in terms of $t$ (starting from $E_2$); this vastly generalizes the corresponding seminal result of \cite{blog} (where this result was established for the de Rham cohomology).

In \S\ref{sprovar} we prove that this property of  generalized coniveau spectral sequences allows to calculate $T(H,M)$ for $M$ being the spectrum corresponding to the inverse limit of a general system of smooth $k$-varieties (as the corresponding colimit).

\subsection{Extending cohomology from $\shc$ to $\gd$ (reminder)}\label{sextkrau}

Certainly, we would like to apply the results of the previous sections to the cohomology of pro-schemes. The problem is that cohomology is "usually" defined on $\sh$ (or on its subcategories). Thus we need a reasonable way for "extending" cohomology to $\gd$.
So we follow \S1.2 of \cite{bger} 
and  describe 
 a general method for extending cohomological functors from a full triangulated $\cu'\subset\cu$ to $\cu$ (after H. Krause; it was noted in Remark 5.1.4.(I.1) of \cite{bpure} that these "extensions" are actually 
Kan ones). Its advantage is that it  yields functors that are "continuous" with respect to inverse limits in $\gdp$ (for $\cu=\gd$); it follows (as we will explain below) that we get "reasonable" cohomology of $\omd(S_+)$ for $S$ being any essentially smooth $k$-scheme.

We note that the construction requires $\cu'$ to be skeletally small, i.e., there should exist a  subset (not just a subclass!) $D\subset \obj \cu'$ such that any object of $\cu'$ is isomorphic to some element of $D$; this is certainly true for $\shc$. 

Suppose that $\au$ is an AB5 abelian category.
We recall that for 
 any small $\cu'$ the category $\adfu(\cupr^{op},\au)$  is abelian also; complexes in it are exact if and only if they are exact when applied to any object of $\cu'$, and the same is true for coproducts.

\begin{pr}\label{pextc}
For $\au$ as above assume that  $H'$ is a cohomological functor from $\cupr$ into $\au$.

I.  
1. Then one can construct an  extension of $H'$ 
to 
a cohomological functor $H$ from $\cu$ into $\au$ (i.e., the restriction of $H$ to $\cu'$ is equal to $H'$). 
 This correspondence $H'\mapsto H$ 
 is functorial  in the obvious sense, and it  respects coproducts.

2. Suppose that in $\cu$ we have a projective system $X_l,\ l\in L$, equipped with a compatible system of
morphisms $X\to X_l$, such that the latter system for any $Y\in \obj \cupr$ induces an isomorphism $\cu(X,Y)\cong \inli \cu(X_l,Y)$. Then we have $H(X)\cong \inli H(X_l)$.

II. Apply the previous assertions for $\cu=\gd$, $\cu'=\shc$. 
 Then  the extension of  $H':\shc{}^{,op}\to \au$ to $H:\gd^{op}\to \au$ satisfies the following properties. 

1. $H$ converts  those inverse limits  in $\gdp$ that are mapped by $\hogd$ 
inside $\gd$ into the corresponding direct limits in $\au$. 

More generally, for the left adjoint $L$ to the embedding  $\gd\to \gdb$ (see Theorem \ref{tnews}(IV))  the composed functor $H \circ L\circ\hogd$ 
maps $\gdp$-inverse limits into $\au$-direct ones.

2. $H$ converts products in $\gd$ into coproducts in $\au$.

3. 
$H$ 
can be characterized (up to a canonical isomorphism) as the only cohomological extension of $H'$ to $\gd$ that converts $\gd$-products into $\au$-coproducts.

\end{pr}
\begin{proof}
Assertion I is a simple  application of the results of \cite{krause}; 
see Proposition 1.2.1 of \cite{bger}. 

II. 
These statements are easy consequences of assertion I.2.

In order to obtain assertion II.1 this result should be combined with   
Proposition \ref{pprop}(\ref{icoclim}) along with Theorem \ref{tnews}(IV). 
Next, 
Proposition \ref{pgdb}(6) (the cocompactness of objects of $\shc$ in $\gd\subset \gdb$) 
immediately yields assertion II.2. 

Assertion II.3 is given by (the dual to) Proposition 5.1.3(11) of \cite{bpure} 
(note that $\shc$ cogenerates $\gd$; see Theorem \ref{tgw}(\ref{iwg6})).
\end{proof}

\begin{rema}\label{rcohp}

1.   In the setting of assertion II we will call  $H$ an {\it extended} cohomology theory. Note that assertion II.3 gives a complete characterization of extended theories.

2. 
Combining assertion II.1 with Proposition \ref{pprop}(\ref{icoclim}),  for 
 any pro-scheme $S=\prli S_i$ (for $S_i\in \sv$) we obtain $H(\omd(S_+))\cong \inli H'(\om(S_{i+}))$ (for any $H'$ and $H$ as above). 
This certainly implies the following: if $H'$ is represented by an object $E$ of $\sh$ then for a pro-scheme $S$ we have $H(\omd(S_+))\cong E^0(S)$ (see Definition \ref{dsh}(1)). Moreover, if $E\in \sh^{t=0}$ then $H(\omd(S_+)[n])\cong H^n_{Nis} (S,\pi^0(E))$ 
(where the latter group is defined by the formula (\ref{enisc})) according to Proposition \ref{psht}(6).

3. 
These statements are  coherent  with the standard way  of extending cohomology from (smooth) varieties to their inverse limits; hence  
Proposition \ref{cdscoh} below  can be applied to the "classical"' 
  $K$-theory, algebraic cobordism, motivic cohomology, etc. of semi-local 
 schemes (cf. 
Theorem \ref{tmotdimt}(III)  below; note that all of these theories factor through $\sh$). 
  In particular, the value of an extended theory at the 
 spectrum of an (essentially smooth) local 
ring will be the corresponding Zariski 
 stalk.

Also, recall that $H$ coincides with $H'$ on $\shc$ (see assertion I.1); 
hence we obtain 
 "expected" values of cohomology for all compact objects of $\sh$ also.

4. Recall that $\gdp$ is the category of (filtered) pro-objects of $\psh$. Hence (by part I.2 of our proposition) all extended cohomology theories factor through the category $\proo - \sh$ of "naive" pro-objects of $\sh$ (that is certainly not triangulated). 
Thus  the pairing $\Phi:\gd^{op}\times \sh$ that we will study in \S\ref{sconi} below factors through $(\proo - \sh)^{op}\times \sh$.

\end{rema}

\subsection{On  cohomology of pro-schemes and its retracts}\label{sext} 

We easily prove that the results of the previous section
 imply 
similar assertions for cohomology; 
 case of extended theories is especially interesting.

\begin{pr}\label{cdscoh}
Let $S$ be 
  an $\afo$-point; assume that $S_0$ is a dense open sub-pro-scheme of $S$ 
and $H$ is a cohomological functor from $\gd$ into an abelian category $\au$.


1. Then $H(\omd(S_+))$ is a direct summand of $H(\omd(S_{0+}))$.

2. Suppose moreover that $S_0=S\setminus Z$, where  $Z$ is a closed sub-pro-scheme of $S$. Then we have
$H(\omd(S_{0+}))\cong H(\omd(S_{+})) \bigoplus \coprod H(\omd(N_{S,Z^\al}/N_{S,Z^\al}\setminus Z^\al)[-1])$, where $\al$ runs through the set of connected components of $Z$.

3. In particular, if each $Z^{\al}$ of codimension $j>0$ in $S$ and the 
 normal bundles of all $Z^{\al}$ in $S$ are trivial then this decomposition 
can be rewritten as $H(\omd(S_{0+}))\cong H(\omd(S_{+})) \bigoplus H(\omd(Z_+\brj)[-1])$.

4. 
Consider the Postnikov tower of $M=\omd (S_+)$ given by Proposition \ref{post}; denote the corresponding complex by $t(M)=M^i$. 
Denote by $T_H(S)$ the {\it Cousin} 
 complex 
$(H(M^{-i}))$. Then there exist some $A^i\in \obj \au$ for $i\ge 0$ such that $T_H(S)$ is $C(\au)$-isomorphic to $H(\omd(S_{+}))\bigoplus A^0\to A^0\bigoplus A^1\to A^1\bigoplus A^2\to \dots$.

\end{pr}
\begin{proof}

1. Immediate  from Theorem \ref{tds}(I.1).

2. Immediate  from  Theorem \ref{tds}(I.3).

3. Immediate  from  part I.4 of that theorem.

 4. 
  See part I.2 of the theorem. 

\end{proof}

\begin{rema}\label{runiv}

1. Our proposition becomes especially useful when combined with Remark \ref{rcohp}(2). So, for $H'$ being any cohomological functor from $\shc$ into an AB5 
 abelian category $\au$ one can take $H$ being the extension of $H'$ to $\gd$. Then in assertion I.3 one can replace any term of the form $H^*(\omd((\prli T_i)_+\lan r\ra))$ by $\inli H'{}^*(\om( T_{i+}\lan r\ra))$ (here $\prli T_i$ is one of  $S$, $S_0$, or $Z$, $T_i\in \sv$, and $r$ is either $0$ or $j$). Moreover, one can also easily express the object $H(\omd(N_{S,Z^\al}/N_{S,Z^\al}\setminus Z^\al )[-1])$ in part 2 of the proposition in terms of $H'$.

Note also that Theorem \ref{tds}(II) provides us with a rich source of $\afo$-points (if $k$ is infinite). 

2. Certainly, assertion 4 (for $H$ being a cohomology theory that factors through $\gd$) is stronger then the universal exactness Theorem 6.2.1 of \cite{suger}. A caution: the definition of the universal exactness given in ibid. is not quite correct; a complex $(C^i)$ of objects of an abelian category $\au$ should be called universally exact whenever for any {\bf AB5} abelian category $\au'$   and an additive functor $F:  \au\to \au'$ respecting filtering direct limits the complex $F(C^i)$ is exact  (see \cite{suzain}).

\end{rema}

Now we apply our results to the calculation of Picard groups of $\afo$-points.

\begin{coro}\label{cretr}
 Let $S=\prli_{i\in I} S_i$ be a connected $\afo$-point; let $S'=\prli S'_i$   be its generic points (for $S_i$ and $S'_i$ belonging to $\sv$).

1. Let $H'$ be a cohomological functor from $\shc$ into an AB5 category $\au$; assume that $\inli H'(\om(S'_{i+}))=0$. Then  $\inli H'(\omd(S_{i+}))=0$ also.

2. 	Any vector bundle $V$ on $S$ (see 
	Definition \ref{dnorm}) is trivial, i.e., there exists $i\in I$ such that $V$ comes from a 
 a trivial vector bundle $V_i$ on $S_i$.

\end{coro}
\begin{proof}
1. Following Remark \ref{runiv}(1) we take $H$ being the extension of $H'$ to $\gdp$. Then the remark demonstrates that  $H(\omd(S'_+))\cong \inli H'(\om(S'_{i+}))$ and   $H(S_+)\cong \inli H'(\om(S_{i+}))=0$ also. Hence $H(\omd(S'_+))=0$ and applying Proposition \ref{cdscoh}(1) we conclude the proof.

2. We recall that the Picard group functor is one of motivic cohomology functors; 
 hence it can be represented by an (Eilenberg-Maclane) object of $\sh$ (being more precise, the Picard group functor on $\sm$ can be factored through a representable functor on $\sh$). 
Thus the functor $\picz$ is $\sh$-representable also.  Moreover, it is well-known that $\picz$ is "continuous", i.e., we have $\picz(S)\cong \inli \picz(S_i)$ and $\picz(S')\cong \inli \picz(S'_i)$. Since  $\picz(S')=\ns$, applying the previous assertion we obtain the result.
\end{proof}

\begin{rema}\label{rpicz}
In particular, for any closed embedding of a (connected) $\afo$-point $T$ into a pro-scheme $S$ the corresponding normal bundle is trivial. 
If $S$ is an $\afo$-point also then we obtain that $\omd((S\setminus T)_+)\cong \omd(S_+)\bigoplus \omd(T_+\brj)[-1]$ where $j$ is the codimension of $T$ in $S$; see Theorem \ref{tds}(I.4).  Certainly, this implies the corresponding result for cohomology (cf. Proposition \ref{cdscoh}(3)).\footnote{In particular, one can prove the triviality of line bundles on semi-local schemes using this observation. Note however that that closed sub-pro-schemes of $\afo$-points are not necessarily $\afo$-points: though $\af^n$ is an $\afo$-point for any $n\ge 0$, plenty of smooth closed subvarieties of $\af^n$ are not $\afo$-points.}

Moreover, one can certainly "iterate" the statement that  $\omd(T_+\brj)[-1]$ is a retract of $\omd(S_+)$; cf. the proof of Corollary \ref{cds}(1).
\end{rema}

\subsection{On generalized coniveau spectral sequences}
\label{sdconi}

Let $H$ be a cohomological functor from $\gd$ into $ \au$, $M\in \obj \gd$.

\begin{pr} \label{rwss}

I.1. 
The weight spectral sequence $T^{\ge 2}(H,M)$ (see Proposition \ref{pwss}(2,3)) corresponding to the Gersten weight
 structure $w$ is canonical and $\gd^{op}$-functorial in  $X$. 

2. $T(H,M)$ converges to $H(M)$ if $M$ is bounded with respect to $w$.

3. 
Let $H$ be the extension to $\gd$ of some cohomological functor $H'$ from $\shc$ into  $\au$  (so, $\au$ is an AB5  category; see Proposition \ref{pextc}),
$M=\omd(S_+)$ for some 
pro-scheme $S$.
Then  the weight spectral sequence $T(H,M)$ corresponding to the Postnikov tower provided by Proposition \ref{post} 
converges and has the form 
\begin{equation}\label{ecssp}   E_1^{pq}=\coprod_{s\in S^p}H^{q}_{p}(s) \implies H^{q+p}(S),\end{equation}  where $S^p$ denotes the set of points of codimension $p$ in $S$, and $H^{q}_{p}(s)=H^q(\omd(s_+)\{p\})$; moreover, for a presentation of $s\in S^p$  as $\prli s_i$ for $s_i\in \sv$ 
 we have   $H^{q}_{p}(s) \cong \inli H^q(\omd (s_{i_+})\{p\})\cong \inli H'{}^q(\om (s_{i_+})\{p\})$.

Furthermore, if $S\in \sv$ (i.e., if  $S$ is a smooth variety that gives the corresponding object of $\popa$) 
then   $T^{\ge 2}(H,M)$  is canonically isomorphic to the corresponding "part" of the  classical coniveau spectral sequence (i.e., we consider  the  coniveau spectral sequence  converging to 
 $H'{}^*(S)\cong H^*(S)$ starting from the $E_2$-sheet; see  \S1 of \cite{suger}). 
In particular, the filtration corresponding to $T(H,M)$  is the coniveau one.

4. Assume  that $H$ equals the extension to $\gd$ of the functor 
 $\sh(-,E):\sh^{c,op}\to \ab$, where  $E\in \obj \sh$, $M=\omd(S_+)\{j\}$ for a pro-scheme $S$ and $j\ge 0$; consider the weight Postnikov tower for $M$ provided by 
Remark \ref{lger}(2).  Then the corresponding choice of $T(H,M)$ is a spectral sequence 
\begin{equation}\label{ecss} 
E_1^{pq}=\coprod_{s\in S^p}E^{q}_{j+p}(s) \implies E^{q+p}_j(S)\end{equation}  (where $S^p$ denotes the set of points of codimension $p$ in $S$).

II.  Let $M'\in \obj \shc\cap \sh^{t\le -r}$ for some $r\in \z$. Then the following statements are valid.

1. $H(M')=(W^{r}(H))(M')$ (see Remark \ref{rintel}(1)).

2. For any $g\in \gd(M,M')$ we have $\imm(H(g))\subset (d{r}(H))(M)$.

3. Assume that $H$ is an extended theory, $M=\omd(Z_+)$ for  $Z\in \sv$. For a morphism $g$ as above 
consider a Noetherian subobject $A$ of  $\imm(H(g))$ (i.e., we assume that any ascending chain of subobjects of $A$ in $\au$ becomes stationary). Then $A$  is {\it supported in codimension $r$}, i.e.,  there exists an open $U\subset Z$ such that $Z\setminus U$ is of codimension $\ge r$ in $Z$ and $A$ is killed by the restriction morphism $H(\omd(Z_+))\to H(\omd(U_+))$.  

\end{pr}
\begin{proof}
I.1. This is just a particular case of  
Proposition \ref{pwss}(2,3).

2. Immediate since $M$ is bounded; see part 4 of that proposition.

 3. Since for $S=\sqcup S_l$ the Postnikov tower $Po_M$ for $M$ equals the product of the corresponding weight Postnikov towers for $\omd(S_{l+})$ (by definition; see Remark \ref{lger}(1)) and $H$ converts products into coproducts (see Proposition \ref{pextc}(II.2)), we can assume that $S$ is connected. 
Thus $Po_M$ is bounded, and we obtain that $T(H,M)$ converges.  Moreover, the calculation of $E_1$-terms in question follows from Proposition \ref{pextc}(II.1,2).


Now we assume that $S\in \sv$. We recall that in 
\S3 of \cite{ndegl}  two exact couples (for the $H'$-cohomology of $S$) were constructed. 
One of them  was obtained by applying $H'$ to 
the geometric towers as in Proposition \ref{post} 
and then passing to the inductive limit (in $\au$).
Moreover, it was shown that this couple yields the same (coniveau) spectral sequence as the other one mentioned in loc. cit. (see \S2.1 of ibid.; cf. also Remark 2.4.1 of \cite{bws}), whereas the latter couple coincides with the "standard" one (constructed using the arguments of    \cite[\S1.2]{suger}; cf. Remark 5.1.3(3) of ibid.). Furthermore,   Remark \ref{rcohp}(2) implies 
 that the  limit mentioned is  (naturally) isomorphic to 
the spectral sequence obtained via  $H$ from $Po_M$.  

4. The proof is quite similar to that of the first part of the previous assertion.

II.1. By Theorem \ref{tgw}(\ref{iwshc}) we have $M'\in \gd_{w\ge r}$. Hence we can take $w_{\ge r}M'=M'$, and the result is immediate from Remark \ref{rintel}(1).

2. Immediate from the $\gd^{op}$-functoriality of our weight filtration (given by the aforementioned remark) along  with the previous assertion.

3. According to the previous  assertion, $A$ lies in $(W^r(H))(M)$.
By  assertion I.3 
this means that $A$ dies in $\inli_{U\subset Z,\ \codim_Z(Z\setminus U)\ge r}H(\omd(U_+))$. Since $A$ is noetherian, it also vanishes in some particular $U$ of this sort.
\end{proof}

\begin{rema}\label{rrwss}

1. Part I.3 of our proposition yields a good reason to call (any choice of)   $T(H,M)$ a {\it generalized coniveau spectral sequence} (for arbitrary $H,\au$, and $M\in \obj \gd$); this will also distinguish (this version of) $T$ from weight spectral  sequences corresponding to other weight structures. We will  give some more justification for this term in Remark \ref{rconiv} below.   So, the corresponding filtration can be called the (generalized) coniveau filtration (for a general $M$).

Note moreover 
that under the assumptions of part I.3 of the theorem the Cousin complex $T_H(S)$ corresponding to 
 our choice of $Po_M$ is isomorphic to the Cousin complex described in \cite[S1.2]{suger}. 

2. It is well known that there exist exact functors from $\sh$ into "all other stable motivic categories"; this certainly includes $\sht$, $\dm(k)$, and $\shmgl$ (see \S\S\ref{sdm}--\ref{swo} below); these functors "commute with twists" and send $\om(S_+)$ for any $S\in \sv$ into the objects corresponding to $S$ in these motivic categories. Thus one can use the spectral sequence (\ref{ecssp}) (as well as its  "twisted" version; cf. (\ref{ecss})) for cohomology theories that factor through any of these categories. This is certainly not surprising since these coniveau spectral sequences can be constructed using the results of \cite{suger}.
 Note also that in \S\ref{ssupl} below it is demonstrated that these 
 spectral sequences can be obtained from the  Gersten weight structures corresponding to these motivic categories.

3. 
Actually, in order to obtain a coniveau spectral sequence for $(H',S)$
using the recipe of 
\cite{deggenmot} and \cite{suger} it is not sufficient to  compute just the cohomology of (the spectra of smooth) varieties. One also needs to apply $H'$ to certain objects of $\opa$ in order to compute the $E_1$-terms of the exact couple, whereas the connecting morphisms of the couple come from the natural comparison morphisms between relative cohomology and the cohomology of varieties (see \S1.1 and Definition 5.1.1(a) of ibid.). So, for those "classical" cohomology theories for which 
 all of this information has an "independent" definition, one should check whether it can be "factored through $\sh$". This seems to be true for all "reasonable" cohomology  theories. 
For $K$-theory this fact is given by Corollary 1.3.6 of \cite{papi}. For \'etale cohomology the proof is easy; one may use an argument from the proof of Theorem 4.1 of \cite{ndegl}.

On the other hand, in order to calculate the coniveau filtration it suffices to know the restriction of $H'$ to the (spectra of smooth) varieties; so this does not require any of this complicated extra information. Besides, our (pretty standard) arguments yield that for 
  any cohomology theory satisfying axioms 5.1.1(a), COH1, and COH3 of \cite{suger} (which is certainly the case for all of the examples we are interested in) the $E_1$-terms of the "standard" coniveau spectral sequences are isomorphic to our (generalized) ones.

4. Assertion II.3 along  with its motivic analogue (see \S\ref{sdm} below and part 2 of this remark) may be quite actual for the study of "classical" motives. Note that one can apply it for $M'$ being a cone of some morphism $\omd(Z'_+)\to \omd(Z_+)$ for some $Z\in \sv$ (here $Z$ can be a point, and we certainly consider the triangle $\omd(Z_+)\stackrel{g}{\to}  M'\to \omd(Z'_+)[1]$). Besides, 
 one can take  $H$  being the $i$-th cohomology functor for some "standard" cohomology theory  and   $i\in \z$.

Moreover, one can easily prove 
the assertion for any (not necessarily additive) functor $\gd^{op}\to \au$ that converts homotopy limits 
into direct limits and sends zero morphisms into zero maps.

Certainly, the statement is interesting only if $r>0$.

\end{rema}

\subsection{
Duality for $\gd$ and $\sh$; comparing  spectral sequences}\label{sconi}\label{sdual}

In order to apply the formalism of orthogonal structures we need the following statement.

\begin{pr}\label{pdualsh}
For each $M\in \obj \sh$ consider the (cohomological) functor $H_M$ from $\gd$ into $ \ab$ obtained by extending $\gd(-,M)$ 
via Proposition \ref{pextc}(II). 

Then the following statements are valid.

1. The  pairing $\Phi:\gd^{op}\times \sh\to \ab:\ \Phi(X,M)=H_M(X)$ is a nice duality of triangulated categories.


2. The Gersten weight structure $w$  on $\gd$ is $\Phi$-orthogonal to the homotopy $t$-structure $t$ on $\sh$. 

3. For any $M\in \obj \sh$ the functor $\Phi(L\circ \hogd(-),M)$ converts  filtered inverse limits in $\gdp$ into direct
 limits in $\ab$. 

4. For any pro-scheme $S=\prli_i S_i$ (for $S_i\in \sv$), $E\in\obj \sh$, and $j\ge 0$ we have  $$\Phi(\omd(S\brj),E)\cong \inli_i \sh(\om(S_{i}\brj),E) =E^{-j}_j(S).$$
In particular, if $j=0$, $E\in \sh^{t=-n}$ for some  $n\in \z$ then the group in question is isomorphic to $H^n_{Nis} (S,\pi^{-n}(E))$.
\end{pr}
\begin{proof}
1. Immediate from Proposition 5.2.5 of \cite{bpure} (as well as from Proposition 2.5.6(3) of  \cite{bger}).

2. According to Corollary 5.5.3(1) of \cite{bpure}, the statement is given by Corollary 5.4.1(9) of ibid.

3. Immediate from Proposition \ref{pextc}(II.1).

4. The previous assertion gives $\Phi(\omd(S\brj),E)\cong \inli_i \Phi(\omd(S_{i}\brj),E)$; this gives the first isomorphism according to the definition of $\Phi$.
The succeeding  equality is just the definition of $E^{-j}_j(S)$.

To establish the "in particular" isomorphism we note that $E^0(S_i)\cong H^n_{Nis} (S_i,\pi^{-n}(E))$ according to Proposition \ref{psht}(6). It remains to recall the definition of    $H^n_{Nis} (S,\pi^{-n}(E))$ (see the formula  (\ref{enisc})).

\end{proof}

Now let us  relate generalized coniveau spectral sequences to the homotopy $t$-structure (in $\sh$). This is a vast extension of 
 \cite[Proposition 6.4]{blog} (where de Rham cohomology was considered; cf. also the calculation of $E_2$-terms for Poincar\'e duality theories in Theorem 6.1 of ibid.) and of  \cite[\S4]{ndegl}.

\begin{coro}\label{ccompss}

Assume that $H$ is represented by a $E\in \obj\sh$ (via our $\Phi$), $M\in \obj \gd$.  

 I.1. Then our generalized coniveau spectral sequence $T^{\ge 2}(H,M)$ 
    can be naturally and $\gd^{op}\times \sh$-functorially expressed in terms of the cohomology of $M$ with coefficients in the $t$-truncations of $E$ (as in Proposition \ref{pdual}). 

2. Assume that $M=\omd(S_+)$ for some pro-scheme $S$. 
Then $T(H,M)$ converges to $H^*(M)$ 
and  $E_2^{pq}
 \cong H^p_{Nis}(S, \pi^q(E))$ (see Remark \ref{rts}(2) and the formula (\ref{enisc})). 

3. 
Assume in addition to the previous assumptions that $E\in \sh^{t=0}$. Then   $T(H,M)$ degenerates at $E_2$, $E_2^{pq}=\ns$ for $q\neq 0$, and  $E_2^{p0}\cong E^p(S)$. 

II.1.  Assume that $M\in \gd_{w=0}$.  Then $T(H,M)$ converges to $H^*(M)$ and degenerates at $E_2$ also. Moreover, $E_2^{pq}=\ns$ for $p\neq 0$ and  $E_2^{0q}\cong \Phi(M,E[q])$.  

2. In particular, the previous assertion  
  can be applied for $M=\omd(S_+)$ for $S$ being any $\afo$-point.  Thus  $k$ is infinite then one can take $S$ to be either a semi-local pro-scheme or is the complement to a semi-local pro-scheme of a divisor with smooth crossings (see Theorem \ref{tds}(II.1) for more detail on this condition).
\end{coro}

\begin{proof}
 I.1. Immediate from Proposition \ref{pdualsh}(2) (combined with Proposition \ref{pdual}).

2. Once again, it we can assume that $S$ is connected. In this case $M$ is $w$-bounded and we obtain convergence. Next, the previous assertion yields that 
$E_2^{pq}\cong \Phi(M[-p],E^{t=q})$, and it remains to apply  Proposition \ref{pdualsh}(4). 

3.  Immediate from the previous assertion (cf. also Corollary \ref{cwss}(3)).

II.1. Immediate from  Proposition \ref{pdualsh}(2)   combined with Corollary \ref{cwss}(1).

2. Immediate from Theorem \ref{tds}(II.1) (cf. also Theorem  \ref{tgw}(\ref{iwg2}) for the first of the cases in question).  
\end{proof}

\begin{rema} \label{rconiv} 

1. Our comparison assertion I.1 is  true (in particular) for the $E$-cohomology of an arbitrary $M\in \obj \shc$;  this  extends to $\shc$ (and to $\sh$-representable cohomology theories) 
  Theorem 4.1 of \cite{ndegl}\footnote{To prove that our result generalizes loc. cit. one should use the $t$-exactness of the connecting functor $\psishmot:\dmk\to \sh$ (that is right adjoint to the natural functor $\sh\to \dmk$) is $t$-exact; see Remark \ref{rfunctpr}(2) below.}  (whereas in \S5.3 of ibid. it is explained that the latter theorem 
	extends the results of \S6 of \cite{blog}). We obtain one more reason  to call $T$ (in this case) a generalized coniveau spectral sequence for (the cohomology of) $\afo$-pro=spectra.

Note also that the methods of D\'eglise do not (seem to) yield the $\shc{}^{,op}$-functoriality of the isomorphism in question. Moreover, the approach of D\'eglise definitely does not yield Proposition \ref{cdscoh}. 

2. If $E\in \sh^{t=0}$ 
 then $E_2(T)$ yields the Gersten resolution for the (strictly $\afo$-invariant) sheaf  $\pi^0(E)=\tpi^0(E)$  (cf. \S6 of \cite{minthesis}); 
 this is why we called $w$ the Gersten weight structure.

3. Recall from  Theorem \ref{tds}(II.2) 
that one can take $S=\spe (R_f)$, where $R$ is the local ring of a point $x$ of 
  some $X\in \sv$ and $f$ is a local parameter at $x$, in part II.2 of our corollary.  Hence the latter assertion essentially generalizes Lemma 14.1 of \cite{minthesis} where the corresponding 
	 $\pi^0(E)$ was  assumed to be a {\it homotopy invariant sheaf with transfers}.
	
	Note here that $k$ was not assumed to be infinite in loc. cit; however,  transfers enable one (using an easy standard argument; see Theorem 6.2.5 of \cite{suger}) to reduce loc. cit. to the case of an infinite $k$. 

\end{rema}

\subsection{On  coniveau spectral sequences for (more general) pro-varieties}\label{sprovar}

Now we associate certain pro-spectra to a class of  regular $k$-schemes that is (essentially) wider than that of pro-schemes.

\begin{defi}\label{dprovar}
We will say that a Noetherian  
  $k$-scheme $S$ is a {\it pro-smooth variety} if it can be presented as an inverse limit of smooth $k$-varieties.

\end{defi}

\begin{rema}\label{rprovar}
\begin{enumerate}
\item\label{ipv1} Obviously, to any pro-smooth variety $S=\prli S_i$ we can associate the pro-spectrum 
 $(\pom(S_{i+}))\in \obj\gdp$ 
 (see (\ref{efun})). So we also get the corresponding $\omd(S_+)\in \obj \gdb$, and $\omdp(S_+)=L(\omd(S_+))\in \obj \gd$. 

\item\label{ipv2} Certainly, any pro-scheme corresponding to an actual 
 $k$-scheme is a pro-smooth variety. Since we will not develop much theory for  pro-smooth varieties, we will not extend Definition \ref{dprovar} to allow 
 pro-smooth varieties with an infinite number of connected components (cf.  \S\ref{sprs}).

\item\label{ipv3} It is well known that all pro-smooth varieties are regular (see \cite[\S1]{spivak}). 
Conversely, the Popescu theorem (see Theorem 1.1 of ibid.) says that any affine regular Noetherian $k$-scheme is a pro-smooth variety. Moreover, Proposition 8.6.3 of \cite{ega43} implies that any open subscheme of a pro-smooth variety is a pro-smooth variety also.

\item\label{ipv4} It appears to be rather difficult to demonstrate for a given $M\in \obj \gdb$ that it does not belong for $\gd$. In particular, we possibly have  $\omd(S_+)\in \obj \gd$ for any  pro-smooth variety $S$ (i.e., it is not necessary to apply $L:\gdb\to \gd$; cf. Remark \ref{rcomparss}(\ref{ic1}) below). However, this conjecture is not relevant for the purposes of the current paper (cf. Remark \ref{rnew}(2)).
\end{enumerate}

\end{rema}

Now we prove that generalized coniveau spectral sequences posses a certain "continuity" property.

\begin{pr}\label{provar}
Let $S=\prli S_i$ be a pro-smooth variety (as above) and $H$  be an extended cohomological functor from $\gd$ into $\au$ (so, $\au$ is an AB5 abelian category). Then the direct limit of generalized coniveau spectral sequences $T^{\ge 2}(H,\om(S_{i+}))$ 
  equals the     generalized coniveau spectral sequence \linebreak $T^{\ge 2}(H,\omdp(S_{+}))$ 
	 (here the transition morphisms are given by the functoriality provided by Proposition \ref{pwss}(3)).
\end{pr}
\begin{proof}
According to Proposition \ref{pextc}(II.1), for any $N\in \obj \shc$ we have $\gd(\omdp(S_{+}),N)\cong \inli\gdb(\omd(S_{i+}),N)=\inli\shc(\om(S_{i+}),N)$. 
Thus $\omdp(S_{+})$ is a $\shc$-limit of the system  $\omd(S_{i+}))$ in the sense of \cite[Definition 5.1.1]{bpure} (cf. Proposition \ref{pextc}(I.2)). So it remains to apply  Remark 5.1.4(II.3--4) of ibid.
\end{proof}

\begin{rema}\label{rcomparss}
\begin{enumerate}
\item\label{ic1} Now we demonstrate that it is necessary to assume in our proposition that $H$ is extended (indeed).

For this purpose we consider a simple example of a pro-smooth variety. So, we assume that $k$ possesses a $\zl$-extension $K$ for $l\neq \cha k$ (though our example can be easily adjusted to avoid this assumption) and take $S_i=\spe( k_i)$, where $k_i$ are the Galois extension of $k$ of degree $l^i$ whose union equals $K$ (and so, $S=\spe (K)$).

Since in this case the 
 system $(S_i)$ is countable, we have a $\gdb$-distinguished triangle $$\prod \omd(S_{i+})[-1]\stackrel {d}{\to}\omd(S_{+})\to \prod \omd(S_{i+})\to \prod \omd(S_{i+})$$
 (cf. Definition 1.6.4 of \cite{neebook}). 
Moreover, this triangle does not split, i.e., $d\neq 0$, since $\omd(S_{+})$ is not a retract of $\prod \omd(S_{i+})$; the latter can be easily verified by  applying the ("extended version" of) the  $\zlz$-etale cohomology functor $H_{et,\zlz}^*(-\times \spe (K))$ (with values in the category of $\zlz[\gal(K/k)]$-modules). Hence $\omd(S_{+})\in \gd_{[-1,0]}$. 

Now we consider a functor $H:\gd\opp\to \ab$ that sends $ M\in \obj \gd$ into \break $\ab(\gd(\prod  \omd(S_{i+})[-1],M),\q/\z)$.\footnote{This functor is easily seen to be (co)homological; this is a certain "dual" of the homological functor $H':M\mapsto \gd(\prod  \omd(S_{i+})[-1],M)$.} Then for the corresponding spectral sequences for $\omd(S_{i+})$ we have $E_2^{pq}=\ns$ for $p\neq 0$ and also for $q\le 0$. On the other hand,   $\gd(\prod  \omd(S_{i+})[-1],M)$ contains a non-zero element $d$; hence $H^0(M)\neq \ns$; thus there exists $(p,q)$ with $p+q= 0$ and $E_2^{pq}(H,\omd(S_{+})\neq \ns$ (note that this spectral sequence converges).

\item\label{icperf} Now we are able to treat the question of "$w-t$-perfectness" of $\Phi$ (the formulations below will explain what do  these words mean). 
 
Since $t$ is non-degenerate, 
 Proposition \ref{pdualsh}(4)  implies for $N\in \obj\sh$ that we have $N\in \sh^{t\le 0}$ (resp. $N\in \sh^{t\ge 0}$) if and only if  $\cu_{w=i}\perp_{\Phi} N$ for all $i<0$ (resp. $i>0$); we will write these facts down 
 as  $\sh^{t\le 0}=(\cup_{i<0}\gd_{w=i})^{\perp_{\Phi}}$ and  $\sh^{t\ge 0}=(\cup_{i>0}\gd_{w=i})^{\perp_{\Phi}}$. Moreover,  
\begin{equation}\label{eperf}
\gd_{w\le 0}={}^{\perp_{\Phi}}\sht^{t\le -1} \end{equation}
according to Corollary 5.4.1(9) of \cite{bpure}. 

On the other hand, combining Proposition \ref{provar}  with the example described 
 above we obtain $\gd_{w\ge 0}\neq {}^{\perp_{\Phi}}\sht^{t\ge 1}$. 
 
Lastly recall that this picture is rather different from the one in the case $\Phi=\cu(-,-):\cu\opp\times \cu\to \ab$. Note that  in the latter setting orthogonal "structures" $w$ and $t$ on $\cu$ are said to be adjacent (recall that Remark 5.2.2(1,2) of \cite{bpure} says that the condition $w\perp_{\Phi}t$ in this case is equivalent to $\cu_{w\ge 0}=\cu^{t\le 0}$, whereas the latter one is the "usual" definition of adjacency that was given in \cite{bws}). 
Thus we have $\cu_{w\ge 0}= {}^{\perp_{\Phi}}\cu^{t\ge 1}$ according to Remark \ref{rts}(3). \footnote{Moreover, the three remaining natural counterparts of the latter statement are fulfilled also.} 

\item\label{ic2} Now we discuss the question of comparing of generalized coniveau spectral sequences for $H^*(\omdp(S_+))$ (for $H$ being an extended functor and a pro-smooth variety $S$) with the "usual" ones; note that the latter can be defined using direct limits and Proposition 8.6.3 of \cite{ega43} (cf. Remark \ref{rprovar}(\ref{ipv3})). 
 An easy limit argument shows that these two spectral sequences are canonically isomorphic if the transition morphisms between $S_i$ are \'etale (cf. the example described above). 

Generalizing this result further appears to be non-trivial; cf. somewhat related results and ("geometrically non-trivial") arguments in \cite{minthesis}.

\end{enumerate}
\end{rema}

\section{On $T$-spectral and motivic Gersten weight structures and the corresponding cohomological dimensions}\label{ssupl}

In this section we describe 
certain 
variations of the arguments and results of the previous ones, and discuss several examples of cohomology theories. We will be somewhat
sketchy sometimes.

In \S\ref{sht} we prove that the (natural analogue of the) Gersten weight structure can  be constructed for the stable motivic category of $T$-spectra also.  
Moreover, the obvious analogues of the properties of $w$ established above can be proved without any difficulty for this weight structure $w^T$ on $\gdt$; this includes several equivalent definitions for the notion of {\it $T$-cohomological dimension}.

In \S\ref{sdm} we study triangulated categories of pro-objects (generalizing $\gd$ and $\gdt$) 
 in more detail. This yields 
 certain conditions ensuring that the Gersten weight structure for $\gdt$ is "compatible" with a similar weight structure on the corresponding category of pro-objects for a model category $\gm$ equipped with a left Quillen functor from a "model for $\sht$". We apply the general result to the study of 
the category $\dmk$ of  Voevodsky motives to obtain a category $\gdm$ of {\it comotives}. The properties of $\gdm$ 
are similar to that of the  category constructed in \cite{bger}; however, our current methods allow us to  
drop the assumption of countability of $k$ (and we discuss the distinctions between methods).  
We also recall the notion of primitivity of schemes; any primitive pro-scheme is a {\it motivic point}.

In \S\ref{swo} we  prove that motivic dimensions of schemes "control"  certain properties of 
 their cohomology  (including some cohomological functors 
 that do not factor through motives). 
For this purpose we introduce the notions of {\it weakly orientable} and {\it very weakly orientable} cohomology; 
 examples include all {\it orientable} cohomology theories as well as those that can be factored through the {\it topological realization} functor $\reco$. 
 We also discuss a certain category $\shmgl$ of cobordism-modules along with the corresponding $\gdmgl$; the corresponding 
{\it cobordism-dimension} of a pro-scheme is not larger than its  motivic dimension. 

 
In \S\ref{sshinvp} we study 
 the categories $\gdt[\sss\ob]$ (where $\sss$ is a set of primes) and $\gdtpl$ (that corresponds to Morel's $\shtpl$). The advantage of the latter category is that 
the corresponding $+$-cohomological dimensions of schemes are not larger than their motivic dimension, whereas all $\zoh$-linear very weakly orientable cohomological functors factor through it (and any $\zoh$-linear theory also does whenever $k$ is unorderable). 
So, we obtain some  splitting results for motivic $+$-pro-spectra of pro-schemes that (do not follow from the results of \S\ref{sht} and) induce the corresponding splittings for a wide variety of cohomology.

In \S\ref{sat} we note that for certain varieties and motivic spectra one can choose quite "economical" versions of weight Postnikov towers and (hence) of the (generalized) coniveau spectral sequences for cohomology. We also discuss Gersten weight structures and the corresponding "Artin-Tate substructures" for "relative" motivic categories (i.e., for various categories of motives over rather general base schemes).

\subsection{
On the $T$-spectral 
Gersten weight structure}\label{sht}

In \S5 of \cite{morintao}   the stable model category of {\it $\p^1$-motivic spectra} was considered. 
As noted in Example 10.38 of \cite{jardbook}, the model category of $T$-spectra in $\doshp$ gives a model for this category (here we apply Theorem 10.40 of ibid. and note that the discrete presheaf $\p^1$ pointed by $0$ is $\afo$-equivalent to $T=\afo/\gmm$).
The latter model category will be denoted by $\psht$; its homotopy category will be denoted by $\sht$. Our constructions and results can be carried over from $\psh$, $\sh$ (and other (pro)spectral categories) to this setting; we will say more on this in Theorem \ref{tshtt}  (cf. also Theorem \ref{tfunct} below). In order to prove the theorem we start from formulating the $T$-spectral analogues of Propositions \ref{psh} and \ref{psht}. 

So let us recall that $\psht$ is (also) equipped with a left Quillen functor from $\doshp$ (similarly to Proposition \ref{ppsh}(3), this statement  also 
 follows from 
 the results of \cite{hoveysp}); hence one can define the natural analogues 
$\pomt$ and $\omt$ of  $\pom$ and $\om$, respectively. 
  Moreover, $\psht$ and $\sht$ are equipped with natural endofunctors $-\wedge T$. 
 The main distinction of $T$-spectra from $S^1$-ones is that $-\wedge T=-\lan 1 \ra$ is a left Quillen auto-equivalence of $\psht$; 
so we will assume it to be invertible on $\sht$.
Similarly to \S\ref{ssh},  the operation $-\wedge T^{\wedge j}[-j]=-\lan j \ra [-j]$ will be denoted by $\{j\}$ (however, we will use this notation for all $j\in \z$).

Furthermore, $\sht$ is (also) endowed with a certain homotopy $t$-structure, which  will be denoted by $t^T$. It is defined (see Theorem 5.2.3 of ibid.; we modify the notation a little) via the functors $\pi^n_m(-)$ for $n,m\in \z$; those send $E\in \sht$ to the Nisnevich sheafification of the presheaf $\tpi^n_m: U\mapsto E^{n}_{m}(U)=\sht(\omt(U_+)\{m\}[-n],E) \cong \sht(\omt(U_+),E\{-m\}[n]) $ (cf. Remark 5.1.3 and Definition 5.1.12 of ibid.). 
 Similarly to 
 Definition \ref{dsh}(1) we 
  extend 
  $E^{m}_{n}(-)$ to 
	 $\popa$.
 
 So, one easily obtains 
the properties of $\sht$ listed below (paying attention to Remark \ref{rts}(5)).

 \begin{pr}\label{pshtt}
The following statements are valid.

\begin{enumerate}
\item\label{itwt}
For any $j\ge 0$ we have $\pomt(-\brj)=\pomt(-)\brj$. 

\item\label{it4} 
The functor $-\{j\}$ is $t$-exact with respect to $t^T$ for any $j\in\z$.

\item\label{ihrt} The restriction of the collection of functors $(\pi^0_n)$ to $\hrt^T$ gives a 
 faithful exact functor $\hrt^T\to\shi^{\z}$.

\item\label{ihrtnis} Moreover, the restrictions of the functors  $\tpi^0_n$ and $\pi^0_n$  to $\hrt^T$ are equal, i.e., the values of  $\tpi^0_n$  are Nisnevich sheaves.

\item\label{it3} For $E\in \obj \sht$ we have $E\in  \sht^{t^T\le 0}$ if and only if $E^{n}_j(\spe (K)_{+})=\ns$ for all $n>0$, $j\in\z$, and any function field $K/k$. 

\item\label{it1} For any $X\in \sv$ we have $\omt(X_+)\in \sht^{t^T\le 0}$.

\item\label{it2} For $E\in \obj \sht$ we have $E\in  \sht^{t^T\ge 0}$ if and only if $E^{n}_j(X)=\ns$ for all $X\in \sv$, $n<0$, and $j\in \z$.

\item\label{ihrtc} For any $E\in \sht^{t^T=0}$, $X\in \sv$, and $n\in \z$ 
there is a natural isomorphism $E^n_0(X)\cong H^n_{Nis} (X,\pi^0_0(E))$.

\item \label{ipshinvpsl} Assume that $k$ is infinite
 and 
 $S$ is semi-local;  then for any $E\in \sht^{t^T\le 0}$, $n> 0$, and $j\in \z$ we have $E^{n}_j(S)=\ns$.

\end{enumerate}

\end{pr}
\begin{proof}

\ref{itwt}. Immediate from the definition of $-\lan 1 \ra$ 
on $\psht$. 

\ref{it4}. Immediate from the definition of $t^T$ (this is Definition 5.2.1 of  \cite{morintao}). 

\ref{ihrt}. It suffices to recall that the heart of $t^T$ is equivalent to 
 the category of {\it homotopy modules} (see Definition 5.2.4 of 
 ibid. or Definition 1.2.2 of \cite{degorient}) and the forgetful functor from homotopy modules into $\shi^\z$ is obviously 
 faithful.

\ref{ihrtnis}. We fix some $E\in \sht^{t^T=0}$, $n\in \z$ and denote $\tpi^0_n(E)$ by $E'$. We recall that 
 this presheaf is a Nisnevich sheaf if and only if for any Nisnevich distinguished square (\ref{enisq}) there is an exact sequence $\ns\to E'(X)\to E'(V)\bigoplus E'(Y)\to E'(W)$. Now, applying 
the $\sht$-version of the distinguished triangle (\ref{enistr}) (cf. Remark \ref{rshtpl}(2) below) we obtain a long exact sequence  $\dots \to \sht(\omt(W_+), E[-1])= E^{-1}_n(W)\to E'(X)\to E'(V)\bigoplus E'(Y)\to E'(W)\to \dots$. It remains to note that $\omt(W_+)\perp E[-1]$ according to the orthogonality axiom of $t$-structures.

\ref{it3}. Applying 
 the previous 
 two assertions we obtain the following fact: it suffices to verify for a sheaf of the type $\pi^n_m(E)$ that it is $0$ if it vanishes at all function fields over $k$. 
 Since all  $\pi^n_m(E)$ are  strictly $\afo$-invariant, 
 the statement follows from Lemma 3.3.6 of  \cite{morintao}. 

\ref{it1}. See Example 5.2.2 of ibid.

\ref{it2}. The "if" implication is 
 given by assertion \ref{it3}, and the converse implication follows immediately  from assertion \ref{it1}. 

\ref{ihrtc}. This fact appears to be well-known (cf. Proposition \ref{psht}(6)); see Theorem 3.7 of \cite{anancurve} (or Remark 1.2.4 of \cite{degorient}).

\ref{ipshinvpsl}. The statement 
follows from the previous assertions via the 
method used in the proof of Proposition \ref{pshinv}. 
\end{proof} 

 Let us note that minor modifications of the methods 
 used 
 for the study of $\sh$ yield
the following results. 
 
 \begin{theo}\label{tshtt}

Let $d\ge 0$;  let $S=\prli S_i$ with $S_i\in \sv$ be a pro-scheme.

 I. There exists a  natural analogue $\gdt$ of $\gd$ that is  closed with respect to all small products,
 and is equipped with an exact auto-equivalence $-\wedge T=-\lan 1 \ra$ 
compatible with the $\sht$-version of this functor.
$\gdt$ canonically contains the triangulated subcategory $\shtc\subset \sht$ densely generated (in the sense described in \S\ref{snot}) by $\omt(X_+)\brj$ 
 for $X\in \sv,\ j\in \z$, whereas   $\omt$ extends to a functor $\omdt:\popa\to \gdt$. 
 
Moreover, the obvious "$T$-versions" of Proposition \ref{pgdb}(4,5), 
Corollary \ref{css}(3), and Proposition \ref{pinfgy}  
 are fulfilled. 
 
  II.  The 
 following statements are  valid.

\begin{enumerate}
\item 
 There exists a cosmashing right non-degenerate weight structure $w^T$ on $\gdt$   such that $\gdt_{w^T\ge 0}$  contains 
$C^T=\{\omdt(X_+)[i]\{j\}\}$ for $X\in \sv,\ i\ge 0,\ j\in \z$, and $\gdt_{w^T\le 0}={}^\perp (C^T[1])$. 

\item\label{wtex} $-\{j\}$ is $w^T$-exact for any $j\in \z$.

\item\label{ihrtt}
 $\hw^T$ is equivalent to the Karoubi envelope of the category of all  $\prod \omdt(\spe (K_{i})_{+})\{j_i\}$ for $K_i$ running through function fields over $k$, $j_i\in \z$. 

\item\label{itopen}  If $f:U\to S$ is an open embedding of pro-schemes such that the complement is of codimension $\ge i$ in $S$ 
 then $\omdt(S/U)\in \gdt_{w^T\ge i}$.

\item\label{itpost} 
Consider the Postnikov tower for $\omdt(S_+)$ given by 
the natural $T$-spectral analogue of the construction in Proposition \ref{post}. 
Then this Postnikov tower is a weight one. 

More generally, for any $j\in \z$ there exists a weight Postnikov tower for $M=\omd (S\brjj)$ with $M^{i}=\prod _{s\in S^i}\omdt(s_+)\{i+j\}$ (cf. Remark \ref{lger}(2)).

\item\label{ifi} Assume that $k$ is infinite.
 Let $K$ be a function field  over $k$; let $K'$ be the residue field for a geometric valuation $v$ of $K$ of rank $r$.
Then $\omdt(\spe (K')_+)\{r\}$ is a retract of $\omdt(\spe (K)_+)$.

\item\label{itwss}  For any cohomological functor $H$ from  $\gdt$ into  $ \au$ the weight spectral sequence $T=T(M,H)$ (for $M\in \obj \gdt$) corresponding to $w^T$ 
is $\gdt{}^{op}$-functorial in $M$ starting from $E_2$. Moreover, if $H$ is an extended (from $\shtc$ to $\gdt$; so $\au$ is an AB5 abelian category) functor   and $M=\omdt(X_+)$  for $X\in \sv$ then
 one can choose $T(H,M)$ to be the "standard" coniveau spectral sequence (starting from $E_1$; certainly, $T(H,M)$ does not depend on any choices starting from $E_2$; cf. \S\ref{sext}).

We will call such a spectral sequence $T(H,M)$ a {\it generalized coniveau spectral sequence} (as we also  did in a similar situation above).

\item\label{iext} 
For all   $M\in \obj \sht$ extend the  functor $\sht(-,M)$ (from $\shtc$ to $\ab$) to a cohomological functor from $\gdt$ to $\ab$. 
 Then the collection of these functors yields a nice duality $\Phi^T:(\gdt)^{op}\times \sht\to \ab$ such that $w^T\perp_{\Phi^T}t^T$.

\item\label{iextl}
 For any $N\in \obj \sht$ we have $\Phi^T(\omdt(S_+),N)\cong N^0_0(S)= \inli N^0_0(S_i)$.

\item\label{iphit}  For any $N\in \obj \sht$ the generalized coniveau spectral sequences for the functor $\Phi^T(-,N)$  can be $\gdt{}^{op}$-functorially expressed in terms of the $t^T$-truncations of $N$ starting from $E_2$ (cf. Proposition \ref{pdual}).

\item\label{itsht} 
We have $\obj\shtc\cap \sht^{t^T\le 0}=\obj\shtc\cap \gd_{w^T\ge 0}$.
 
\end{enumerate}

III. 
For $M=\omdt(S_+)$;  
we will say that  $S$ is {\it  of $T$-cohomological dimension at most $d$} if $M\in \gd_{w\le d}$.
In the case $d=0$ will also say that $S$ is a {\it $T$-point}. 

Then the following conditions are equivalent.

\begin{enumerate}
\item\label{ied} $S$  is of $T$-cohomological dimension at most $d$.

\item\label{iet1}
$\omdt(S_+)\in \gdt_{[0,d]}$. 

\item\label{ietcompl}
$\omdt(S_+)$ is a retract of $\omdt(S_{0_+})$ for any     open 
  sub-pro-scheme $S_0$ of $S$ with complement of codimension at least $d+1$.

\item\label{ietsle}
$\omdt(S_+)$ is a retract of $\omdt(S^{\le d}_+)$ for the pro-scheme $S^{\le d}$ defined in Remark \ref{lger}(1).

\item\label{iet2} 
We  have 
$M\perp\omdt(X)\{j\}[n]$ for any $X\in \sv$, $n>d$, and $j\in \z$.

\item\label{iet3} For any cohomological functor $H$ from $\gdt$ to $\au$ 
	 and the corresponding weight spectral sequence $T(H,M)$ (see Proposition \ref{pwss}) we have $E_2^{pq}=0$  for any $q\in \z$ and $p>d$.

\item\label{iet4} 	 We have $E_2^{pq}T(H,M)=\ns$ for any $q\in \z$, $p>d$, and  $H$ that is $\gdt$-represented by $\omdt(X)$ for some $X\in \sv$. 

\item\label{iet2p}  For any $X\in \sv$ and $n>d$ we have $M\perp \omdt(X)[n]$.

\item\label{iett} $\Phi^T(M,N)=\ns$ for any $N\in \sht^{t^T\le -d-1}$.

\item\label{ietc} For any  $N\in \obj \shi$ of the form $\pi^0_0(E)$ for $E\in \sht^{t_T=0}$ 
and  $n> d$ we have $H^n_{Nis} (S,N)=\ns$.

\end{enumerate}
IV. Assume that $S$ is a $T$-point.
\begin{enumerate}

\item\label{idst1}
If 
$S_0=S\setminus Z$, where  $Z$ is a closed sub-pro-scheme of $S$, then 
\begin{equation}\label{edec}
\omdt(S_{0+})\cong \omdt(S_+) \bigoplus \omdt(N_{S, Z}/N_{S, Z}\setminus Z)[-1]. 
\end{equation}

\item\label{isplc}
 For the Postnikov tower of $M=\omdt(S_+)$ given by assertion II.\ref{itpost} 
 denote the corresponding complex by $t(M)=M^i$. 
For $H$ as above denote by $T_H(S)$ the 
complex 
$(H(M^{-i}))$. Then there exist some $A^i\in \obj \au$ for $i\ge 0$ such that $T_H(S)$ is $C(\au)$-isomorphic to $H(S)\bigoplus A^0\to A^0\bigoplus A^1\to A^1\bigoplus A^2\to \dots$.

\item\label{idst1t}
The decomposition 
(\ref{edec}) takes the form
$\omdt(S_{0+})\cong \omdt(S_+) \bigoplus \omdt(Z_+) \brj[-1]$ 
 whenever $Z$ is of codimension $j$ (everywhere) in $S$ and the normal bundle to $Z$ in $S$ is trivial. In particular, this is the case if $Z$ is a $T$-point of codimension $j$ in $S$. 
Moreover, $S_0$ is a $T$-point (under these assumptions) if $j=1$.
 
\item\label{idst2} The natural analogues of 
assertions IV.\ref{idst1},\ref{idst1t} hold for the $H$-cohomology of schemes in question, where $H$ from $\gdt$ into $ \au$ is an arbitrary cohomological functor  (and $\au$ is some abelian category). 
\end{enumerate}

V.1. Assume that $S$ is the inverse limit of certain pro-schemes $S_j$ of $T$-cohomological dimension at most $d$. Then $S$ is  of $T$-cohomological dimension at most $d$ also.

2. If $S$ is  of $\afo$-cohomological dimension $d$ then it is of $T$-cohomological dimension at most $d$  also.

	3. In particular,  if $S$ is of dimension $\le d$ (or all its points are of codimension at most $d$ in it), then  $\omdt(S_+)\in \gdt_{[0,d]}$.

	4. Moreover, $S$ is an $T$-point if $k$ is infinite and $S$ can be presented as the complement to a semi-local pro-scheme of a divisor with smooth crossings (that may be empty; see Theorem \ref{tds}(II.1)). 
	

\end{theo}
\begin{proof}
I. The main distinction of this assertion from its $\sh$-analogue is that we want to extend 
 $-\lan 1 \ra$ to an invertible exact endofunctor on $\gdt$. This is easy since (as we have already said) $-\lan 1 \ra=-\wedge T$ possesses a "nice lift" to $\psht$. Also, $\gdt$ is cogenerated by $\{\omt(X_+)\brjj\}$ for $X\in \sv$, $j\in \z$ instead of $\{\omt(X_+)\}$ only; yet the corresponding modifications of the proofs caused by this fact are quite obvious (cf. the proof of Proposition \ref{proobj}(\ref{ifunprotr}) below). 


II. Once again, we can use the $\sht$-analogues of the arguments used in the previous sections (along  with Proposition \ref{pshtt}). The main distinction is given by the $-\brjj$-stability (for any $j\in \z$) of the set $C^T$ that cogenerates $w$; since $-\brjj$ is an automorphism of $\gdt$, 
this functor is $w^T$-exact on the nose. 

III. 
 Proposition \ref{pshtt} along  with 
assertion I is easily seen to allow us to apply the obvious $\gdt$-versions of the arguments used  for the proof of  Theorem \ref{tds}(III) and Remark \ref{rds}(2) to
 obtain the equivalence  of our conditions \ref{ied}, \ref{iet1}, \ref{ietcompl}, \ref{ietsle}, \ref{iet2}, \ref{iett}, \ref{iet3}, and \ref{ietc}. We also obtain that the conditions we listed imply condition \ref{iet4}, and the latter yields condition \ref{iet2p}.

Hence it remains to verify that condition \ref{iet2p} implies condition  \ref{iett}. 
For this purpose we recall a collection of facts that appear to be well-known for $\sht$; 
 the paper \cite{bondegl} essentially contains a generalization of these facts to the context of "relative motivic categories".\footnote{The main results of ibid. rely on certain resolution of singularities statement; hence, if we apply Corollary 3.3.7(2)  "directly" in the case $p=\cha k>0$ then we will only obtain the $t$-exactness statement in question for the category $\sht[1/p]$ (see \S\ref{sshinvp}  below for the notation) instead of $\sht$. However, one can easily apply the (rather simple) arguments used in the proof of ibid. to $\sht$ also (for any $p$) since one can use Proposition \ref{pshtt}(\ref{it3}) instead of the corresponding case of  Theorem 3.3.1 of ibid.}

So, we consider the localizing subcategory $\she$ (see  Definition \ref{dcomp}(1)  below) of $\sht$  generated by $\{\omt(X_+):\  X\in \sv\}$.
Well-known homological algebra results imply the existence of a left adjoint $s$ to the embedding $i:\she\to \sht$ (this functor is important for the construction of the so-called {\it slice filtration}). 
 Moreover, by  Remark \ref{rts}(4) there exists a $t$-structure $t^{eff}$ for $\she$ such that $\sht^{t^{eff}\le 0}$ is the smallest extension-closed subclass of $\obj \she$ that  contains $\omt(X_+)[n]$  for all $X\in \sv$ and $n\ge 0$ along with all small $\she$-coproducts of  its elements. 
  It is easily seen that condition \ref{iet2p} implies that for any $E\in \obj \she^{t^{eff} \le -d-1}$ we have $E^0_0(S)=\ns$ (since the class of those $N\in \obj \she$ satisfying $E^0_0(S)=\ns$ is closed with respect to coproducts and extensions). 
Hence it remains to check that $s$ sends $\sht^{t^T\le 0}$ into $\she^{t^{eff}\le 0}$. However, the functor $s$ is $t$-exact with respect to these $t$-structures; this statement is given by Corollary 3.3.7(2) of \cite{bondegl} if $\cha k=0$ and can be easily proved similarly to the proof of loc. cit. if $\cha k>0$.

IV. Once again, the $T$-stable versions of the arguments used in the proof of Theorem \ref{tds}(I) (and the applications of the result to cohomology; cf. Proposition \ref{cdscoh}) yield most of the results easily. We should only note that to obtain the "in particular" part of assertion IV.\ref{idst1t} one should argue similarly to Remark \ref{rpicz} (note here that $\picz$ is $\sht$-representable).

V.1. 
Similarly to the proof of Theorem \ref{tds}(IV), it suffices to combine  condition III.\ref{iet2} (or III.\ref{iet2p}) with the "continuity"  provided by assertion II.\ref{iextl}. 

2. Immediate from the equivalence of the conditions \ref{ied} and \ref{ietc} of assertion III. 

3,4. These statements obviously  follow from  the previous assertion combined  with  Theorem \ref{tgw}(\ref{iwg5})  and Theorem \ref{tds}(II.1), respectively.\footnote{Certainly, assertion V.3 also follows from assertion II.\ref{itpost} immediately.}
\end{proof}

 \begin{rema}\label{rshtpl}

1. One can certainly "iterate" part IV.\ref{idst1t}  of our theorem to obtain "new" $T$-points out of "old" ones. In particular, if $S_0$ is the complement to $S$ of a divisor with smooth crossings (cf. Theorem \ref{tds}(II.1)) such that all these crossings are $T$-points then $S_0$ is a $T$-point also.

2. There is a well-known exact functor $\sh\to \sht$ 
 whose composition with the functor $\om:\opa\to \sh$ yields $\omt$; see Remarks 8.2.1 of \cite{levhomotop}.
This functor is easily seen to be compatible with a certain naturally defined exact functor $\gd\to \gdt$ (this requires modifying the model for $\sht$ and $\gdt$, but this only replaces these categories by equivalent ones; we will say more on this matter in Remark \ref{rfunctpr} below). 

This observation can be used for carrying over some of the properties from $S^1$-(pro)spectra to $T$-ones directly; this includes the Gysin triangle for pro-schemes (see part I of Theorem \ref{tshtt}) 
 and part IV.1 of the theorem. 

Moreover, one can prove that this comparison functor $\gd\to \gdt$ is weight-exact (with respect to $w$ and $w^T$; cf. Theorem \ref{tfunct}(\ref{ihtpwe}) below), and this gives an alternative proof of some other parts of our theorem.


3. Similarly to \S\ref{sext}, applying the splitting results for $H$-cohomology, where $H$ is the extension to $\gdt$ of a cohomological functor $H'$ from $\shtc$ into an AB5 category $\au$, one obtains the natural versions of these results for the corresponding direct limits of $H'$-cohomology.

4. The author does not know whether all $T$-points are also $\afo$-points.

\end{rema}

To finish the proof of Theorem \ref{tshtt} we recall some basics on localizing subcategories; we will also need this formalism below.

\begin{defi}\label{dcomp}
Let $\cu$ be a triangulated category closed with respect to (small) coproducts.

 1. For $\du\subset \cu$ ($\du$ is a triangulated category that may be equal to $\cu$)
  we will say that 
	a class  $\cp\subset \obj \cu$ generates $\du$ {\it as a localizing subcategory} of $\cu$ if  
$\du$ is the smallest full strict triangulated subcategory of $\cu$ that contains $\cp$ 
and is closed with respect to  $\cu$-coproducts.
If 
 this condition is fulfilled then we will also say that $\cp$ {\it cogenerates} $\du\opp$ as a {\it colocalizing} subcategory of $\cu\opp$ (cf. Theorem \ref{tnews}(II)). 

2. We will say that  $\cu$ is {\it compactly generated} and (respectively) that $\cu\opp$ is cocompactly generated if
we can choose a set $\cp$ of compact objects that generates $\cu$ as its own localizing subcategory.

\end{defi}

\begin{lem}\label{locat}
1. Assume that $\cu$ is closed with respect to coproducts and generated by a class $\cp$ of its objects as its own localizing subcategory. Then $\cp\perpp=\ns$. 

2. Dually, assume that $\cu$ is closed with respect to products and cogenerated by a class $\cp$ of its objects as its own colocalizing subcategory. Then $\perpp\cp=\ns$. 
\end{lem}
\begin{proof}
Assertion 1 is given by Proposition 1.1.4(I.1) of \cite{bpure} (that relies on Proposition 8.4.1 of \cite{neebook}) and assertion 2 is its dual.

\end{proof}

\subsection{On 
Gersten weight structures for categories "reasonably connected" with $\sht$, 
comotives, and motivic dimensions}
\label{sdm} 

In \cite{bger} a certain category of ("effective") {\it comotives}  was constructed; this was a certain "completion" of Voevodsky's category $\dmge(k)$ (of effective geometric motives). It is quite easy to demonstrate that the category constructed in ibid. can be studied using the methods of (\S\ref{scomot}--\S\ref{sapcoh} of) the current paper; moreover, one does not have to mention model categories in this argument (cf. Remark \ref{roldger}(3) below). 
However, we prefer to develop a general technique for studying Gersten weight structures in various "motivic" categories. We will apply these statements both to comotives and 
 to  certain categories of (pro)-modules over the Voevodsky spectrum $\mgl$. The proposition 
 below is rather difficult to grasp, so we note that it is mostly a generalization of some of the results of the previous section (whereas the choice between upper and lower indices in the notation is often motivated by typographical reasons).

We start from recalling some properties of triangulated categories of pro-objects; this construction originates from \cite{tmodel} (along with \cite{strictmodel}) and was applied for the construction of $\gdb$ in \S\ref{scomot}.

\begin{pr}\label{proobj}
Let $\gm$ be a proper simplicial stable model category; denote its homotopy category by $\shgm$. 

Then the following statements are valid.
\begin{enumerate}
\item\label{iprobig} There exist a stable ({\it strict}) model structure on the category $\proo-\gm$ of (filtered) pro-objects of $\gm$ such that the cofibrations and weak equivalences are essentially levelwise ones, i.e., they are the morphisms that are $\proo-\gm$-isomorphic to morphisms of the form $(X_i)_{i\in I}\to (Y_i)_{i\in I}$ coming from compatible cofibrations (resp. weak equivalences)  $X_i\to Y_i$  (cf. Proposition \ref{pgdb}(2)).

\item\label{iptriang}
 The homotopy category $\gdbgm$ is triangulated and closed with respect to (small) products. 

\item\label{icoco}
The obvious functor $\gm\to\proo-\gm$ yields a full embedding $c_{\gm}$ of $\shgm$ into $\gdbgm$, and the essential image of this embedding is the subcategory of cocompact objects of $\gdbgm$ that cogenerates $\gdbgm$ (as its own colocalizing subcategory). Furthermore, for any projective system $X_i$ in $\gm$ and $M\in \obj \shgm$ we have (for the corresponding pro-object $(X_i)$) 
$\gdbgm((X_i),c(M))\cong \inli \shgm(X_i,M)$. 

\item\label{ifunpro}
Assume that $\gm'$ is a proper simplicial stable model category also; let $L:\gm \leftrightarrows \gm':R$ be a 
Quillen pair. Then the 
 obvious extensions $\proo-L$ and $\proo-R$ of $L$ and $R$ to pro-objects 
form a Quillen pair  ($\proo-\gm\leftrightarrows \proo-\gm'$) also. This pair is a Quillen equivalence 
 whenever $(L,R)$ is.

\item\label{ifunprotr} The derived functors $F:\gdbgm\leftrightarrows \gdbgmp:G$ of $\proo-L$ and $\proo-R$ are exact, and we have $F\circ c_{\gm}\cong c_{\gm'}\circ \ho(L)$.

Moreover, $F$ and $G$  respect (small) products.

\item\label{ifunequ} Assume that left Quillen functors $L$ and $L':\gm\to \gmp$ are connected by a zig-zag of transformations each of those converts all cofibrant objects of $\gm$ into weak equivalences.
Then the corresponding $F$ and $F'$ are isomorphic.

\item\label{itens} 
For a  monoid $(I,+)$ assume that for each $i\in I$ we are given a left 
left Quillen endofunctor $\{i\}_{\gm}$ of $\gm$ and the correspondence $\{-\}_{\gm}$ in lax monoidal the following sense: 
for any $i_1,i_2\in I$ 
the $\gm$-endofunctors $\{i_1+ i_2\}_{\gm}$ and $ \{i_1\}_{\gm}\circ \{i_2\}_{\gm} $ are connected by a zig-zag of transformations  each of those converts all cofibrant objects of $\gm$ into weak equivalences, and $\{0\}_{\gm}$ is 
also connected with the identity functor by a chain of transformations of this sort.\footnote{
Actually, we are only interested in the cases $I=\z$ and $I=\n$. So the reader may assume that $I$ is commutative (though this is not important for us).} 

Then for the corresponding endofunctors  $\{i\}_{\gdbgm}$ of $\gdbgm$ (for $i\in I$) we have  $\{i_1+ i_2\}_{\gdbgm}\cong  \{i_1\}_{\gdbgm}\circ \{i_2\}_{\gdbgm} $
and $\{0\}_{\gdbgm}$ is isomorphic to the identity of $\gdbgm$. 

\item\label{itenspr} 
Let $(\gm',\wedge)$ be a 
 monoidal model category in the sense of 
\cite[Definition 4.2.6]{hovey}  and assume that we are given  a 
 bifunctor 
  $\otimes:\gm'\times \gm\to \gm$ that turns $\gm$ 
into a	$\gm'$-module category in the sense of \cite[Definition 4.2.18]{hovey}. Then for any cofibrant $C^1,C^2\in \obj \gmp$  that are weakly equivalent  (i.e., become isomorphic  in $\ho(\gmp)$)
 the functors
 $C^1\otimes -$ and $C^2\otimes -$ are connected by a zig-zag of transformations as in assertion \ref{ifunequ}. 

Moreover, if $(I,+)$ is a monoid such that $I\subset \obj \gmp$, $I$ contains the tensor unit $1_{\gm'}$, and for any $i_1,i_2\in I$ the objects $i_1+i_2$ and $Q(i_1)\wedge Q(i_2)$ are weakly equivalent, where $Q$ is the cofibrant replacement functor for $\gm'$.
	Then the set of functors  $Q(i)\otimes -$ for $i\in I$ 
	satisfies the conditions of the previous assertion.
\end{enumerate}
\end{pr}
\begin{proof}
\ref{iprobig}.  This is Proposition 5.5.2(1) of \cite{bpure} (that relies on the results of  \cite{tmodel}.

\ref{iptriang}. Recall that the homotopy category of any stable homotopy category is triangulated and closed with respect to products.

\ref{icoco}. These statements are contained in Proposition 5.5.2 of \cite{bpure} (and easily follow from the results of ibid.) also.

\ref{ifunpro}. This is a particular case of Theorem 6.4 of \cite{strictmodel} (cf. also Propositions 6.11 and 6.12 of \cite{tmodel}).

\ref{ifunprotr}. We apply the previous proposition. Certainly,  derived functors of  (left or right) Quillen functors between stable model categories is exact. The equality of the compositions in question is  immediate from the obvious equality of compositions of the corresponding left Quillen functors.

Next, $G$ respects products since it has a right adjoint $F$.

It remains to prove that $F(\prod_{j\in J} X^j)\cong \prod F(X^j)$ for any set of $X^j\in \obj \gdbgm$. We argue similarly to the proof of Corollary  \ref{css}.

Certainly, there is a canonical morphism $f: F(\prod_{j\in J} X^j)\to \prod F(X^j)$. Now,  the images of the functor $c_{\gm'}$ cogenerate the category $\gdbgmp$ according to assertion \ref{icoco}. Hence the dual to Proposition 1.1.4(I.1) of \cite{bpure}  (cf. Corollary \ref{css}(\ref{iisom})) yields that it suffices to verify the following:  $\gdbgmp(-,C)(f)$ is an isomorphism for any $C\in c_{\gm'}(\obj \shgmp)$. Since $C$ is cocompact, we should check that $\bigoplus \gdbgmp(F(X^j),C) \cong \gdbgmp(F(\prod_{j\in J} X^j,C)$.

It suffices to check that $\gdbgm(\prod_{j\in J} X^j,G(C))\cong \bigoplus_{j\in J}\gdbgm (X^j,G(C))$ (for any $C$ as above). 
Therefore it suffices to verify that $G(C)$ belongs to $c_{\gm}(\obj \shgm)$.

We can certainly assume that $C$ comes from a fibrant object $C'$ of $\gmp$. The corresponding constant object of $\proo-\gmp$ is fibrant (also); here we apply the description of   $\proo-\gmp$-fibrations in Definition 5.3 of \cite{tmodel} noting  that the set of morphisms corresponding to this object via Definition 5.2 of ibid. 
  is the one-element set $\{C'\to \pt\}$. Thus $G(C)\cong c_{\gm}(R(C'))$ and we obtain the result.

\ref{ifunequ}. It suffices to recall the definition of left derived functors and that weak equivalence of $\proo-\gm$ are the essentially levelwise  ones.

\ref{itens}. Immediate from the previous assertion.

\ref{itenspr}. Some parts of our argument are rather similar to the proof of 
\cite[Proposition 4.3.1]{hovey}. 

To prove the first part of the statement 
recall that any cofibrant $C\in \obj \gm$ the functor $-\otimes C$ is a left Quillen one (see Remark 4.2.3 of ibid.); hence it respects trivial cofibrations. Thus it suffices to connect $C^1$ with $C^2$ by a zig-zag of  trivial $\gmp$-cofibrations. Now, the morphism $f^1$ from $C^1$ into its fibrant replacement  $C^3$ is a trivial cofibration.
Since   $C^2$ is cofibrant and $C^3$ is fibrant, the isomorphism between these objects in $\ho(\gmp)$ 
can be lifted to a $\gmp$-morphism $f^2$. $f^2$ is a weak equivalence since it becomes an isomorphism in $\ho(\gmp)$; since $C^3$ is cofibrant, $f^2$ is a cofibration by the $2$ out of $3$ axiom for the commutative triangle $\pt\to C^2\to C^3$.

 Next, for the  cofibrant replacement morphism $Q(1_{\gmp})\stackrel{q}{\to}1_{\gmp}$ 
the corresponding transformation $Q(1_{\gmp})\otimes -\implies 1_{\gmp}\otimes -$ converts cofibrant object of $\gm$ into weak equivalences
 according to part 2 of \cite[Definition 4.2.18]{hovey}, and the functor $1_{\gmp}\otimes -$ is certainly the identity of $\gm$.

Hence it remains to prove the following: if  $C_1,C_2$, and $C_3$ are cofibrant objects of $\gm'$
and $C_1\wedge C_2$ is weakly equivalent 
 to $C_3$ then the corresponding functors 
$C_1\otimes-\circ (C_2\otimes-)$ and $C_3\otimes-$
are connected by a zig-zag of transformations  as in assertion \ref{ifunequ}.
Now, 
we certainly have $(C_1\otimes-)\circ C_2\otimes-=(C_1\wedge C_2)\otimes-$ (by the corresponding associativity axiom). 
Moreover,  the object $C_4=C_1\wedge C_2$ is $\gmp$-cofibrant (see the proof of \cite[Proposition 4.3.1]{hovey}). 
Thus applying the first part of the assertion we conclude the proof.

\end{proof}

\begin{rema}\label{rtens}

Suppose that 
under the assumptions of part \ref{itenspr} of our proposition 
the following additional property is fulfilled:  for any cofibrant $C\in \obj \gm$ the functor 
 $-\otimes C$ respects weak equivalences.

1. Then one can certainly 
use the functors $i\wedge-$ instead of  $Q(i)\wedge-$ in the calculation of $\{i\}_{\gdbgm}$. 

2. If $\gmp=\gm$, $\wedge=\otimes$, and the functors $C\otimes -$ for cofibrant $C$ in $\gm$ respect weak equivalences also (and so, the monoid axiom of \cite[Definition 3.3]{schwalg} is fulfilled).
Then Proposition 12.7 of \cite{tmodel} says that 
the obvious definition gives a certain tensor structure on the category $\gdbgm$. 
However, 
 $\gdbgm$ is 
 rarely a closed monoidal category; 
 so we prefer not to consider tensor products of (two) "general" pro-objects.

\end{rema}

Now we apply our proposition to the study of categories "connected with" $\sht$ and $\gdt$.
We will use the notation $\pshtp$ for the  
 category of 
symmetric $T'$-spectra on the smooth Nisnevich site considered in  \S2.1 of \cite{cdint} (setting $X=\spe k$). 
Here the "starting" model structure on $\dosh$ is the so-called projective one, 
 and we will write $T'$ for 
 the cofibrant replacement of the projective line $\p^1$ pointed by $\infty$. This model category is well-known to be Quillen  equivalent to $\psht$.\footnote{Both of these categories are connected with left Quillen functors with the proper model category of $\p^1$-spectra "constructed from  the injective model structure for $\dosh$"; 
the equivalences in question are provided by Theorem 5.7 of \cite{hoveysp}.} 
Moreover, $\pshtp$ is a proper symmetric monoidal simplicial model subcategory 
 that satisfies the monoid axiom of \cite{schwalg} (see Remark \ref{rtens}). 
The tensor product will be denoted by $\wedge$; note that the
aforementioned functor $\wedge T$ on $\sht$ 
 is essentially 
given by a particular case of this tensor structure.

To formulate the theorem below we  adopt  the notation of Proposition \ref{proobj} (for $\gm$ and related matters)  
and introduce some more notation (related to certain assumptions).

Let $\phwgm:\pshtp\to \gm$ be a left Quillen functor. 
We will write $\phgm$ for the functor $\ho(\phwgm)$, $\psigm$ is the right adjoint to $\phgm$ (that is the homotopy functor for the right adjoint to $\phwgm$);
 $\phgdbm=\ho(\proo-\phwgm):\gdbt\to \gdbgm$.  
For any $X/U\in \opa$ and $i\in \z$ we write $\omm(X/U\{i\})$ for the object $ \phgm(\omt(X/U)\{i\})$; for $E\in \obj \shgm$, $X\in \sv$, and any $m,n\in \z$  we will write  
 $E^{m}_{n}(X)$ 
 for the group $\shgm(\omm(X)\{m\},E[n])\cong \shgm(\omm(X_+)\{m\},E[n])$.  
Moreover, for a pro-scheme $X=\prli X_i$ ($X_i\in \sv$) we will use the notation $E^{m}_{n}(X)$ for   $\inli_i E^{m}_{n}(X_{i})$. 

\begin{theo}\label{tfunct}
The following statements are valid under the assumptions made above for any pro-scheme $S$ and any $d\ge 0$.
\begin{enumerate}
\item\label{iadjgm} For any 
 $m,n\in \z$ and $E\in \obj \gm$ we have $E^{n}_{m}(S)\cong \psigm(E)^{n}_{m}(S)$ (see the notation of \S\ref{sht}). In particular, if $\psigm(E)\in \sht^{t^T=0}$ then $E^{n}_{m}(S)\cong H^n_{Nis}(S, \pi^0_{m}(\psigm(E)))$. 

Moreover, $\pi^0_{0}(\psigm(E))$ is isomorphic to the sheafification of the presheaf $X\mapsto E^0_0(X)$.

\item\label{iw} 
Assume 
 that $\cp$ is  a set of objects of $\shgm$. Denote by $\shccp$ the triangulated subcategory of $\shgm$ densely generated by $\cp$; $\shcp\supset \shccp$ will denote the localizing subcategory of $\shgm$ generated by $\cp$, and $\gdcp$ will denote the colocalizing category cogenerated by 
$c_{\gm}(\cp)$ in $\gdbgm$. 
 Then 
 there exists a cosmashing weight structure $w^{\cp}$ on $\gdcp$ such that $\gdcp{}_{w^{\cp}\le 0}={}^{\perp_{\gdcp}}(\cup_{i> 0}
\cp[-i])$.

\item\label{iprot}
Assume that  all elements of $\cp$ are compact in $\shgm$. 
Then there exists   a $t$-structure $t^{\cp}$ on  $\shgm$ 
such that $\shgm{}^{t^{\cp} \ge 0}=(\cup_{i> 0}\cp[i])^{\perp_{\shgm}}$.\footnote{This equality together with the corresponding one for $\cp$ mean that $w^{\cp}$ is cogenerated by $c(\cp)$ and $t^{\cp}$ is generated by $\cp$ in the sense of  Remark \ref{rts}(4).  
 Obviously, these conditions determine $w^{\cp}$ and $t^{\cp}$ completely.} Moreover, $t^{\cp}$ restricts to a $t$-structure $t'^{\cp}$ on the category $\shcp$; we have $\shgm{}^{t^{\cp}\le 0}=\shcp{}^{t'^{\cp}\le 0}$, whereas  $\shgm{}^{t^{\cp}\ge 0}$ is the class of all extensions of elements of $\shcp{}^{t^{\cp}\le 0}$ by that of $P$, 
where $P=\{N\in \obj \shgm:\ \psigm(N)=0\}$. 

Furthermore, if for any $Y\in \obj \shgm$ 
we will write $\Phi^{\cp}(-,Y)$ 
for the extension  via the method mentioned in Proposition \ref{pextc} of  the functor represented by $Y$ from $c_{\gm}(\shccp)\subset \gdcp$ to $\gdcp$, then $\Phi^{\cp}(-,-)$ is a nice duality $\gdcp\opp\times \shgm\to \ab$, and $w^{\cp}$ is $\Phi^{\cp}$-orthogonal to  $t^{\cp}$.

\item\label{iconnfun} 
If $\cp=\{\omm(X_{+}\{i\}):\ X\in \sv,\ i\in \z \}\subset \obj \shgm$ then the functor  $\phgdbm$ restricts to an   
 exact functor $\phgdm:\gdt\to \gdcp$ that (essentially) coincides with $\phgm$ on $\shcp$ and respects products (here we identify an object of $\sht$ with its image in $\gdtb$ via the corresponding embedding $c_T:\sht\to \gdt$).  
Moreover,  $\phgdm$ is right weight-exact (with respect to $w^T$ and $w^{\cp}$);  thus  for any 
 $r\in \z$ we have $\ommd(S_+)\{r\}\in \gdcp_{w^{\cp}\ge 0}$,  where $\ommd(S_+)\{r\}$ denotes $\phgdm(\omdt(S_+)\{r\})$. 

\item\label{ihtpwe} Assume in addition that 
  the following {\it  homotopy compatibility}  (cf. 
 Proposition 3.2.13 of \cite{bondegl}) condition is fulfilled:  for $X=\prli X_i$ being the spectrum of any function field over $k$ (for $X_i\in \sv$), any $Y\in \sv$, $r,r'\in \z$, and $j>0$ we have $\omm(Y_+)\{r'\})_r^{j}(X)=\ns$.

Then the functor $\phgdm$ is weight-exact with respect to $w^T$ and 
 $w^{\cp}$; hence $S$ is {\it of $\gm$-dimension at most $d$} (i.e., $\ommd(S_+)\in \gdcp_{w^{\cp}\ge d}$) whenever $S$ is of $T$-cohomological dimension at most $d$.\footnote{Certainly, the notion of $\gm$-dimension depends on $\cp$ also; however, $\cp$ will be fixed for all our choices of  $\gm$.} 

Moreover, 
the natural $\gm$-versions of parts II.\ref{ihrtt}--\ref{itwss}
 and IV.\ref{idst1}--\ref{isplc}  
 of Theorem \ref{tshtt} are fulfilled.
Furthermore, 
 $S$ being of $\gm$-dimension at most $d$  is equivalent to 
the $\gm$-versions of conditions  III.\ref{ied}--\ref{iet4} of the theorem.

\item\label{ihtpte} Assume in addition (to the assumptions of the previous assertion) that all elements of $\cp=\{\omm(X_{+}\{i\})\}$ are compact. Then for $M\in \obj \shgm$ we have $M\in \shcp{}^{t'^{\cp}\ge 0}$ (resp. $M\in \shcp{}^{t'^{\cp}\le 0}$) if and only if $M^n_m(\spe K)=\ns$ for all function fields $K/k$, $m\in \z$, and $n< 0$ (resp. $n> 0$).

Moreover, the functor $\psigm$  
 is $t$-exact with respect to $t^{\cp}$ and $t^T$, respectively. 

\item\label{itwist} 
Suppose that we are given a bi-functor $\otimes: \pshtp\times \gm\to \gm$ as in Proposition \ref{proobj}(\ref{itenspr}) that turns $\phwgm$ into a  (lax) 
$\pshtp$-module functor, i.e., for any  
$M_1,M_2\in \obj \pshtp$ 
we have $M_1\otimes \phwgm(M_2)\cong \phwgm(M_1\wedge M_2)$. 
Then this statement can be applied for $I=\{T^i[-i]:\ i\in \z\}$ (with the operation induced by $\wedge_{\sht}$); respectively, the corresponding functors 
$\{i\}_{\gm}$ are Quillen auto-equivalences.
Moreover, the functors $\{i\}_{\gdbgm}$ restrict to $w^{\cp}$-exact auto-equivalences $\{i\}_{\gdcp}$ of $\gdcp$; thus the homotopy compatibility condition is equivalent to  $\inli \shgm(\omm(X_{i+}\{r\}),\omm(Y_+[j]))=\ns$ for any $X_i,Y$, $r$ and $j$ as in assertion \ref{ihtpwe}. 

Furthermore, if all elements of $\cp$ are compact then the functors $\{i\}_{\shgm}$ are $t$-exact (with respect to $t^{\cp}$).

\item\label{idimgm} 
Assume that  $\phwgm$ is a 
$\pshtp$-module functor (as above), all elements of $\cp$ are compact, and the homotopy compatibility condition is fulfilled. Then 
 $S$ is of  $\gm$-dimension at most $d$ if and only if for any $N\in \obj \shcp{}^{\tcp=0}$ and $\pi=\pi^0_0(\psigm(N))$ we have $H^n_{Nis}(\pi,X)=\ns$ for all $N>d$ and $X\in \sv$. Moreover, these conditions are equivalent to either of the $\gm$-versions of conditions III.\ref{iet2p}--\ref{iett} of Theorem \ref{tshtt}.
\end{enumerate}
\end{theo}
\begin{proof}
\ref{iadjgm}.  
The first and the third part of the assertion are obvious adjunction statements; to obtain the "in particular" part one should combine the first part with Proposition \ref{pshtt}(\ref{ihrtc}).

\ref{iw}. Immediate from Theorem \ref{tnews}(I, II.2).

\ref{iprot}.  
All of the statements in question would follow from Corollary 5.4.1 and Proposition 1.3.4(5)  of \cite{bpure} (cf. the proof of  Proposition \ref{pdualsh}) if one  replaces $P$ by $P'=(\cup_{i\in \z}\cp[i])\perpp$ in its formulation. Thus it remains to prove that $P'=P$.

We recall that the category $\sht$ is compactly generated by the set $G=\{\omt(X)\{i\}:\ X\in \sv,\ i\in \z\}$.
Hence  applying Lemma \ref{locat}(1) and the adjunction we obtain that $P'$  equals the class $\{M\in \obj \shgm: 
(\cup_{i\in \z}G[i])\perp \psigm(M)\}$. Hence $P'=P$ indeed.

\ref{iconnfun}. Proposition \ref{proobj}(\ref{ifunpro},\ref{ifunprotr}) implies that the category $\gdtb\supset \gdt$ is equivalent to $\gdpshtp$. Thus it also gives the existence of a functor $\gdtb\to \gdbgm$ that essentially coincides with $\phgm$ on $\sht\subset \gdtb$, exact and  respects products. The latter property of the functor immediately implies that this functor restricts to a functor $\gdt\subset \shtc$; certainly, this restriction is also exact and respects products.
It is also right weight-exact according to Theorem \ref{tnews}(V), and it remains to recall that  $\omdt(S_+)\{r\}\in \gdt_{w^{T}\ge 0}$. 

\ref{ihtpwe}. 
Since $\phgdm$ is right weight-exact according to  the previous assertion, 
 to verify its weight-exactness 
 we should check that $\phgdm(\gdt_{w^T\le 0})\subset \gdcp_{w^{\cp}\le 0}$.

Applying (the dual to) Corollary 2.5.2(II.2) of \cite{bpure} (along with Theorem \ref{tshtt}(II.\ref{ihrtt})) we obtain the following: we have $\phgdm(\gdt_{w^T\le 0})\subset \gdcp_{w^{\cp}\le 0}$ whenever   $\phgdm(\omdt(\spe (K)\{r\})\in \gdcp_{w^{\cp}\le 0}$ for any function field $K/k$ and any $r\in \z$. Thus we should check that $\omm(\spe (K)_+)\{r\}\perp \omm (Y_+)\{r'\}[i]$ for any $Y$, 
$r'$, and $i$ as in the formulation of our assertion. Thus it remains to apply Proposition \ref{proobj}(\ref{ifunprotr},\ref{icoco}).

It certainly follows that $S$ is  of $\gm$-dimension at most $d$ whenever $S$ is of $T$-cohomological dimension at most $d$.
We also obtain the $\gm$-versions of parts II.\ref{itopen}--\ref{itwss}\footnote{Actually, the 
$\gdcp{}^{op}$-functoriality of the corresponding weight spectral sequences is provided by Proposition \ref{pwss}(3).} and  IV.\ref{idst1}--\ref{isplc} of Theorem \ref{tshtt}. 

 The proof of the $\gdcp$-version of part II.\ref{ihrtt} of that theorem (i.e., that $\gdcp_{w^{\cp}=0}$ equals the Karoubi-closure of the class of all products of objects of the form $\ommd(X_+)\{r\}$ for $r\in \z$ and $X$ being a function field over $k$)  
 is quite similar to that of Theorem \ref{tgw}(\ref{iwgh}).

The equivalence of $S$ being of $\gm$-dimension at most $d$ to the conditions in question is also similar to the $\sh$-version of this statement; see (the proof of) Theorem \ref{tds}(III) and Remark \ref{rds}(2). 

\ref{ihtpte}. The proof is quite similar to the arguments of \cite[\S3.3]{bondegl}; we sketch it here briefly. First we take the cohomological functor $\shgm(\phgm(-),M)$ from $\gdt$ 
 to $\ab$ and note that Theorem \ref{tshtt}(II.\ref{itpost},\ref{itwss})  
  provides us with a spectral sequence $$\coprod_{x\in X^p}M^{q}_{m+p}(x) \implies M^{q+p}_m(X)$$ for any $m\in \z$ and $X\in \sv$ (where $X^p$ denotes the set of points of codimension $p$ in $X$). We immediately obtain the equivalence in question for the condition $M\in \shcp{}^{t'^{\cp}\ge 0}$ (so, homotopy compatibility is not needed for this equivalence statement). This spectral sequence argument also implies that for any non-zero $M\in \obj \shcp$ there exists a triple $(K,i,j)$ such that $M^n_m(\spe K)\neq \ns$ (where $K$ is a function field over $k$, and $n,m\in \z$). 

Now denote by $L$ the class of all $N\in \obj \shcp$ such that $N^n_m(\spe K)=\ns$ for all function fields $K/k$, $m\in \z$, and  $n> 0$. We should check that $L=\shcp{}^{t^{\cp}\le 0}$. 
Homotopy compatibility certainly implies 
 that $\cp\subset L$. Next, Remark \ref{rts}(4) says that $\shgm{}^{t^{\cp}\le 0}$ equals the smallest extension-closed subclass of $\obj \shgm$ that contains 
$\cp[i]$ for all $i\ge 0$; thus  
$\shgm{}^{t^{\cp}\le 0}\subset L$ (note that $L[1]\subset L$).

Now we should prove for a fixed $M\in L$ that $M\in \shcp{}^{t^{\cp}\le 0}$. Consider the $t^{\cp}$-decomposition triangle 
\begin{equation}\label{etdm}
M^{t^{\cp}\le 0}\to M\to M^{t^{\cp}\ge 1} \to M^{t^{\cp}\le 0}[1] \end{equation}  for $M$. According to the non-vanishing statement proved above, it suffices to verify for $M'=M^{t^{\cp}\ge 1}$ that $M'^n_m(\spe K)= \ns$ for any   function field over $K/k$ and $n,m\in \z$.  The latter statement is valid for $n\le 0$ by the property of $t$ that we have proved above, and one easily gets it for $n>0$ 
using the long exact sequence obtained by applying the functor $E\mapsto E^n_m(\spe K)$ to the triangle (\ref{etdm}) (one should  recall that $(M^{t^{\cp}\le 0})^n_m=\ns $ for $n>0$).

Lastly, the $t$-exactness of $\psigm$ follows from the first part of the assertion (via a simple adjunction argument).

\ref{itwist}. The first part of the assertion is obvious (recall that the tensor product in $\sht$ is defined by means of cofibrant replacements in $\pshtp$); it certainly follows that the homotopy compatibility condition is equivalent to its restriction described in the formulation. Thus all the functors $\{ i\}_{\gdbgm}$ are auto-equivalences of $\gdbgm$, and for any $i\in \z$ the functor  $\{-i\}_{\gdbgm}$ is essentially inverse to  $\{ i\}_{\gdbgm}$. Since all $\{ i\}_{\gm}$ are essentially bijective on $\cp$, we obtain that the functors $\{ i\}_{\gdbgm}$ restrict to exact auto-equivalences of $\gdcp$. These restrictions are right weight-exact according to 
Theorem \ref{tnews}(V). Next, Proposition \ref{pbw}(\ref{ilrwe}) implies that  these functors $\{ i\}_{\gdcp}$ are also left weight-exact. 
The proof of the "furthermore" part of the assertion is rather similar; cf. Proposition 1.2.6(II.2) of \cite{bpure}.

\ref{idimgm}. 
Since $\phgdm$ is right weight-exact according to part \ref{iconnfun}, we obtain that $S$ is of $\gm$-dimension at most $d$ if and only if $\ommd(S_+)\in \gdcp_{\wcp\le d}$.
Thus arguing similarly to the proof of Theorem \ref{tds}(III) 
 one can easily verify that   $S$ is of $\gm$-dimension at most $d$ if and only if $N^n_0(S)=\ns$ for all $N\in \shcp{}^{\tcp=0}$ and $n>d$. Hence it remains to combine assertion \ref{iadjgm}  with the $t$-exactness of $\psigm$ provided by part \ref{ihtpte} to obtain the first part of our assertion.

The proof of the "moreover" equivalence statement is similar to the arguments in the proof of Theorem \ref{tshtt}(III).
 \end{proof}

\begin{rema}\label{rfunctpr}

1.  
Probably, all existing models for $\sh$ and for $\sht$ are proper and Quillen equivalent (though it is 
 far from being easy to find references for this claim). 
 It follows that the corresponding versions of $\gd$ and $\gdt$ 
 are equivalent to each other. In particular, Proposition \ref{pinfgy} and its $\gdt$-version are valid for any choice of models for $\sh$ and $\sht$, respectively.


2. The choice of $\pshtp$ as a "starting model category" is certainly not the only one possible. In particular, for any $\gm$ 
satisfying the assumptions of part \ref{idimgm} of our theorem the obvious re-formulations of all the assertions are fulfilled for $\phgm$ replaced by any tensor left Quillen functor $\phi_{\gm,\gm'}$ such that the composition $\phi_{\gm'}= \phi_{\gm,\gm'}\circ \phgm$ satisfies the $\gm'$-versions of our assumptions.

Moreover, one may also "start from a model for $\sh$" (and this model can be a monoidal model category also). This would 
 give an alternative proof of some of the parts of Theorem \ref{tshtt}. Furthermore, the corresponding functor $\phishmot:\gd\to \gdm$ (cf. Proposition \ref{pmott} below) is weight-exact and $\psishmot:\dmk\to \sh$ is $t$-exact.

3. In "most of the interesting cases" $\gm$ is a 
monoidal model category  and  $\phgm$ is a monoidal functor.
\end{rema}

Now we apply 
Theorem \ref{tfunct} to the study of Voevodsky motives. We start with a few definitions.

\begin{defi}\label{dmotprim}
\begin{enumerate}
\item\label{ihitr}
We will use the notation $\smc$ for the category of {\it smooth correspondences} (see \cite{1}).

We will write $\hitr$ the full subcategory of  $\adfu(\smc\opp,\ab)$ consisting of those functors that are homotopy invariant (i.e., $F(X)\cong F(X\times \afo)$ for any $X\in \sv$) and whose restrictions to $\sm$ give Nisnevich sheaves; objects of $\hitr$ are called {\it homotopy invariant sheaves with transfers}. 

\item\label{imotdim}
We will say that a pro-scheme $S$ is of {\it motivic dimension at most $d$}  if $H_{Nis}^n(S,N)=\ns$ (see (\ref{enisc})) for any $n>d$ and $N\in \obj \hitr$. 

We will call pro-schemes of motivic dimension at most $0$ {\it motivic points}.

\item\label{iprim}
If $k$ is infinite then a pro-scheme will be called primitive if all
of its connected components are affine (essentially smooth) $k$-schemes and their coordinate rings
$R_j$ satisfy the following primitivity criterion: for any $n>0$ 
every polynomial in $R_j[X_1,\dots,X_n]$ whose coefficients generate
$R_j$ as an ideal over itself, represents an $R_j$-unit.

If $k$ is finite, 
then we will  (by an abuse of notation, in this paper) call a pro-scheme primitive whenever it is semi-local in the sense of  Remark \ref{rsemiloc} (i.e., if all of its connected components are 
affine essentially smooth semi-local). 

\item\label{imodel} We will use the notation $\pdmk$ for the symmetric monoidal model category  of Tate spectra built out 
of complexes of sheaves with transfers over $k$ 
considered in 
\S5 of \cite{degmod} (cf. Example 7.15 of \cite{cdhom} for the description of the corresponding model structure). 
\end{enumerate}
\end{defi}

\begin{pr}\label{pmott}
Let $S$ be a pro-scheme and $d\ge 0$. 
\begin{enumerate}
\item\label{imotcat}
The category $\dmk=\ho(\pdmk)$ is the "big" category of Voevodsky motives over $k$ (as described in \cite[\S5]{degmod}),\footnote{Recall that 
\S5.11 of \cite{degmod} gives the existence of a full exact embedding of the Voevodsky category $\dmgm(k)$  (see \cite{1}) into $ \dmk$.} 
and there exists a monoidal left Quillen functor $\phwmot:\pshtp\to \pdmk$  (cf. Definition 4.2.20 of \cite{hovey})
satisfying all the assumptions of Theorem \ref{tfunct}(\ref{idimgm}) (here we set $M\otimes N=\phwgm(M)\otimes_{\pdmk} N$ for any $M\in \obj \pshtp$ and $N\in \obj \pdmk)$. Moreover, the functor $\phimot=\ho(\phwmot): \sht\to \dmk$ sends $\omt(X_+)\{i\}$ into the 
object $\mg(X)\{i\}=\mg(X)(i)[i]$ for any $X\in \sv$ and $i\in \z$. 

\item\label{itmot} The $t$-structure $\tmot$ corresponding to $\cp=\{\mg(X)\{i\}:\ X\in \sv,\ i\in \z  \}$ coincides with the homotopy  $t$-structure described in Corollary 5.14 of \cite{degmod}. 


\item\label{icalch} 
The restriction $\pi^0_{mot}$ of the composition $\pi^0_0\circ \psimot$ to the heart of  $\tmot$  factorizes into the composition of an exact essentially surjective functor $\hrtmot\to\hitr$ with the natural exact embedding $\hitr\to \shi$.


\item\label{imotw}
$S$ is of motivic dimension at most $d$ if and only if  
 $\mgd(S) \in \gdm_{[0,d]}$; here $\gdm\subset \ho(\proo-\pdmk)$ is the motivic version of $\gdcp$ corresponding to $\cp=\phimot(\omt(X_+)\{i\})$ for $X\in \sv$, $i\in \z$ (whose objects we will call {\it comotives}) and we consider the weight structure $\wmot$ coming from this $\cp$, 
 whereas $\mgd(S)=\phimotgd(\omdt(S_+))$  denotes the comotif of $S$.

\item\label{imotr}
The functor $\phimotgd$ (i.e., the version of $\phgdm$ corresponding to $\phwmot$)  from $\gdt$ into the category $\gdm$ is weight-exact with respect to $w^T$ and  $\wmot$. Thus if 
 $S$ is of $T$-dimension at most $d$ 
 then it is of motivic dimension at most $d$ also.  

Moreover, $S$ is of motivic dimension at most $d$ if and only if $\mgd(S)$  is a retract of $\mgd(S_0)$ for any open sub-pro-scheme $S_0$
of $S$ with complement of codimension more than $d$.

\item\label{imdiv}
Assume that  $S$ is a motivic point, $Z$ is its closed subscheme (everywhere) of codimension $j$ that is a motivic point also.

 Then for $S_0=S\setminus Z$ we have
$\mgd(S_0)\cong \mgd(S) \bigoplus \mgd(Z)\{j\}[j-1]$.
Furthermore, $S_0$ is a motivic point (under these assumptions) if $j=1$. 

\item \label{imprim} 
If $S$ is primitive then $S$ is a motivic point.

\item \label{ires}
 Let $K$ be a function field  over $k$ and let $K'$ be the residue field for a geometric valuation $v$ of $K$ of rank $r$.
Then $\mgd(\spe (K')_+)\{r\}$ is a retract of $\mgd(\spe (K)_+)$. 

\end{enumerate}
\end{pr}
\begin{proof}
\ref{imotcat}. Certainly, $\pdmk$ is simplicial; it is proper according to Proposition 7.13 of \cite{cdhom}  (cf. Example 7.15 of ibid.). The calculation of $\ho(\pdmk)$ is contained in \S5 of \cite{degmod} (see Proposition 5.10 of loc. cit.). The existence of the Quillen adjunction in question (with $\phwmot$ being monoidal; cf. Remark \ref{rtens}(3)) is provided by  Example 2.2.6 of \cite{degorient} (see \S11.2.16 of \cite{cd} for more detail). 
Moreover, the objects of the motivic version of $\cp$ are of the form $\mg(X)\{i\}=\mg(X)(i)[i]$ (for  $X\in \sv$  and $i\in \z$); hence they are compact in $\dmk$. 

Lastly, we should verify the homotopy compatibility property. According to Theorem \ref{tfunct}(\ref{itwist}), it suffices to prove that 
$\inli \dmk(\mg(X_{i})\{r\},\mg(Y)[j])=\ns$
 for any $X_i,Y\in\sv$ such that $\inli X_i=X$ is (the spectrum of) a $k$-function field, 
$r\in \z$, and $j\ge 0$. The latter statement is immediate from 
 Theorem 5.11 of ibid. (cf. Theorem 3.7 of ibid.).

\ref{itmot}.  Immediate from (the last formula in) \S5.7 of \cite{degmod}.

\ref{icalch}. 
 The functor $\psimot$ is $t$-exact (with respect to $\tmot$ and $t^T$) according to Theorem \ref{tfunct}(\ref{ihtpte}) (combined with the previous assertions). Applying Proposition \ref{pshtt}(\ref{ihrtnis})  we obtain that for any $E\in \dmk^{\tmot=0}$ the sheaf $\pi^0_{mot}(E)$ sends any $U\in \sv$ into $\dmk(\mg(U),E)$. Hence $\pi^0_{mot}(E)$ belongs to $\obj \hitr$ and all homotopy invariant sheaves with transfers 
can be obtained this way according to Corollary 5.14 of \cite{degmod} (here we use the fact that the corresponding functor $\Omega^{\infty}$ is $t$-exact; it is surjective on objects as being adjoint to a full embedding $\sigma^{\infty}$; see formula (5.7.a) of ibid.).

\ref{imotw}. 
 Immediate from Theorem \ref{tfunct}(\ref{idimgm})  (that can be applied according to assertion \ref{imotcat}) combined with assertion \ref{icalch}.


\ref{imotr}. The weight-exactness of $\phimotgd$ is given by assertion \ref{imotcat} also.
Moreover, the previous assertion 
implies that 
$S$ if of motivic dimension at most $d$ if 
and only if $S$ is of $\gm$-dimension at most $d$ for $\gm=\pdmk$. Hence we obtain the remaining parts of the assertion (see Theorems \ref{tfunct}(\ref{ihtpwe}) 
  and \ref{tshtt}(III.\ref{ietcompl}, respectively).

\ref{imdiv}. Recall that the comotivic version of the formula (\ref{edec}) is fulfilled according to assertion \ref{imotr}. Thus it remains to note that the normal bundle to $Z$ in $S$ is trivial; the latter follows  the fact that $\picz$ is $\sht$-representable (cf.  Remark \ref{rpicz}).

\ref{imprim}. In the case of  infinite $k$ the statement follows from Theorem 4.19 of \cite{walker}. 
If $k$ is finite then $S$ is semi-local (by our convention); thus we may apply Corollary 4.18 of \cite{3} instead. 

\ref{ires}. The previous assertion allows us to apply the argument used for the proof Corollary \ref{cds}(1) in this setting without any difficulty (so, we don't have to assume that $k$ is infinite). 

\end{proof}

\begin{rema}\label{roldger}
1. Certainly, combining Theorem \ref{tfunct} with our proposition one can easily prove the obvious comotivic analogues of (the remaining) parts of Theorem \ref{tshtt}. In particular, one obtains several  equivalent definitions of motivic dimension; note that this notion was not discussed in \cite{bger}. 

2. 
Recall  that motivic  and   \'etale 
 cohomological functors naturally factor through $\dmgm(k)$. Moreover, if we fix an embedding of $k$ into $\com$ (if it exists) then morphic (see Theorem 5.1 of \cite{chu}) and  singular cohomology (along with its $\q$-linear version with values in the category of mixed Hodge structures) factors through $\dmgm(k)$ also. Another example for $k$ that can be embedded into $\com$ (so, of characteristic zero and cardinality at most continuum) is given
	by the so-called  mixed  cohomology  (see 
 \S2.3 of \cite{hu}). 
Thus the results above can be applied to all these cohomology theories; this includes "motivic" functoriality of the corresponding (generalized coniveau)  filtrations and spectral sequences. 

Note also that (some versions of) \'etale cohomology can certainly be factored through the corresponding categories of Galois modules.\footnote{Since we need the target of cohomology to be an AB5 category, one should make some effort to embed the values of these cohomology theories into abelian categories satisfying this axiom; in particular, one should consider some sort of ind-mixed Hodge structures here.}
 This certainly makes  the applications of our direct summand results to these cohomology theories (along  with singular and mixed cohomology) "more interesting". 

3.  Let us discuss other distinction of the results of this subsection from that of  
 \cite{bger}.

The properties of the category of comotives constructed in ibid. (we will use the notation $\gdo$ for it in this remark) by are rather similar to that of $\gdm$. One can carry over the arguments of the previous sections to $\gdo$ without much difficulty; in particular, it is no problem to replace the model category formalism of the current paper by the differential graded methods described in (\S5 of)  ibid. 
However, the arguments applied in ibid. were significantly distinct from the ones of the current paper. 
In particular, $\gdo$-version of Proposition \ref{pinfgy} was established in ibid. using a "purely triangulated" argument (relying on countable homotopy limits in triangulated categories; cf. \S1.6 of \cite{neebook}).
This argument 
 only 
 worked under the assumption that $k$ is countable. 

Moreover, the weight structure 
on $\gdo$ was  constructed "starting from generators of the heart" (cf. Remark \ref{rwgen}), whereas the negativity (see Definition \ref{dwso}(\ref{id6})) of $\hw$ was established only under the same restriction of $k$. 
Respectively, the  computation of the heart of 
 in ibid. was 
 "automatic", whereas our one (cf. Theorem \ref{tgw}(\ref{iwgh})  that follows Theorem \ref{tnews}(III)) is 
 quite non-trivial (and the author was quite amazed to find out that this result is parallel to its analogue from \cite{bger}).
Similarly, 
the verification of the fact that $w\perp_{\Phi} t$ (carried over in \cite{bpure}) requires more effort in our setting (cf. Proposition \ref{pdualsh}(II)) than  the 
proof of Proposition 4.5.1(2) of ibid.; one may say that the current version of the proof is "much more conceptual".

4. Recall that the Gysin 
 distinguished triangle for motives has the form
\begin{equation}\label{egysm}
\mg(X\setminus Z)\to \mg(X)\to \mg(Z)\lan c\ra
 \to \mg(X\setminus Z)[1]
\end{equation}
for any closed embedding $Z\to X$  of codimension $c$ of smooth varieties. We conjecture that this statement carries over to pro-schemes in the obvious way (i.e., one does not have to assume that the normal bundle is trivial to compute 
the cone of $\mg(X\setminus Z)\to \mg(X)$); this would give an "improvement" of part \ref{imdiv} of our proposition. The restriction of this conjecture to the case of a countable $k$ was established in ibid. using the aforementioned "triangulated countable homotopy limit" methods. 
Note also that we certainly have the corresponding long exact sequences for any cohomology of pro-schemes "extended from" $\dmgm \subset \gdm$.

5. One can certainly "iterate"  part \ref{imdiv} of our proposition to obtain that  the complement to a motivic point of a divisor with smooth crossings (cf. Theorem \ref{tds}(II.1)) such that all these crossings are motivic points also is a motivic point also; cf. Remark \ref{rshtpl}(1). 

6. It is well known that the 
 correspondence $X\mapsto \mg(X),\ \spv\to\dm$, factors through a natural  embedding of the category of Chow motives $\chow(k)$ into $\dmk$ (see  Corollary 6.7.3 of \cite{bev}). Moreover, loc. cit. also implies that $\chow(k)$ 
 is a negative subcategory of $\dmk$ (see Definition \ref{dwso}(\ref{id6})). This gave the possibility to construct certain Chow weight structures on several "versions" of $\dm(-)$ (see \S2.1 and \S3.1 of \cite{bokum}); the hearts of these weight structures contain the category $\chow(-)$ (or its "modifications").  One can also dualize these arguments; in particular, this gives a weight structure on the $\zop$-linear version of $\gdm$ (cf. Proposition \ref{plocoeffsht} below)  whose heart consists of retracts of $\gdm$-products of $\zop$-linear Chow motives (here $\zop=\z$ if $p=\cha k=0$, and one can apply the dual either to Theorem 4.5.2 of \cite{bws} or to Corollary 2.3.1 of \cite{bsnew} to construct this weight structure).

On the other hand, we note that there cannot exist any weight structure whose heart contains $\om(\spv)\subset \obj \shc$ or  $\omt(\spv)\subset \obj \shtc$ since these subcategories are not negative.
Indeed,  Morel's morphism $\eta\in S^{00}\{2\}[1]\to S^{00}\{1\}[1]$ (see Remark \ref{rdwo}(1) below) is well-known to be non-zero for any $k$ (see Remark 6.4.3 of \cite{morintao}); this certainly implies that $\omt(\spv)$ 
 is not negative in $\shtc$. It follows that $\om(\spv) $ is not negative $ \obj \shc$; recall that  $\shtc$ can be described as the "twist-stabilization" of $\shc$. 

It is an interesting question whether there exist Chow weight structures on the categories $\shtpl$ and $\gdtpl$ (see Proposition \ref{prpl} below); 
this depends on the vanishing of the $\tau$-positive parts of 
  certain 
	stable homotopy groups of spheres over all (finitely generated) extensions of $k$. 

7. Certainly, one can apply the splitting provided by parts \ref{imdiv}--\ref{ires} of our proposition to the extension to $\gdm$ of a cohomological functor $H'$ from $\dmgm$ into an AB5 category $\au$. Thus (in the notation of our proposition) for any presentations of $\spe K$ and $\spe K'$ as $\prli X_i$ and $\prli X'_i$ respectively, where $X_i,X'_i\in \sv$, we obtain that $\inli H'(\mg(X'_i)\{r\})$ is a retract of $\inli H(\mg(X_i))$ (see Proposition \ref{pextc}(I.2)). 

8. One can make the connecting functor $\phimot$ "more explicit" using Theorem 58 of \cite{roe} and  Theorem 5.8 of \cite{hoysteenr}; one only has to 
 consider motives with $\zop$-coefficients if $p>0$  (cf. \S\ref{sshinvp} below). 
\end{rema}

\subsection{
Comparing motivic dimensions with  $T$-ones;  (weakly) orientable cohomology}\label{swo}

A disadvantage of the category $\dmk$ is that there are plenty of important cohomology theories that do not factor through it. Still it turns out that the notion of motivic dimension is relevant for most of these theories. 

We will prove several statements related to this claim. We start from the following notion of orientability.

 \begin{defi}\label{dwo}
1. We will say that a spectrum $E\in \sht^{t^T=0}$ is {\it orientable} if 
 it belongs to the essential image of the functor $\psimot$. 

2. We will say that $E\in \obj \sht$ is {\it weakly orientable} if  $E^{t^T=i}$ is orientable for any $i\in \z$.

More generally, we will say that $E$ is  {\it very  weakly orientable} if  for any $i\in \z$ the spectrum $E^{t^T=i}$ possesses an 
increasing  filtration $F_N,\ N\in \z$, such that 
$F_N(E^{t^T=i})= E^{t^T=i}$ for $N$ large enough, 
 the $\hrtt$-factors of the filtration are orientable, and for any $r\ge 0$ there exists $N^{i,r}\in \z$ such that for any smooth variety $X/k$ of dimension at most $r$ we have $(F_{N^{i,r}}(E^{t^T=i}))^0_0(X)=\ns$.

\end{defi}

\begin{rema}\label{rdwo}
1. Now we relate our definition of orientability to Morel's one.

Recall that for any $E\in \sht^{t^T=0}$ the sequence $(E_n)=(\pi^0_{-n}(E))$ is functorially endowed with a  system of morphisms $\eta_E=(\eta_n:\, E_n\to E_{n-1})$ (for $n$ running over $\z$). 
Recall that the latter 
 is defined in the terms of Morel's algebraic Hopf map. Following \S6.2 of \cite{morintao} we can define it as the image in $\sht$ of the $\dopsh$-morphism of discrete simplicial sheaves corresponding to the morphism $(\af^2\setminus \ns,\ (1,1))\to (\p^1,\ (1:1)):\ (x,y)\mapsto (x:y)$ of pointed smooth varieties.
This gives a morphism $S^{00}\{2\}[1]\to S^{00}\{1\}[1]$ (see loc. cit.; the choice of the corresponding isomorphisms is irrelevant for our purposes). 
 Tensoring it by   $S^{00}\{-1\}[-1]$ one obtains a morphism $\eta: S^{00}\{1\}[1]\to S^{00}$ (see Lemma 6.2.3 of loc. cit.). Thus $\eta\otimes \id_E$ gives a morphism $E\to E\{-1\}$, and we define $\eta_*$ as the induced morphisms on $\pi^0_{*}(E)$. 

So, 
we recall the following (main) Theorem 1.3.4 of  \cite {degorient}: $E$ is orientable if and only if it $\eta_E=0$; thus, our definition of orientability is equivalent to the "nicer"  Definition 1.2.7 of ibid.

2. 
 This argument 
 relies on the following simple observation: since $\psimot$ is $t$-exact (with respect to $\tmot$ and $t^T$), a spectrum $E\in \sht^{t^T=0}$ belongs to its essential image if and only if it is isomorphic to $\psimot(D)$ for some $D\in \dmk^{\tmot=0}$. Moreover, this $D$ is functorially determined by $E$. 

3. Certainly, for any $m\in \z$  any object $E$ of $\sht$  is orientable (resp. weakly orientable, resp. very weakly orientable) if and only if $E\{m\}$ is.

4. 
The characterization of orientability in terms of $\eta$ certainly implies that the property of being weakly orientable depends only on the functors $\sht(\omt(-), E[n]):\opa\to \ab$ (for $n\in \z$). 

Similarly, 
 if we fix some filtrations $F_N(E^{t^T=i})$ for all $i\in \z$ then to check whether this filtrations satisfy the conditions in Definition \ref{dwo}(2) it suffices to look at the functors  $\sht(\omt(-), F_N(E^{t^T=i})[n])$ (for all $n\in \z$). 

5. The definition of very weakly orientable theories is rather ad hoc; its main idea is to fit with Theorem 
 15 of \cite{bacons} that we will recall below. 

Now, for any $E\in \sht^{t^T=0}$ part 1 of this remark yields the following  
natural filtration $F_N^{\eta}E$ for $E$ with $F^{\eta}_0E=0$ and orientable functors:  $F^{\eta}_{N}E=\imm(E\{-N\}\to E)$ for $N<0$ with the morphism being the composition of the corresponding $\eta_{E\{-i\}}$, and $F^{\eta}_{N}E=F^{\eta}_{0}E$ for $E>0$. So it could 
make sense to consider those $E$ for which this filtration "converges" (in some sense).

6. Another possible approach is to consider the localizing subcategory of $\sht$ generated by orientable objects of $\sht^{t^T=0}$. Note however that this subcategory appears not to be compactly generated; hence one cannot use Theorem \ref{tfunct} to study it. 


\end{rema}

\begin{pr}\label{ppacycl}
Let $E\in \obj \sht$; suppose that a pro-scheme $S$  if of motivic dimension at most $c$ (for some $c\ge 0$). 

1. Assume that $E$ is very 
weakly orientable. Then for the 
generalized coniveau spectral sequence $T(H,\omdt(S_+))$ for $H$ being the extension to $\gdt$ of the functor $\sht(-,E)$ we have $E_2^{pq}=\ns$ for $p>s$.

2. 
Assume that  $E\in  \sht^{t^T\le 0}$ 
and for any $m\le 0$ we either have $(E^{t^T=m})^j_0(S)=\ns$ for all $j>c-m$ or
 $E^{t^T=m}$ is very weakly orientable.

Then  for any $i>c$ 
we have $E^{i}_0(S)=\ns$. 

3. Assume that $E$ is very weakly orientable, $S$ is a motivic point. Then for any $n\in  \z$ we have $E^n_0(S)\cong (E^{t^T=n})^0_0(S)$.
Moreover, if  $E$ is weakly orientable then the group in question is isomorphic to $\Phi^{mot}(\mgd(S), D)$,\footnote{Certainly, this formula gives the "motivic version of $D^0_0(S)$."} 
where 
  $D^n$ is the "preimage" of $E^{t^T=n}$ in $\dmk^{\tmot=0}$ (i.e., $\psimot(D^n)\cong E^{t^T=n}$, see Remark \ref{rdwo}(2)) and 
$\Phi^{mot}:\gdm{}^{,op}\times \dmk\to \ab$ is the duality provided by Theorem \ref{tfunct}(\ref{iprot}) that corresponds to motives. 

\end{pr}
\begin{proof} 
1. According to Proposition \ref{pdual}, we should check that $(E^{t^T=q})^{p}_0(S)=\ns$ for $p>0$. 
 
It certainly suffices to verify the 
	 vanishing 
	 statement in question for all connected components of $S$ "separately"; thus we may assume that $S$ is of dimension 
	 $r$ for some $r\ge 0$. Then for any smooth variety $X/k$ of dimension $r$ the group $F_{N^{q,r}}(E^{t^T=q})(X')=\ns$ for any $X'$ that is \'etale over $X$; hence 
$H^p_{Nis} (S, \pi^0_0(F_{N^{q,r}}(E^{t^T=q}))
=H^p_{Nis} (X, \pi^0_0(F_{N^{q,r}}(E^{t^T=q}))=\ns$. Hence it suffices to verify that $H^p_{Nis} (X, F')=\ns$ for $F'$ being a factor of the filtration $F_N$ (recall the "stabilization" property of the filtration).
 Since $F'$ is orientable, it remains to apply Proposition \ref{pshtt}(\ref{ihrtc}). 

 2. The generalized coniveau spectral sequence  considered in assertion 1 
 obviously converges to $E^{p+q}_0(S)$ (cf. Corollary \ref{ccompss}(I.2)).
	Recall once again that $(E^{t^T=q})^{p}_0(S)\cong H^p_{Nis} (S,\pi^{q}_0(E))$; hence
	it suffices to verify the vanishing of these groups for any $(p,q)$ such that $p>c-q$. 
	This vanishing is 
	tautological if $(E^{t^T=q})^p_0(S)=\ns$; hence we can assume that $E^{t^T=q}$ is very weakly orientable and apply the previous assertion.
	
3. 
Once again 
 we consider the same generalized spectral sequence as above; 
 recall that it converges to $E^{p+q}_0(S)$. Since $E$ is very weakly orientable, $E_2^{pq}=\ns$ for $p\neq 0$ according to assertion 1. Thus  $E^{n}_0(S)\cong E_2^{0n}\cong (E^{t^T=n})^{0}_0(S) $. 

The second part of the assertion is a simple adjunction statement. 

\end{proof}

Now  we relate motivic dimension to $T$-cohomological dimension and to a collection of cohomology theories.

\begin{theo}\label{tmotdimt}
Assume that  $S$ is a pro-scheme  of motivic  dimension at most $d$ (for some $d\ge 0$), $H$ is the extension to $\gdt$ of a cohomological functor $H'$ from $\shtc$ into $ \au$ (for $\au$ being an AB5 abelian category).

Then the following statements are valid.

I. Assume that $d=0$ and $H'$ is represented by a weakly orientable spectrum.

1. Let $j:S\to S'$ be an open embedding of $S$ into another motivic point. Then $H(\omdt(S'))$ is a retract of  $H(\omdt(S))$. 

2. The Cousin complex $T_H(S)$ (see Proposition \ref{cdscoh}(4)) 
 is $K(\ab)$-isomorphic to $H(\omdt(S_+))$ (considered as a complex put in degree $0$).

II. Assume that $k$ is  of finite $2$-adic cohomological dimension.  Then 
$S$ is of $T$-cohomological dimension at most $d$. 

III. The following cohomological functors are represented by weakly orientable spectra: algebraic cobordism, algebraic $K$-theory, 
semi-topological $K$-theory, 
  semi-topological cobordism (the latter 
	 two theories are defined whenever  an embedding of $k$ into $\com$ is fixed), and real $K$-theory corresponding to an embedding of $k$ into $\re$. 

IV. Suppose that $k$ is embedded into $\com$; denote by $\reco:\sht\to \shtop$ the composition of the natural functor $\sht\to SH(\com)$ with the 
 {\it stable topological realization}  functor $\rco:SH(\com)\to \shtop$  considered in \S A.7 of \cite{papi}. 
 Then for any $N\in \obj \shtop$ the composition functor
$\shtop(-,N)\circ \reco$ is represented by a very weakly orientable object $E_N$ of $\sht$. Moreover, if the functor $\shtop(-,N)$ yields a complex oriented cohomology theory 
 then this $E_N$ is weakly orientable.

\end{theo}
\begin{proof}
I. Since both assertions are formulated in terms of $E$-cohomology of motivic points, applying Proposition \ref{ppacycl}(3) we reduce them to Proposition \ref{pmott}(\ref{imotr}) (cf. also Theorem \ref{tshtt}(IV.\ref{isplc})).

II. Once again, we can assume that $S$ is connected and (so) of $T$-cohomological dimension at most $r$ for some $r\ge 0$. Hence it suffices to prove under this assumptions that of $S$ is also of $T$-dimension at most $r-1$ unless $r\le d$.

So we assume that $r>d$ and denote $r-1$ by $c$. According to Theorem \ref{tshtt}(III) (see condition \ref{iet2p} in it) to prove that  $S$ is of  $T$-cohomological dimension at most $c$ it suffices to verify for any $X\in \sv$ that 
 $\gdt(\omdt(S_+),\omdt(X_+)[r])=\ns$; note that for $E=\omt(X)$ the group in question equals $E^r_0(S)$.

So, it remains to verify the conditions of the previous lemma are fulfilled. Since $S$ is of $T$-cohomological dimension at most $r$, we have
$(E^{t^T=m})^j_0(S)=\ns$ for all $j>r$. Thus it suffices to verify the existence of the corresponding filtration for $m=0$. The existence of this filtration is provided by Theorem 
 15 of \cite{bacons}; note that $E$ is compact and belongs to $\sht^{t^T\le 0}$ (and it is noted in the proof of Theorem 15 of ibid. that the factors of the filtration $\tilde{F}_*(E^{t^T=0})$ are orientable). 

III. All of these functors are well-known to be representable by objects of $\sht$. 
  To obtain 
 their weak orientability  we discuss various notions of orientability.

According to \S3.8 of \cite{paor}, all the 
 cohomological functors in question except semi-topological cobordism 
are oriented (when "restricted to $\opa$"; cf. Remark \ref{rdwo}(4))  in the sense of Definition 3.2 
of ibid. 
As explained in \S1.3 of \cite{papico}, 
 this is equivalent to being oriented in the sense of  Definition 1.2 of ibid. Applying Theorem 2.7 of ibid. we obtain that the corresponding  spectrum $E$ is a module over the Voevodsky's algebraic cobordism spectrum $\mgl$ in the category $\sht$\footnote{Actually, one can also apply the earlier Theorem 4.3 of \cite{vez} to obtain the latter implication.} for all of the theories listed except semi-topological cobordism; the latter one is considered in Remark 5.8 of \cite{heller}. 
Applying Corollary 4.1.7 of \cite{degorient} we obtain that $E$ is weakly orientable.

IV. Certainly, to prove the existence of $E_N$ we should verify that $\reco$ possesses a right adjoint. Now, the existence of the right adjoint to the base change functor $\sht\to \sh(\com)$ is well-known, 
whereas the existence of a right adjoint to $\rco$ is guaranteed by Theorem A.45 of \cite{papi}. 

The remaining statements easily follow from the well-known 
 coherence between the motivic Hopf map and its topological version. 
We recall that the fourth iteration power of the topological $\eta$ is zero (see  Theorem 14.1(i) of \cite{toda62}); 
 thus for all $i\in \z$ for the filtration $F^{\eta}_N(E_N^{t^T=i})$ described in Remark \ref{rdwo}(5) we have $F^{\eta}_{-4}=0$. It certainly follows that $E_N$ is very weakly orientable.

Lastly, if  $N$ represents a complex oriented cohomology theory then the topological $\eta$ gives a zero operation on its values (on topological spaces). Hence the action of the algebraic version of $\eta$ on cohomology of complex varieties is zero also and we obtain the result.

\end{proof}

\begin{rema}\label{rconvdim}

1.  
Note that Proposition \ref{ppacycl} gives a Gersten-type resolution of cohomology of motivic points with coefficients in very weakly orientable spectra.
Thus most of properties of cohomology motivic points  provided by Proposition \ref{pmott} carry over to   a wide range of cohomology theories  that do not factor through $\dmk$ (cf. Remark \ref{roldger}(2)). 

On the other hand, the results for pro-schemes of larger motivic dimensions are less satisfactory (if $k$ is not as in part II of our theorem). We will describe two methods for improving them below.

2. 
We can specify two types of examples for part I.2 of our theorem: (pro-open) embeddings of the generic point of $S$ into a (connected) motivic point $S$, and embeddings into (a motivic point) $S$  of an "open subscheme with nice complement" as mentioned in Remark \ref{roldger}(5).

3. The author suspects that instead of the finiteness of the $2$-adic cohomological dimension condition in part II of our theorem it suffices to assume that $k$ is unorderable (cf. \cite{binf}; a slightly weaker  result  is given by   Theorem \ref{tprimacycl}(I.2)  below). 
For this purpose one can try to "approximate" a pro-scheme $S/k$ by schemes defined over (perfect) subfields of $k$ that are of finite  $2$-adic cohomological dimension. This idea seems to be rather promising in the case of a primitive $S$.

4. If $k$ is formally real then Theorem 33 of \cite{bacons} 
suggests that certain "real cohomology" of $S$ can also be studied to bound its $T$-cohomological dimension. 

 Moreover, Theorem 2.3.1 of \cite{binf} (cf. Lemma 19 of \cite{bacons}) 
 suggests  considering 
the $2$-completion of $\sht$ instead of $\sht$ 
 if $k$ 
 is formally real.

5. In the case $\cha k=0$ one can avoid induction in the proof of 
part 2 of our theorem.

Moreover, instead of the Bachmann's filtration one can use the well-known slice filtration (cf. Theorem 14 and the proof of Theorem 16 of \cite{bacons}, and Theorem 2 of \cite{levconv}).

Furthermore, in the case $\cha k=p>0$ one can still use the slice filtration instead of Bachmann's one if the category $\sht$ will be replaced by $\sht[1/p]$ (cf.  
S\ref{sshinvp} below). 
\end{rema}

Now we recall that the Voevodsky algebraic cobordism spectrum $\mgl$ is a commutative 
ring object in the category $\pshtp$ (see Theorem \ref{tfunct}). 

So we can consider the appropriate category of (right) modules over $\mgl$; cf. 
Example 1.3.1(3) of \cite{bondegl} 
 and the proof of \cite[Proposition 3.10]{cdint}.  
We apply Theorem 4.1(3) 
 of \cite{schwalg} 
  to obtain a model structure on the category of ring objects in $\pshtp$. Moreover, the cofibrant replacement $\mglp$ 
of $\mgl$ with respect to this model structure is a cofibrant object of $\pshtp$. 

 So we consider the model category $\mglmod$ of right $\mglp$-modules in $\pshtp$ that is given by Theorem 4.1(1)  of \cite{schwalg}; recall that the obvious forgetful functor $\forg: \mglmod\to \pshtp$ "detects" fibrations and weak equivalences. 
 $\mglmod$ is certainly stable, and Proposition 7.2.14 of \cite{cd} says that the "free module" functor $\free: \pshtp\to \mglmod$  ($M\mapsto M\wedge \mgl$) 
is a left Quillen functor with 
 $\forg$ being  its right adjoint. 
So we apply Theorem \ref{tfunct} to the functor $\free$; we consider   $\shmgl=\ho(\mglmod)$ and $\gdmgl$.


\begin{pr}\label{pmgl}
\begin{enumerate}
\item\label{icob1}
The couple $(\mglmod,\free)$ fulfils all the assumptions of Theorem \ref{tfunct}(\ref{idimgm}).

\item\label{icob2}
Assume that 
a pro-scheme $S$ is of motivic dimension at most $d$ (for some $d\ge 0$). 
Then $S$ is also {\it of cobordism-dimension at most $d$}, i.e.,   $S$ is of $\gm$-dimension at most $d$ for $\gm=\mglmod$.
\end{enumerate}

\end{pr}
\begin{proof}
1. $\mglmod$ is  a proper model category: 
 its right properness is obvious, whereas the left one follows from Theorem 5.3.1 of \cite{pavl}
(cf. also Proposition 1.10 of \cite{hoveypr}). Next, the functor $\free$ respects the compactness of objects since  its  right adjoint $\forg$ obviously respects coproducts. We have a pairing $\otimes: \pshtp\times \mglmod\to \mglmod$ given by the wedge product; it obviously satisfies all the properties we need. 

It remains to verify the homotopy compatibility property.

So, for $S$ being the spectrum of a function field over $k$,  
any $X\in \sv$, 
and $E=\omt(X_+)\wedge \mgl$ we should check that $(E)^{i}_j(S_+)=\ns$ if $i>0$. 
The latter is equivalent to $ E\in \sht^{t^T\le 0}$; hence it follows from the 
$t$-non positivity of $\omt(X_+)$ (see Proposition \ref{pshtt}(\ref{it1}))  and of $\mgl$ (see Corollary 3.9 of \cite{hoycobord}), along with the fact that $\sht^{t^T\le 0}$ is $\wedge$-closed (see \S1.2.3 of \cite{degorient}). 

2. It certainly suffices to check that 
the functor $\psimgl$ (the version of $\psigm$ for $\gm=\mglmod$) maps $\shmgl^{t^{mgl}=0}$ (here $\shmgl=\ho(\mglmod)$ and $t^{\mgl}$ is the corresponding homotopy $t$-structure) inside the essential image of $\psimot$.  
The latter statement easily follows from Corollary 4.1.7 of \cite{degorient}; note that any right module over $\mgl$ in $\sht$ is also a left module (since $\mgl$ is a commutative ring object in $\sht$). 
\end{proof}

\begin{rema}\label{rmgl} 

1. The existence of the connecting functors $\sh\to \sht\to \dmk$ and $\sht\to \shmgl$ certainly implies that there are "more" of (cohomological) functors that factor through $\shc$ than of those that factor through $\shtc$ or other motivic categories. Still, author does not know of many examples of cohomological functors on $\shc$ that do not factor through $\shtc$. 
Besides, if $f:\cu\to \du$ is one of the comparison functors mentioned, then the knowledge that $H^*:\cu^{op}\to \au$ factors through $\du^{op}$ yields that the corresponding generalized coniveau spectral sequences and filtrations are $\du^{op}$-functorial, which is certainly stronger than $\cu^{op}$-functoriality. So, it always makes sense to factor $H^*$ through a "more structured" motivic category (if possible).\footnote{Also, the "cohomological dimensions" of pro-schemes with respect to $\du$ may be smaller than the corresponding one for $\cu$.}

Thus, for cohomology theories that are $\sh$-representable one can apply the results of \S\ref{sapcoh}, whereas to $\sht$-representable theories  it makes (more) sense to apply Theorem \ref{tshtt}. 
 Recall that 
 $\sht$-representable theories include Hermitian K-theory and Balmer's Witt groups (when $p\neq 2$; see Theorems 5.5 and 5.8 of \cite{horn}, respectively); these cohomology theories along with the functors $\sht$-represented by $S^{00}\{j\}$ (i.e., motivic cohomotopy) 
appear to be the most interesting non-orientable ones. 

2.  There probably exists a connecting functor $ \shmgl\to \dmk$ whose composition with $\phimgl$ equals $\phimot$; see  Remark 1.3.4 of \cite{bondegl} for more detail. This would easily imply that motivic dimensions of pro-schemes are actually equal to their cobordism-dimension (cf. Proposition \ref{pmgl}(2)).

A related question is whether all oriented cohomology theories listed in Theorem \ref{tmotdimt}(III) are represented by {\it strict $\mgl$-modules}, i.e., by 
 objects of $\sht$ that lift to $\obj \shmgl$. 

Note however the spectra of the form $N\wedge \mgl$ (for $N\in \obj \sht$; so, these are the objects in the image of $\forg$) give a rich source of functors that can be factored through $\shmgl$. 
It follows in particular that all the direct summand results for schemes of motivic dimension at most $d$ (see Proposition \ref{pmott}(\ref{imotr})) carry over to their algebraic cobordism groups.

3. We have replaced $\mgl$ by its cofibrant replacement (in the category of ring objects) for the purpose of obtaining a proper model category (using Theorem 5.3.1 of \cite{pavl}). The price for this is that we cannot be sure that the ring spectrum $\mglp$ is commutative; this prevented us from constructing a tensor product on $\mglmod$. Note however that the category $\shmgl$ is 
 is isomorphic to $\ho(\modd-\mgl)$; thus both of these categories are symmetric monoidal.

\end{rema}

\subsection{
On "localizations of coefficients" for $\sht$ and the "$\tau$-positive acyclity" of motivic points}\label{sshinvp} 

Now we  relate motivic dimensions of schemes to one more wide class of their cohomologies. 
The results of this subsection can probably be deduced from Theorem \ref{tfunct}; yet we prefer to apply the "triangulated" methods of \cite[\S5.6]{bpure}.

We start from discussing "localizing of coefficients" for triangulated categories. 

\begin{defi}\label{dlocoeff}
Let $S\subset \p$ be a set of prime numbers; denote the ring  $\z[\sss\ob]$ by $\lam$. 
Let $\cu$ be a triangulated category.

\begin{enumerate}
\item\label{idlin} We will say that $M\in \obj \cu$ is $\lam$-linear  if $s\id_M$ is an automorphism for any $s\in \sss$.

\item\label{idlcat1} We will use the notation $\cu\otimes \lam$ for the category  whose object class is $\obj \cu$ and morphism groups are defined by the formula $\cu\otimes \lam(M,N)=\cu(M,N)\otimes_\z \lam$ (for all $M,N\in \obj \cu$).

\item\label{idlcat2} We will write $\cu[\sss\ob]$ for the subcategory of $\lam$-linear objects of $\cu$.\footnote{This notation is not "very logical" for a "general" triangulated category $\cu$; yet we will soon recall that it is really useful if $\cu$ is compactly generated (or cocompactly cogenerated).}


\end{enumerate}
\end{defi}


Now we relate localizing of coefficients to our context.

\begin{pr}\label{plocoeffsht}
Let $\sss\subset \p$ be as above.

I. Then the following statements are valid.

\begin{enumerate}
\item\label{ipliso}
$\sht[\sss\ob]$ is isomorphic to the Verdier localization of $\sht$ by its localizing subcategory generated by 
 $\co(s\id_M)$ for $s\in \sss$ and $M\in \obj \shtc$, and $\gdt[\sss\ob]$ is isomorphic to the localization of $\gdt$ by its colocalizing subcategory cogenerated by $\co(s\id_M)$ for $s\in \sss$ and $M\in \obj \shtc$.

We will use the notation $l_\sss^{\sht}$ and $l_\sss^{\gdt}$ for the corresponding localization functors; for any $P\in \obj \opa$ we we write $
\omdt[\sss\ob](P)$ for  $l_\sss^{\gdt}(\omdt(P))$.

\item\label{iplsht1} 
The restrictions of $l_\sss^{\sht}$ and $l_\sss^{\gdt}$ 
to $\shtc$ induce full  embeddings of $\shtc\otimes \lam$ into $\sht[\sss\ob]$ and $\gdt[\sss\ob]$, respectively. Moreover, for any $M\in \obj \shtc$ and $N\in 
\obj \sht$ we have $\sht[\sss\ob](l_{\sss}(M),l_{\sss}(N))\cong \sht(M,N)\otimes_{\z}\lam $.

\item\label{iplshtw}
The category $\sht[\sss\ob]$ (resp. $\gdt[\sss\ob]$) is  (co)compactly (co)generated by $l^{\sht}_\sss(\obj \shtc)$ (resp. by $l^{\sht}_\sss(\obj \shtc)$), and  $\kar_{\sht[\sss\ob] }(\shtc\otimes \lam)$ (resp. $\kar_{\gdt[\sss\ob] }(\shtc\otimes \lam)$) is its subcategory of (co)compact objects that is equivalent to the category $\kar(\shtc\otimes \lam)$ that will be denoted by $\sht[\sss\ob]^c$.

\item\label{iplsht2} The couple $(\sht^{t^T\le 0}\cap \obj \sht[\sss\ob], \sht^{t^T\ge 0}\cap \obj \sht[\sss\ob])$ is a $t$-structure on  $\sht[\sss\ob]$ that  will denoted by  $t^T[\sss\ob]$; hence $\hrt^T[\sss\ob]\subset \hrt^T.$ 

\item\label{iplsht3} The couple $(l_\sss^{\gdt}(\gdt_{w^T\le 0}), \kar_{\gdt[\sss\ob]} l_\sss^{\gdt}(\gdt_{w^T\le 0}))$ is a cosmashing weight structure on $\gdt[\sss\ob]$ (and we will write  $w^T[\sss\ob]$ for it); thus 
$l_\sss^{\gdt}$ is weight-exact.

\item\label{iplsht4} The restriction $\Phi^T[\sss\ob]$ of $\Phi^T$ to $\gdt[\sss\ob]\opp\times \sht[\sss\ob]$ is a nice duality and $w^T[\sss\ob]\perp_{\Phi^T[\sss\ob]} t^T[\sss\ob]$.

II. The natural analogues of all of the assertions of Proposition \ref{pshtt} and Theorem \ref{tshtt}(II, III, IV.\ref{idst1}--\ref{isplc},V) for $\sht[\sss\ob]$ and $\gdt[\sss\ob]$ are fulfilled; one should replace $\shtc$ by $\kar(\shtc\otimes\lam)\cong \kar_{\sht[\sss\ob]}(\shtc\otimes\lam)\cong \kar_{\gdt[\sss\ob]}(\shtc\otimes\lam)$, $t^T$ by $t^T[\sss\ob]$,    $w^T$ by $w^T[\sss\ob]$,  $\omdt(-)$ by $\omdt[\sss\ob]$,  strictly homotopy invariant sheaves of abelian groups by those of $\lam$-modules in these statements,  consider the corresponding extended cohomology functors and $\lam$-linear $T$-cohomological dimensions (see below).

III. An object $E$ of $\sht$C is $\lam$-linear whenever the groups $E^i_j(X)$ are $\lam$-linear for all $i,j\in \z$ and $X\in \sv$.

\end{enumerate}
\end{pr}
\begin{proof}
I. Immediate from Proposition 5.6.2 of \cite{bpure}  (which relies significantly 
on earlier well-known results including  Appendix B of \cite{levconv} and Appendix A of \cite{kellyth}) along with its categorical dual.

II. Given assertion I, these statements easily follow from their $\sht$-versions. 

III. The "only if" implication is obvious. 

Now assume that all $E^i_j(X)$ are $\lam$-linear. Then for any $s\in \sss$ and $E_s=\co s\id_E$ it is easily seen that all the groups $E_s{}^i_j(X)$ vanish. 
 Since the objects $\omt(X_+)\{j\}$ generate $\sht$ as its own localizing subcategory, it follows that $E_s=0$ (see Lemma \ref{locat}(1)); hence  $E$ is $\lam$-linear indeed.

\end{proof}

\begin{rema}\label{rloc}
So, one may say that  a pro-scheme $S$ is of $\lam$-linear $T$-cohomological dimension at most $d$ whenever $l_\sss(\omdt(S_+))\in \gdt[\sss\ob]_{[0,d]}$ (here we consider the weight structure $w^T[\sss\ob]$).  

Now, the weight-exactness of $l_\sss$ (see part I.\ref{iplsht3} of our proposition) implies that the $T$-cohomological dimension of a pro-scheme is not smaller than its $\lam$-linear $T$-cohomological dimension. Note also that the $\lam$-linear version of Theorem \ref{tshtt} gives several equivalent conditions for $S$ to be of $\lam$-linear $T$-cohomological dimension. 

The author 
does not have any examples of a strict inequality here; yet it seems that the answer may depend on ($k$ and) whether $S$ contains the prime $2$. 

\end{rema}

Now we fix $\sss=\{2\}$ (we will write "$[1/2]$" instead of "$[\{2\}\ob]$" in the corresponding notation) and proceed to define the  $\tau$-positive parts of our categories. So, we recall (see the text preceding Lemma 6.7 of \cite{levconv}) that 
$\sht[1/2]$ decomposes into the direct sum of two triangulated subcategories $\shtpl$ and $\shtmi$ (that are closed with respect to small coproducts), and this decomposition restricts to a decomposition $\sht[1/2]^c\cong \shtcpl \bigoplus \shtcmi$, where the latter are the corresponding subcategories of compact objects.

Now we formulate the $\shtpl$-analogue of the previous proposition. 

\begin{pr}\label{prpl}
I.\begin{enumerate}\item\label{iplu}
The projection $\prpl:\sht[1/2]\to  \shtpl$ 
 is uniquely determined by the following conditions: it is exact, respects coproducts, and restricts to the projection $\shtc\to \shtcpl$.  

\item\label{iplt} The couple $(\sht[1/2]^{t^T[1/2]\le 0} \cap \obj \shtpl, \sht[1/2]^{t^T[1/2]\ge 0} \cap \obj \shtpl)$ is a $t$-structure on $\shtpl$ that  will be denoted by $t^+$; hence $\hrt^+\subset \hrt^T[1/2]\subset \hrtt$. 

\item\label{iplgd} There also exists an exact 
  projection $\prplgd$ of $\gdt[1/2]$ onto its subcategory $\gdtpl$ cogenerated by $\obj \shtcpl$ (cf. Proposition \ref{plocoeffsht}(I.\ref{iplshtw})); this functor respects products.

\item\label{iplw} $(\prplgd(\gdt[1/2]_{w^T[1/2]\le 0}), \prplgd(\gdt[1/2]_{w^T[1/2]\le 0}))$  is a cosmashing weight structure on $\gdtpl$ (we will write $w^+$ for it); thus  $\prplgd$ is weight-exact. 

\item\label{ipld} The restriction $\Phi^+$ of $\Phi^T$ (and so, also of $\Phi^T[1/2]$) to $\gdtpl\opp\times \shtpl$ is a nice duality and $w^+\perp_{\Phi^+} t^+$.
\end{enumerate}

II. The natural analogue of Proposition 
\ref{plocoeffsht}(II) for $\shtpl$ and $\gdtpl$ is fulfilled. 
\end{pr}
\begin{proof}
I. Obviously, the projection functor $\prpl:\sht[1/2]\to  \shtcpl$ as  described in (the aforementioned)  \S6 of \cite{levconv} is exact and respects coproducts. 
It follows that  one can apply  Proposition 5.6.4 of \cite{bpure} in this setting, and combining it with Proposition \ref{plocoeffsht}(I)) 
 one easily obtains all the statements in question.

II. Similarly to  Proposition \ref{plocoeffsht}(II), these statements easily follow from their $\sht[1/2]$-versions combined with assertion I of our proposition. 
\end{proof}

\begin{rema}\label{rmi}
1. Certainly, the obvious analogue of  our proposition for $\shtmi$ is also fulfilled. However,  we will not consider it here.  

2. Note that $\shtpl=\shtoh$ if $-1$ is a sum of squares in $k$, i.e., if $k$ is unorderable. Indeed, in the case $\cha k>0$ this fact is given by Lemma 6.8 of \cite{levconv}. For $\cha k=0$ this statement can be easily extracted from the proof of Lemma 6.7 of ibid. 

3. For any set $\sss\subset \p$ containing $2$ the obvious  
$\lam$-linear  analogues of  all the results this section concerning $\shtpl$ are certainly valid (since passing to 
 $\lam$-linear context requires absolutely no changes in the arguments).

4. We will say that  a pro-scheme $S$ is of $+$-dimension at most $d$ if for the functor $\omdtpl=\prplgd\circ l_{\{2\}}^{\gdt}\circ \omdt$ we have $\omdtpl(S_+)\in \gdtpl_{w^+\le d}$. Then part II of Proposition \ref{prpl} says in particular that this condition is equivalent to the vanishing of $H^n_{Nis}(S,N)$ for $n>d$ and $N$ belonging to the image of $\pi^0_0\circ \prpl(\shtoh^{t[1/2]=0})$ in $\shi$. 

\end{rema}

We easily obtain the following remarkable results.

\begin{theo}\label{tprimacycl}
I. Let $S$ be a pro-scheme of motivic dimension at most $d$ (for some $d\ge 0$).

1. Then $S$ is also of $+$-dimension at most $d$.

2. Assume that $k$ is unorderable. Then $S$ is also of $\zoh$-linear $T$-cohomological dimension at most $d$.

Moreover, if   $N\in \obj \shi$ is of the form $\pi^0_0(E)$ for $E\in \sht^{t^T=0}$
and  $n> d$ then the groups  $H^n_{Nis} (S,N)$ are $2$-torsion for all $n>d$ (i.e., all  elements of these groups are annihilated by powers of $2$). 

II. Assume that $E$ is a very weakly orientable spectrum. Then $l_{\{2\}}(E)\in \obj \shtpl\subset \shtoh$.  
\end{theo}
\begin{proof}
I.1. According to Proposition \ref{ppacycl}(2), it suffices to verify that any element of $\shtpl^{t^+=0}$ is orientable (as an object of $\hrtt$). 
Now we recall that  the objects of $ \shtpl$ inside $\shtoh$ are characterized by the condition $\tau =\id$, where $\tau=\id+\eta\circ [-1]$ is the element of the Milnor-Witt ring of $k$ (cf. \cite{movo} and Remark \ref{rdwo}(1)).  Rewriting the relation 4 in \cite[Definition 6.3.1]{movo} as $\eta \tau+\eta=0$ we obtain that $\eta_E=0$ for any $E\in \obj \shtpl$.

2. The first part of the assertion follows immediately from  assertion I.1  since $\shtpl=\shtoh$ in this case (see Remark \ref{rmi}(2)).

Now, consider $N=\pi^0_0(E)$ as in the "moreover" part of the assertion. Consider the  spectrum $E'=l_{\{2\}}E\in \sht[1/2]^{t^T[1/2]=0}$; we have $H^n_{Nis}(S,\pi^0_0(E'))\cong E'^n_0(S)\cong E^n_0(S)\otimes_{\z}\z[1/2]\cong H^n_{Nis} (S,N)\otimes_{\z}\z[1/2]$ (the second isomorphism follows easily from Proposition \ref{plocoeffsht}(\ref{iplshtw})). Since for any $n>d$ we have $E'^n_0(S)=\ns$ (according to the first part of the assertion), we obtain that the group $H^n_{Nis} (S,N)$ is $2$-torsion indeed.

II. Since $\hrtt[1/2]$ splits as the direct sum of the categories $\sht^+$ and $\sht^-$,  
 we can "project" the ($\zoh$-linear version of the) filtration of $E$ provided by Definition \ref{dwo}(2) onto $\prpl\circ l_{\{2\}}(E)$ and $\prmi\circ l_{\{2\}}(E)$. Hence it suffices to verify that any orientable object of  $\hrt^-$ is zero. Now, the objects of $ \shtmi$ inside $\shtoh$ are characterized by the condition $\tau =-\id$.  Since $\tau=\id+\eta\circ [-1]$, the equality $\eta_E=0$ implies that $2\id_E=0$, and it remains to recall that $E$ is $\zoh$-linear.

\end{proof}

\begin{rema}\label{rpacycl}

1. Hence motivic dimensions of pro-schemes contains much information on their $\zoh$-linear very weakly orientable cohomology (for a "general" $k$; here we apply Proposition \ref{plocoeffsht} to relate $\zoh$-linearity of objects of $\sht$ to the $\zoh$-linearity of the functors they represent); 
 this is the same thing as $\zoh$-linear cohomology that factors through $\sht$ if $k$ is unorderable.

Note also that for a cohomological functor $H$ from $\shtc$ into an AB5 category $ \au$ one can construct its $\zoh$-linear  version as follows: one should put $H[1/2](M)=\inli(H(M)\stackrel{\times 2}{\to }H(M)\stackrel{\times 2}{\to }H(M)\stackrel{\times 2}{\to }\dots)$.\footnote{Note however that this direct limit may be zero.} 

2. One can easily define $\zoh$-linear motivic dimensions of schemes and prove that they are equal to their $+$-dimensions.

3. Hence we obtain that primitive schemes are $+$-points.

 The author wonders whether one can prove (some version of) the $\sht$-acyclity of primitive schemes "directly", i.e., using a version of Voevodsky's split standard triple argument (as in the proof of Theorem 4.19 of \cite{walker}). Possibly this can be done using Voevodsky's framed correspondences (see \cite{garpan}). 

\end{rema}

\subsection{
On Artin-Tate spectra and relative motivic categories 
}\label{sat}

The formalism of weight structures yields much flexibility for the calculation of  generalized coniveau spectral sequences $T(H,M)$ (that certainly yield spectral sequences that are canonically isomorphic starting from $E_2$). Usually none of these spectral sequences are "simple"; yet for certain $M$ in $ \obj \shc\subset \obj \gd$ (or in $ \obj \shtc\subset \obj \gdt$) there exist very "economical" choices of weight Postnikov towers.

Consider the  triangulated subcategory of $\shc$ (resp. of $\shtc$)  densely generated by 
 $\om (\spe (K)_+)\{j\}$ for   $K$ running through {\bf finite} extensions 
 of $k$ and $j\ge 0$ (resp.  
 $\omt (\spe (K)_+)\{j\}$ for $j\in \z$); we will call objects of these categories  {\it Artin-Tate spectra}. 
Remark \ref{rwgen} yields the existence  of weight structure for these categories; 
 their hearts 
 equal $\kar_{\shc}(\bigoplus_{i=1}^n \om (\spe (K_i)_+)\{j_i\})$ and $\kar_{\shtc}(\bigoplus_{i=1}^n \omt (\spe (K_i)_+)\{j_i\})$ (for $K_i$ running through finite extensions of $k$ and $j_i\ge 0$ for the first category and $j_i\in \z$ for the second one). 
Moreover,   weight Postnikov tower with respect to these weight structures are certainly also  Gersten weight Postnikov tower with respect to $w$ and $w^T$, respectively (by Proposition \ref{pbw}(\ref{iwpostc})).

Now, 
 plenty of varieties  yield Artin-Tate spectra (both in $\shc$ and $\shtc$). 
 One can demonstrate this by noting that the   category of Artin-Tate spectra contains $\om (\spe (K)_+\lan j\ra)$, $\om (\gmmpl)$, and $\om (\p^n_+)$ for any $n\ge 0$, and it is a monoidal triangulated subcategory of $\shc$ (whereas the smash product of motivic spectra corresponds to the product of varieties). 
One can also apply Corollary \ref{cretr} (giving the triviality of bundles over $\afo$-points) and Theorem \ref{tshtt}(IV.\ref{idst1t}), respectively, here. 

Lastly, one can detect Artin-Tate spectra using the weight complex functor; 
see Corollary 8.1.2 of \cite{bws}.

\begin{rema}\label{relgws} Now we recall that Gersten weight structures can also be constructed for a wide range of  "relative" motivic categories, and describe the corresponding "Artin-Tate substructures".

1. Firstly we note that one can apply Proposition 5.5.2 and Corollary 5.5.3 of \cite{bpure} to any proper stable simplicial model category $\gm$ whose homotopy category will be denoted by $\cu'$ (to obtain the corresponding version 
 $\cu$  of $\gdb$)  along with a set $C$ of compact objects of $\cu'$. In particular, elements of $C$  have natural (and cocompact) images in  
 $\cu$. We will use the notation $\eu$ for the subcategory cogenerated by these images  (cf. \S\ref{scwger}); we obtain the existence of a duality $\Phi^{\eu}:\eu\opp\times \cu'\to \ab$. Next, if $C\subset C[1]$ then the weight structure $w_{\eu}$ cogenerated by (the image of) $C$ in $\eu$ is $\Phi^{\eu}$-orthogonal to the $t$-structure $t_{\cu'}$ generated by $C$ in $\cu'$.

Now, all existing (stable) "relative motivic" categories (i.e.,  triangulated categories of certain relative motives or motivic spectra over a base scheme $S$ that we assume to be noetherian separated excellent of finite Krull dimension; we will call objects of these categories $S$-spectra) appear to possess proper stable simplicial models. Thus can apply the general formalism of \cite{bpure} to this setting. However, to obtain "geometric" consequences as a result one needs a "geometric" description of $\hw$ or of (certain) weight Postnikov towers. 

So, $\hw$ should contain a "geometrically significant" subcategory $H$ of $\eu$ such that $H\perp_{\Phi^{\eu}} 
\cu'^{t_{\cu'}\le -1}\cup \cu'^{t_{\cu'}\ge 1}$. A natural candidate (and the only one known to the author) for this orthogonality relation is Theorem 3.3.1 of \cite{bondegl}. So, one can take for $\cu'$ any of the motivic categories listed in Example 1.3.1 of ibid.\footnote{Recall that this list includes certain instances of  $SH(S)$ and $DM(S)$; moreover, the restrictions on the base scheme in loc. cit. can probably be weakened in the case of $DM(S)$ by using Theorem  3.4.2 of \cite{bkl}.} 
and consider the corresponding $t_{\cu'}$. 

Now, we try to describe the consequences of this choice (so, we will try to recover $w_{\eu}$ from  $t_{\cu'}$).  
Loc. cit. implies that $\hw_{\eu}$ should be "cogenerated" by 
 certain twists\footnote{Being more precise, for $\spe (F)$ being of finite type over $S$, $F$ is a field, to obtain an element of  $\eu_{w_{eu}=0}$ one should shift the Borel-Moore $S$-pro-spectrum of $\spe (F)$ by $\de(F)$ and twist it by $\{n\}$ for any $n\in \z$, where $\de$ is the corresponding dimension function (cf. \S1.1 of ibid.).} of the corresponding {\it Borel-Moore $S$-pro-spectra} of those spectra of fields that are essentially of finite type over $S$ (and this fact is also coherent with the descriptions of $t_{\cu'}$ given by Proposition 2.3.1 of ibid.).

The problem with constructing this $w_{\eu}$ is that to apply an argument similar to that in the proof of  Theorem \ref{tgw}(\ref{iwg5}) 
(that relies on Proposition \ref{pinfgy}(1)) one needs certain "canonical models" of Borel-Moore $S$-spectra of finite type $S$-schemes (which are regular but not necessarily smooth over $S$).  For this purpose one can probably relate (Borel-Moore  $S$-spectra of quasi-projective) $S$-schemes to smooth $S$-schemes similarly to  Remark 1.2.2(1,4) of \cite{bkl}.

However, the author did not check this argument in detail. To avoid it one can apply some of the  methods of \cite{bger}  (that rely on countable homotopy limits in triangulated categories defined a-la \cite{bokne}). The only price that should be paid for choosing this reasoning is that $S$ should be assumed to be countable (i.e., to possess a Zariski covering by spectra of countable rings).

2. The reason why the author haven't studied relative Gersten weight structures yet is that the Borel-Moore $S$-pro-spectrum of 
$X/S$ 
 appears not to belong to $\eu_{w_{eu}=i}$ for any $i\in \z$ unless $X$ lies over a finite collection of (Zariski) points of $S$. So, we get no new ("relative") analogues of  our Theorem \ref{tds} (along with Proposition \ref{cdscoh})  using weight structures of this type. 

3. However, relative Gersten weight structures may be useful for studying the corresponding versions of coniveau spectral sequences (cf. Definition 3.1.5 of \cite{bondegl}). So we also discuss the Artin-Tate versions of these weight structures.
If we fix $\cu'$ then the Gersten-Artin-Tate   weight structure is  defined on the corresponding version $DAT(S)$ of  Artin-Tate-$S$-spectra; the latter  is the subcategory of the category of compact objects of $\cu'$ (recall that the latter coincides with  the category of cocompact objects of $\eu$) densely generated by twists of Borel-Moore motives of spectra of fields that are finite over locally closed points of $S$. 

Note here that any closed point $s$ of  $S$ is certainly locally closed. However, the case where 
all locally closed points are closed in $S$  is not really interesting since one can easily check (using the Grothendieck's six functor formalism for $\cu'$; cf. Theorem 2.4.50 of \cite{cd}) that $DAT(S)$ is the direct sum of the corresponding $DAT(s)$ (for $s$ running through closed points of $S$). Still for certain (base) schemes some of their non-closed points are still locally closed (i.e., they are  their locally closed subschemes). In particular, all semi-localizations of  dimension $1$ schemes at  finite collections of points possess this "strange" property. 
 It appears that for schemes of this type the category $DAT(S)$ (due to the existence of a locally closed $s'\in S$ whose closure contains $s\neq s'$) does not "split".

\end{rema}


\end{document}